\documentclass[11pt]{amsart}
\usepackage[top=1in, bottom=1in, left=1in, right=1in]{geometry}  % Page margins
\usepackage{amsmath,amssymb,amsfonts,amscd,amsxtra,amsthm}  % AMS mathematical packages
\usepackage{mathtools,mathrsfs}  % Enhanced math tools and script fonts
\usepackage{graphicx,subfig}  % Graphics and sub-figures
\usepackage{latexsym}  % LaTeX symbols
\usepackage{xcolor,color}  % Color support
\usepackage{epstopdf}  % EPS to PDF conversion
\usepackage{tikz, tikz-cd}  % TikZ graphics and commutative diagrams
\usepackage{bm, bbm}  % Bold math and blackboard bold fonts
\usepackage{enumitem}  % Customized lists
\usepackage{cite}  % Bibliography management
\usepackage{caption}  % Figure and table captions
\usepackage{multirow}  % Table multirow cells
\usepackage{algorithmic}  % Algorithmic environment
\usepackage{algorithm}  % Algorithm floating environment
\usepackage{rotating}  % Rotated content (sideways tables/figures)

\usepackage[toc,page]{appendix}  % Appendix support
\usepackage{placeins}  % Keep floats near their sections
\usepackage[colorlinks=true, linkcolor=blue, citecolor=magenta]{hyperref}  % Hyperlinks

\numberwithin{equation}{section}  % Number equations within sections

\usepackage{cleveref}  % Cross-referencing with lowercase labels

\newtheorem{theorem}{Theorem}[section]
\newtheorem{lemma}{Lemma}[section]
\newtheorem{corollary}{Corollary}[section]
\newtheorem{proposition}{Proposition}[section]
\newtheorem{definition}{Definition}[section]
\newtheorem{assumption}{Assumption}[section]
\newtheorem{remark}{Remark}[section]

\crefname{theorem}{Theorem}{Theorems}
\Crefname{theorem}{Theorem}{Theorems}

\crefname{lemma}{Lemma}{Lemmas}
\Crefname{lemma}{Lemma}{Lemmas}

\crefname{corollary}{Corollary}{Corollaries}
\Crefname{corollary}{Corollary}{Corollaries}

\crefname{proposition}{Proposition}{Propositions}
\Crefname{proposition}{Proposition}{Propositions}

\crefname{definition}{Definition}{Definitions}
\Crefname{definition}{Definition}{Definitions}

\crefname{assumption}{Assumption}{Assumptions}
\Crefname{assumption}{Assumption}{Assumptions}

\crefname{remark}{Remark}{Remarks}
\Crefname{remark}{Remark}{Remarks}

\crefname{example}{Example}{Examples}
\Crefname{example}{Example}{Examples}

\crefname{section}{section}{sections}
\Crefname{section}{Section}{Sections}
\graphicspath{{pics/}} % Folder for images (ensure 'pics' directory exists)

\title[Nonlinear subwavelength resonances and BICs]{Nonlinear subwavelength resonances and bound states in the continuum in metascreens}
\thanks{This work is partially supported by the Swiss National Science Foundation grant number 200021-236472.}
\author{Habib Ammari}
\address{(H. Ammari) Department of Mathematics, ETH Z\"urich, R\"amistrasse 101, 8092 Z\"urich, Switzerland}
\email{habib.ammari@math.ethz.ch}

\author{Yu Gao}
\address{(Y. Gao) Department of Mathematics, ETH Z\"urich, R\"amistrasse 101, 8092 Z\"urich, Switzerland}
\email{yugaoy@ethz.ch, yugaomath@gmail.com}

\begin{document}

\begin{abstract}

This paper establishes a mathematical framework for nonlinear subwavelength resonances and bound states in the continuum (BIC) in an acoustic metascreen with a cubic Kerr nonlinearity. We first use the quasiperiodic
Dirichlet-to-Neumann operator to reduce the open resonance problem to an
interior nonlinear variational problem. We then decompose the function space in which the variational problem is posed as the direct sum of two spaces and project the variational problem onto these two subspaces. Solving the projected equations successively yields a finite-dimensional nonlinear resonance equation with controlled remainders. We next apply the implicit function theorem near simple capacitance modes. This proves the existence and asymptotic expansions of linear subwavelength resonance branches and their small-amplitude nonlinear continuations. Finally, reflection symmetry gives a classification of the subwavelength branches.
We characterize the symmetric resonance branches and prove that antisymmetric branches are exact BICs in both the linear problem and the nonlinear problem.

%There are
%$n_{\pi}^{\mathrm f}+n_{\pi}^{\mathrm p}$ symmetric resonance branches, while
%the $n_{\pi}^{\mathrm p}$ antisymmetric branches are exact BICs in both the
%linear problem and the nonlinear problem.
\end{abstract}

\maketitle


\bigskip

\tableofcontents

%%%%%%%%%%%%%%%%%%%%%%%%%%%%%%%%%%%%%%%%%%%%%%%%%%%%%%%%%%%%%%%%%%

\section{Introduction}

\subsection{Background}

Controlling waves at subwavelength scales is a central objective in acoustic, elastic, and electromagnetic wave physics. Metamaterials provide a natural platform for this purpose: their effective response is shaped not only by the material parameters of their constituents, but also by geometry, contrast, and resonance at scales that are much smaller than the typical size of their building blocks.
Subwavelength resonant metamaterials have been realized in both acoustics and
optics, with applications including wave manipulation, field enhancement, and
sensing (see, e.g.,
\cite{liu2000locally,cummer2016controlling,ma2016acoustic,
evlyukhin2012demonstration,kuznetsov2016optically,tzarouchis2018light}).
In this work, we focus on high-contrast acoustic resonators.

For finite systems of high-contrast acoustic resonators, air bubbles in water
provide a canonical example. The classical starting point is Minnaert's study
of the sound emitted by air bubbles in water
\cite{minnaert1933musical}. Such bubbles resonate at wavelengths much larger
than their physical size. A rigorous mathematical analysis of this
subwavelength resonance was developed in \cite{ammari2018minnaert}; see also \cite{cbms}. In the single-resonator case, the leading resonant frequency is governed by a capacitance-to-volume ratio. For multiple resonators, the continuous resonance problem admits a finite-dimensional leading-order reduction governed by a
capacitance matrix; its eigenvalues and eigenvectors determine the leading
subwavelength resonant frequencies and modes \cite{feppon2022modal,ammari2024functional,ammari2019fully}. This
capacitance approach has also led to effective medium theories for bubbly
fluids near the Minnaert frequency \cite{ammari2017effective} and to
subwavelength resonance models for auditory signal filtering
\cite{ammari2019fully}.

In periodic high-contrast systems, the same principle leads to quasiperiodic
capacitance matrices depending on the Bloch parameter. In the linear
subwavelength regime, these matrices capture the leading behaviour of
subwavelength band functions, Bloch modes, or resonance branches, depending on
the geometry of the periodic structure
\cite{ammari2024functional,ammari2021bound}. Related capacitance-based methods
have also been used to study exceptional points in parity--time-symmetric
subwavelength metamaterials \cite{ammari2022exceptional}. However, open periodic arrays also possess radiation channels. At real frequencies in the radiation
continuum, propagating Rayleigh--Bloch modes can carry energy away from the
array. In the linear setting and away from Rayleigh thresholds, scattering
resonances can be characterized as poles of the outgoing resolvent or scattering
matrix, equivalently as nontrivial outgoing solutions of the homogeneous
problem. Such leaky resonances lie in the lower half of the complex-frequency plane, unless a decoupling mechanism is present.

A particularly important decoupling mechanism is symmetry. Special periodic
structures can support real-frequency modes inside the radiation continuum whose
propagating Rayleigh--Bloch coefficients vanish. These embedded eigenfunctions
are localized in the open direction although radiation channels are available;
in the physics literature, they are known as bound states in the continuum, or
BICs \cite{friedrich1985interfering,hsu2016bound}. When the protecting symmetry
is broken, the embedded eigenvalue typically turns into a nearby leaky
resonance, and its interference with the non-resonant background can produce an
asymmetric Fano line shape \cite{fano1961effects}. This BIC-to-Fano mechanism
has been studied for periodic gratings and slabs
\cite{bonnet1994guided,shipman2005resonant,shipman2007guided,
shipman2010resonant,shipman2012total},
for narrow-slit gratings \cite{lin2020mathematical,lin2021fano}, for
periodically repeated high-contrast dimers \cite{ammari2021bound}, and for
all-dielectric Maxwell metasurfaces \cite{ammari2024fano}.

The preceding discussion concerns the linear regime. In Kerr-type media, the
material response depends on the local field intensity
\cite{boyd2008nonlinear}. In acoustics, such a response can be modeled through
an intensity-dependent compressibility or bulk modulus, while in optics it is
often modeled through an intensity-dependent refractive index. Since
subwavelength resonators strongly confine fields, weak nonlinearities can lead
to amplitude-dependent resonant frequencies and mode profiles.
More broadly, nonlinear metamaterials enable intensity-dependent wave control,
enhanced wave--medium interactions, and localized states; see, e.g.,
\cite{kivshar2018all,koshelev2020subwavelength,rev-acoustics,bryn2019}.

The mathematical theory of nonlinear subwavelength resonances remains less
developed than its linear counterpart. For small high-contrast scatterers,
nonlinear resonances have been analyzed through asymptotic expansions
in the volume and the contrast \cite{meklachi2018asymptotic}. For finite
acoustic high-contrast resonator systems, nonlinear capacitance-type reductions
have been derived, revealing amplitude-dependent resonances and additional
nonlinear modes \cite{ammari2025nonlinear}. In \cite{clemens2026}, a perturbative cascade expands the resonant frequency and field in powers of the square root of the material contrast. The two-way correspondence with a finite discrete nonlinear capacitance system is rigorously proved. For periodic nonlinear resonator crystals, capacitance-operator and tight-binding approximations have been used to construct subwavelength localized states \cite{ammari2025resonator}. In the related dielectric setting, small-amplitude nonlinear resonances have been shown to bifurcate from linear high-index resonances, with symmetry-breaking branches in symmetric dimers \cite{ammari2025dielectric}.

\subsection{Main results}

We study nonlinear subwavelength resonances in an open periodic array of
multiple acoustic high-contrast resonators with a cubic Kerr nonlinearity. The
main purpose is to derive a finite-dimensional nonlinear resonance equation
and to identify the symmetry mechanism that turns such branches into exact bound
states in the continuum.

In the linear medium, boundary-integral formulations provide an effective tool
to describe resonance problems and obtain discrete approximations of subwavelength resonances \cite{ammari2024functional}. The nonlinear problem is different: the resonance equation is no longer described directly by a boundary integral equation. We therefore eliminate the exterior field through the Dirichlet-to-Neumann (DtN)
operator, following the strategy for nonlinear Helmholtz resonance problems
\cite{ammari2025nonlinear,ammari2025resonator}. This gives the interior
nonlinear resonance problem
\eqref{eq:interior-nonlinear-resonance-problem} and its variational form
\eqref{eq:def-nonlinear-variational-form}. The DtN operator and the
capacitance expansion both require the inverse of the
single-layer operator. Previous works used this inverse implicitly
\cite{ammari2021bound,ammari2022exceptional}, and explicit first-order
capacitance computations were carried out for the symmetric dimer geometries
considered there. We prove a new and explicit inverse formula for the truncated
operator $\widehat{\mathcal S}_D^{\alpha,k}$ where $D$ is the set of resonators in the unit cell, $\alpha$ is the quasiperiodicity, and $k$ is the wave number; see
\cref{lem:H0-isomorphism,prop:truncated-single-layer-inverse}. This yields
computable formulas for the leading capacitance matrix $C^0$ and the first
radiative correction $C^1$ for general multiple-resonator configurations.

We then analyze the interior variational problem through the decomposition
$H^1(D)=\mathcal X(D)\oplus\mathcal Z(D)$ in
\eqref{eq:XZ-decomposition}. The space $\mathcal X(D)$ consists of functions
that are constant on each resonator, and the space $\mathcal Z(D)$ consists of zero-average functions on the
resonator components. Then we project the variational form itself onto
$\mathcal X(D)$ and $\mathcal Z(D)$. This differs from the
reductions in \cite{ammari2025nonlinear,ammari2025resonator}, where an
additional variational problem is introduced before passing to a
finite-dimensional system. For fixed reduced amplitudes and a frequency
parameter, we first solve the $\mathcal Z(D)$-projected equation and then
substitute the correction into the $\mathcal X(D)$-projection; see
\cref{rem:lyapunov-schmidt-strategy}. This gives the projected system
\eqref{eq:projected-Z-nonlinear-form}--\eqref{eq:projected-X-finite-dimensional-system}.
The linear zero-average equation is solved by Lax--Milgram, while the
nonlinear one is solved by a contraction argument. In both cases, we obtain
estimates for the zero-average correction and for the remainder in the
finite-dimensional reduced equation; see
\cref{prop:linear-projected-reduction,prop:nonlinear-projected-reduction}. We
then apply the finite-dimensional implicit function theorem near a simple
capacitance eigenmode. This gives the linear subwavelength
resonance expansion and its locally unique small-amplitude nonlinear
continuation, including the amplitude-dependent frequency shift, the leading
radiative correction, and higher-order remainders; see
\cref{thm:linear-simple-resonance-branch,thm:nonlinear-modal-continuation}.

Finally, we identify when these resonance branches become exact BICs. The
preceding expansions show that the radiative correction generally creates an
imaginary part in the resonance frequency. At the $\Gamma$ point, that is, the center of the Brillouin zone, the
subwavelength regime has a single propagating Rayleigh--Bloch channel. We
impose reflection symmetry in the periodic direction. The reflection induces a
component permutation with fixed components and reflected pairs, and it
decomposes both the reduced amplitude space and the function space into
symmetric and antisymmetric subspaces. The capacitance problem respects this
decomposition. Its antisymmetric subspace has dimension $n_{\pi}^{\mathrm p}$,
the number of reflected pairs, and the first radiative correction satisfies
$C^1q=0$ for every antisymmetric reduced vector; see
\cref{prop:reduced-reflection-decomposition,lem:antisymmetric-C1-vanishing,rem:antisymmetric-candidates}.
Thus, antisymmetric capacitance modes are the reduced candidates for embedded
eigenvalues. We then restrict the exact linear and nonlinear resonance problems
to the real antisymmetric function space and prove that solutions of the
restricted problems solve the full outgoing problems. Combining this lifting
with resonance expansions gives the symmetry classifications in
\cref{thm:linear-bic-symmetry-classification,thm:nonlinear-bic-symmetry-classification}.
Under the corresponding simplicity assumptions, the $n_{\pi}^{\mathrm p}$
antisymmetric capacitance modes generate $n_{\pi}^{\mathrm p}$ exact linear BIC
branches and, for reflection-symmetric Kerr coefficients, $n_{\pi}^{\mathrm p}$
nonlinear BIC branches in the subwavelength regime. Along these branches the
frequencies are real and the propagating Rayleigh coefficient vanishes.

The reduction developed here can also be adapted to real-frequency scattering
problems. After the exterior field is eliminated by the DtN map, an incident
wave enters the interior formulation through a boundary forcing term, and the
same projection leads to a forced finite-dimensional amplitude equation. For
fixed geometry and incident channel, the linear problem gives scattering
coefficients that depend on frequency but not on incident intensity. In
contrast, in the Kerr case, the reduced equation is nonlinear in the modal
amplitudes, so the reflection and transmission coefficients may depend on the
incident intensity. This distinction provides a reduced framework for studying nonlinear frequency shifts, amplitude-dependent transmission, and multiple steady-state scattering responses near subwavelength resonances. It also gives a scattering
interpretation of the symmetry-protected BICs. When the symmetry protection is
weakly broken, either by a geometric perturbation or by detuning the
quasiperiodicity from the symmetry point, an exact BIC is expected to become a
high-$Q$ quasi-BIC resonance. In the linear case, this produces a narrow
Fano-type scattering profile, whereas Kerr nonlinearity may shift the resonance,
deform the line shape, and induce switching between coexisting scattering
states.

\subsection{Outline}

This paper is organized as follows. In \cref{section:problem-setting}, we
formulate the nonlinear periodic transmission problem and introduce the
quasiperiodic Green function, layer potentials, and capacitance matrices.
In \cref{section:general-subwavelength-resonances}, we use the DtN operator to reduce the
resonance problem to a projected system on
$\mathcal X(D)\oplus\mathcal Z(D)$, derive the linear subwavelength expansion,
and prove the small-amplitude nonlinear continuation. In
\cref{section:symmetric-configurations}, we impose reflection symmetry at the
$\Gamma$ point and prove the existence and geometric count of linear and
nonlinear BICs in the antisymmetric subspace. Finally,
\cref{section:numerical-experiments} validates the theoretical results through numerical experiments.

\subsection{Notation}

Throughout the paper, we work in $\mathbb R^d$, with $d\in\{2,3\}$, and write each point as $x=(x_\ell,x_d)$. Here, $x_\ell\in\mathbb R^{d-1}$ denotes the coordinates along the periodic directions, while $x_d\in\mathbb R$ is the transverse coordinate. Periodicity is taken with respect to the $(d-1)$-dimensional lattice
\[
\Lambda:=L\mathbb Z^{d-1}, \qquad \Lambda^*:=(2\pi/L)\mathbb Z^{d-1},
\]
where $\Lambda$ is the direct lattice and $\Lambda^*$ is the reciprocal lattice. The periodicity cell, its measure, and the first Brillouin zone are denoted by
\[
Y:=\bigl[-L/2,L/2\bigr]^{d-1}, \qquad
|Y|=L^{d-1}, \qquad
Y^*:=\bigl[-\pi/L,\pi/L\bigr]^{d-1}.
\]
The corresponding infinite periodic strip and its truncation at height $h>0$ are
\[
\Omega:=Y\times\mathbb R, \qquad
\Omega_h:=Y\times(-h,h).
\]
For a Bloch parameter $\alpha\in Y^*$ and a reciprocal lattice vector $\eta\in\Lambda^*$, we set
\begin{align*}
    \alpha_\eta:=\alpha+\eta,
    \qquad
    \beta_\eta(k):=\sqrt{k^2-|\alpha_\eta|^2}.
\end{align*}
The square root is chosen on the outgoing branch.
In particular, for real $k>0$ away from Rayleigh thresholds, $\beta_\eta(k)>0$ when $|\alpha_\eta|<k$, whereas $\operatorname{Im}\beta_\eta(k)>0$ when $|\alpha_\eta|>k$.

The structure in one period consists of $N$ resonators $D_1,\ldots,D_N\subset\Omega$. They are assumed to be pairwise disjoint, connected, and to have $C^2$ boundaries. We write
\begin{align}\label{eq:resonator-domain-volume-matrix}
    D:=\bigcup_{j=1}^N D_j, \qquad
    V:=\operatorname{diag}\bigl(|D_1|,\ldots,|D_N|\bigr),
\end{align}
where $V$ is the diagonal volume matrix. If $E\subset D$ is measurable or if $\Gamma\subset\partial D$ is surface-measurable, then $\chi_E$ and $\chi_\Gamma$ denote the corresponding characteristic functions. When confusion cannot arise, we also use $1$ for the constant function on the relevant set. Whenever the denominators are nonzero, we define the average-normalized characteristic functions by
\[
\chi_E^{\mathrm{av}}:=\frac{\chi_E}{|E|}, \qquad
\chi_\Gamma^{\mathrm{av}}:=\frac{\chi_\Gamma}{|\Gamma|}.
\]

We employ the following standard function spaces:
\[
L^2(D):=\bigl\{ u:D\to\mathbb C \bigm| \int_D |u|^2\,dx<\infty \bigr\},
\qquad
H^1(D):=\bigl\{ u\in L^2(D) \bigm| \nabla u\in L^2(D)^d \bigr\}.
\]
The boundary space $L^2(\partial D)$ is defined analogously with respect to the surface measure $d\sigma$.
All spaces are taken to be complex-valued unless otherwise stated.
The $L^2$ inner product on $D$ and the duality pairing between $H^{-1/2}(\partial D)$ and $H^{1/2}(\partial D)$ are defined to be linear in the first entry and conjugate-linear in the second:
\[
(u,v)_D:=\int_D u\,\overline v\,dx,\qquad
\langle \phi,\psi\rangle_{\partial D}:=\int_{\partial D}\phi\,\overline\psi\,d\sigma.
\]
Restrictions to each connected component $D_j$ are denoted analogously:
\[
(u,v)_{D_j}:=\int_{D_j}u\,\overline v\,dx,\qquad
\langle \phi,\psi\rangle_{\partial D_j}:=\int_{\partial D_j}\phi\,\overline\psi\,d\sigma.
\]
When both arguments are square integrable, the duality pairing coincides with the $L^2(\partial D)$ inner product. The symbol $\nu$ denotes the outward unit normal to $\partial D$. The subscripts $\pm$ indicate the traces taken from the exterior and interior of $D$, respectively. The jump $[\cdot]$ across $\partial D$ is defined as the exterior trace minus the interior trace.

We shall frequently use the finite-dimensional space of componentwise constant functions
\[
\mathcal X(D) := \operatorname{span}\{\chi_{D_1},\ldots,\chi_{D_N}\},
\]
together with its zero-average complement
\[ \mathcal Z(D) := \bigl\{ v\in H^1(D): (v,\chi_{D_j})_D=0,\ 1\le j\le N \bigr\}.
\]
Then we have the decomposition
\begin{equation}\label{eq:XZ-decomposition}
H^1(D)=\mathcal X(D)\oplus\mathcal Z(D).
\end{equation}
For $q=(q_1,\ldots,q_N)^\top\in\mathbb C^N$, we denote its piecewise constant lift by
\begin{equation}\label{eq:piecewise-constant-lift}
u_q:=\sum_{j=1}^N q_j\chi_{D_j}\in\mathcal X(D).
\end{equation}
Accordingly, every $u\in H^1(D)$ can be written uniquely as
\[
u=u_q+z, \qquad u_q\in \mathcal X(D),\qquad z\in\mathcal Z(D),
\]
where $q_j=(u,\chi^{\rm av}_{D_j})_D$, $1\le j\le N$. Since $u_q$ is constant on each component and $z$ has zero average on each component, the decomposition is orthogonal with respect to the standard
$H^1(D)$ inner product: the $L^2$ cross term vanishes by construction and the
gradient cross term vanishes because $\nabla u_q=0$ in each $D_j$. Hence,
\[
    \|u\|_{H^1(D)}^2
    =
    \|u_q\|_{L^2(D)}^2
    +
    \|z\|_{H^1(D)}^2
    =
    q^*Vq+\|z\|_{H^1(D)}^2 .
\]
Equivalently, with $\|q\|_V:=(q^*Vq)^{1/2}=\|u_q\|_{L^2(D)}$, we have $\|u\|_{H^1(D)}^2
    =
    \|q\|_V^2+\|z\|_{H^1(D)}^2$.
Since $V$ is positive definite, $\|\cdot\|_V$ is equivalent to the Euclidean norm
on $\mathbb C^N$. Consequently,
\[
    \|u\|_{H^1(D)}
    \asymp
    \|q\|+\|z\|_{H^1(D)},
\]
with constants depending only on the fixed geometry of $D$.

\section{Problem setting and preliminaries}
\label{section:problem-setting}

This section introduces the nonlinear periodic transmission problem and the
layer-potential framework used in the following. We first formulate the scattering and
resonance problems, then recall the quasiperiodic Green function and the
associated boundary integral operators, and finally introduce the capacitance
matrices governing the subwavelength reduction.

\subsection{Problem formulation}

The set of resonators $D$ and the surrounding medium are characterized by their mass
densities and bulk moduli, denoted by $(\rho_b,\kappa_b)$ in $D$ and
$(\rho_m,\kappa_m)$ in $\Omega\setminus\overline D$, respectively. All four
parameters are positive constants. The corresponding wave speeds, wavenumbers,
and density contrast are
\[
    c_b:=\sqrt{\kappa_b/\rho_b},
    \qquad
    c_m:=\sqrt{\kappa_m/\rho_m},
    \qquad
    k_b:=\omega/c_b,
    \qquad
    k_m:=\omega/c_m,
    \qquad
    \delta:=\rho_b/\rho_m .
\]
We work in the high-contrast regime
\begin{equation}\label{eq:high-contrast-regime}
    \delta\ll 1,
    \qquad
    c_b,c_m=\mathcal O(1).
\end{equation}
The material coefficients are
\[
    \rho(x):=\rho_b\chi_D(x)
    +\rho_m\chi_{\Omega\setminus\overline D}(x),
    \qquad
    \kappa(x):=\kappa_b\chi_D(x)
    +\kappa_m\chi_{\Omega\setminus\overline D}(x).
\]
The Kerr coefficient may vary between resonators, but is zero in the surrounding medium:
\[
    \sigma_D(x):=\sum_{j=1}^N\sigma_j\chi_{D_j}(x),
    \qquad
    \sigma_j\in\mathbb R .
\]
Accordingly, we define the cubic Kerr nonlinearity by
\begin{align}\label{eq:def-Nsigma}
    \mathcal N_\sigma[u](x):=\sigma_D(x)|u(x)|^2u(x).
\end{align}

For a fixed Bloch parameter $\alpha\in Y^*$ and a prescribed
$\alpha$-quasiperiodic incident field $u^i$, the nonlinear scattering problem is
to find an $\alpha$-quasiperiodic field $u$ such that
\begin{equation}\label{eq:periodic-transmission-problem}
    \begin{cases}
        \displaystyle
        \nabla\!\cdot\!\bigl(\frac{1}{\rho}\,\nabla u\bigr)
        +\dfrac{\omega^2}{\kappa}
        \bigl(u+\mathcal N_\sigma[u]\bigr)=0
        & \text{in } \Omega\setminus\partial D,\\[2mm]
        \left.u\right|_+=\left.u\right|_-
        & \text{on } \partial D,\\[2mm]
        \displaystyle
        \frac{1}{\rho_m}\left.\partial_\nu u\right|_+
        =
        \frac{1}{\rho_b}\left.\partial_\nu u\right|_-
        & \text{on } \partial D,\\[3mm]
        u(x_\ell+\zeta,x_d)
        =
        e^{\mathrm i\alpha\cdot\zeta}u(x_\ell,x_d)
        & \text{for } \zeta\in\Lambda .
    \end{cases}
\end{equation}
The first equation is understood separately in $D$ and in
$\Omega\setminus\overline D$. Equivalently,
\[
    \Delta u+k_b^2\bigl(u+\sigma_j|u|^2u\bigr)=0
    \quad \text{in }D_j,\quad 1\le j\le N,
    \qquad
    \Delta u+k_m^2u=0
    \quad \text{in }\Omega\setminus\overline D .
\]
Together with the two interface conditions, this gives the nonlinear Helmholtz
equation inside the resonators and the linear Helmholtz equation in the exterior. It remains to impose the outgoing condition at infinity. The scattered field $u-u^i$ is required to be
outgoing as $|x_d|\to\infty$. More precisely, an outgoing
$\alpha$-quasiperiodic field $w$ admits the Rayleigh--Bloch expansion
\cite{ammari2018mathematical}
\begin{equation}\label{eq:outgoing-radiation-condition}
    w(x_\ell,x_d)
    =
    \begin{cases}
        \displaystyle
        \sum_{\eta\in\Lambda^*}
        w^\alpha_\eta(+h)\,
        e^{\mathrm i\alpha_\eta\cdot x_\ell}
        e^{+\mathrm i\beta_\eta(k_m)(x_d-h)},
        & x_d\ge +h,\\[3mm]
        \displaystyle
        \sum_{\eta\in\Lambda^*}
        w^\alpha_\eta(-h)\,
        e^{\mathrm i\alpha_\eta\cdot x_\ell}
        e^{-\mathrm i\beta_\eta(k_m)(x_d+h)},
        & x_d\le -h ,
    \end{cases}
\end{equation}
where $h>0$ is chosen such that $\overline D\subset\Omega_h$. Here, the Fourier coefficients are
\begin{equation}\label{eq:rayleigh-coefficients}
    w^\alpha_\eta(\pm h)
    :=
    \frac{1}{|Y|}
    \int_Y
        w(x_\ell,\pm h)\,
        e^{-\mathrm i\alpha_\eta\cdot x_\ell}
    \,dx_\ell,
    \qquad
    \eta\in\Lambda^* .
\end{equation}
For the scattering problem, we set $w=u-u^i$. The homogeneous case $u^i=0$
leads to the nonlinear resonance problem, where the unknowns are nontrivial
outgoing pairs $(\omega,u)$. In what follows, we study this problem in the
subwavelength regime of \cref{ass:subwavelength-low-frequency}.

\begin{assumption}
\label{ass:subwavelength-low-frequency}
Fix a vector $a\in\mathbb R^{d-1}$ satisfying $|a|<c_m^{-1}$. We consider
\[
    \alpha=\omega a,
    \qquad
    0<|\omega|<\omega_{\rm sw},
    \qquad
    \omega_{\rm sw}:=
    \frac{2\pi c_m}{L(1+c_m|a|)} .
\]
\end{assumption}

\begin{remark}
For $0<|\omega|<\omega_{\rm sw}$ and every
$\eta\in\Lambda^*\setminus\{0\}$, the reverse triangle inequality gives
\[
    |\alpha+\eta|
    \ge |\eta|-|\alpha|
    \ge 2\pi/L-|\omega| |a|
    >|\omega|/c_m
    = |k_m|.
\]
Thus, no nonzero diffraction order reaches a Rayleigh threshold in the
subwavelength regime. Moreover, if $\omega>0$ is real, then $|\alpha|=\omega|a|<\omega/c_m=k_m$,
so the order $\eta=0$ is the unique propagating order, while all
$\eta\ne0$ orders are evanescent.
\end{remark}

\subsection{Green's function}

We recall the quasiperiodic Green function associated with the lattice
$\Lambda$ \cite{ammari2018mathematical}. For a Bloch parameter $\alpha$ and a wavenumber $k$, the outgoing $\alpha$-quasiperiodic Green
function for the Helmholtz operator admits the spectral representation
\begin{equation}\label{eq:def-Gak}
    G^{\alpha,k}(x)
    =
    \sum_{\eta\in\Lambda^*}
    \frac{
        e^{\mathrm i\alpha_\eta\cdot x_\ell}
        e^{\mathrm i\beta_\eta(k)|x_d|}
    }{
        2|Y|\,\mathrm i\,\beta_\eta(k)
    },
    \qquad
    x\notin \Lambda\times\{0\},
\end{equation}
which satisfies, in the sense of distributions, the following equation:
\[
    \bigl(\Delta+k^2\bigr)G^{\alpha,k}
    =
    \sum_{\zeta\in\Lambda}
    e^{\mathrm i\alpha\cdot\zeta}\,
    \delta\bigl(x-(\zeta,0)\bigr).
\]
For the degenerate case $\alpha=k=0$, we denote by $G^{0,0}$ the periodic Laplace
Green function
\begin{equation}\label{eq:def-G00}
    G^{0,0}(x)
    =
    \frac{|x_d|}{2|Y|}
    -
    \sum_{\eta\in\Lambda^*\setminus\{0\}}
    \frac{
        e^{\mathrm i\eta\cdot x_\ell}
        e^{-|\eta||x_d|}
    }{
        2|Y|\,|\eta|
    } .
\end{equation}

We work in the low-frequency scaling $\alpha=\omega a$ and $k=\omega/c$,
with $|a|<c^{-1}$.
We define the positive constant $\tau := \sqrt{c^{-2} - |a|^2} > 0$.
In the applications below, the exterior Green function is obtained by taking $c=c_m$ and $\tau := \sqrt{c_m^{-2} - |a|^2}$.
As $\omega\to0$, the Green function admits the following asymptotic expansion
(see, e.g., \cite{ammari2017mathematical}):
\begin{equation}\label{eq:green-low-frequency-expansion}
G^{\alpha,k}(x)
=
\sum_{n=-1}^{\infty} \omega^n G_n^{a,c}(x),
\end{equation}
uniformly on compact subsets of $\mathbb R^d \setminus (\Lambda \times \{0\})$. The first three coefficients are given explicitly by
\begin{equation}\label{eq:green-coefficients}
\begin{aligned}
G_{-1}^{a,c}(x)
= \dfrac{1}{2\mathrm{i}\,\tau |Y|}, \quad
G_{0}^{a,c}(x)
= G^{0,0}(x) + \dfrac{a\cdot x_\ell}{2\tau |Y|}, \quad
G_{1}^{a,c}(x)
= \dfrac{\mathrm{i}\bigl(a\cdot x_\ell + \tau |x_d|\bigr)^2}{4\tau |Y|}
+ \mathrm{i}\, a\cdot \widehat G_1(x),
\end{aligned}
\end{equation}
where $\widehat G_1 : \mathbb R^d \to \mathbb R^{d-1}$ is a real-valued function, independent of $a$ and $c$, and satisfies
\[
\widehat G_1(-x_\ell,x_d) = -\widehat G_1(x_\ell,x_d),
\qquad
\widehat G_1(x_\ell,-x_d) = \widehat G_1(x_\ell,x_d).
\]
We decompose the first-order coefficient into its even and odd parts:
\begin{align}\label{eq:G1-parity-decomposition}
G_1^{a,c}(x)
=
\mathrm{i}
\left(
\dfrac{(a\cdot x_\ell)^2 + \tau^2 x_d^2}{4\tau |Y|}
\right)
+
\mathrm{i}
\left(
\dfrac{(a\cdot x_\ell)|x_d|}{2|Y|}
+ a\cdot \widehat G_1(x)
\right)
= \mathrm{i}\, G_{1,\mathrm{e}}^{a,c}(x)
+ \mathrm{i}\, G_{1,\mathrm{o}}^{a,c}(x).
\end{align}
Here, $G_{1,\mathrm{e}}^{a,c}$ is real-valued and even with respect to both $x_\ell$ and $x_d$, whereas $G_{1,\mathrm{o}}^{a,c}$ is real-valued, odd in $x_\ell$, and even in $x_d$. The higher-order coefficients $G_n^{a,c}$ satisfy the recurrence relations:
\begin{equation}\label{eq:green-coefficient-recursion}
\Delta G_n^{a,c}(x)
+
\frac{1}{c^2} G_{n-2}^{a,c}(x)
=
\sum_{\zeta\in\Lambda}
\frac{(\mathrm{i}\, a\cdot\zeta)^n}{n!}
\,\delta\bigl(x-(\zeta,0)\bigr),\qquad n \ge 1.
\end{equation}

\subsection{Layer potentials}

We define the layer potentials and boundary integral operators associated with
the quasiperiodic Green function. For $\psi\in H^{-1/2}(\partial D)$, the
single-layer potential is
\[
    \mathcal S_D^{\alpha,k}[\psi](x)
    :=
    \int_{\partial D}G^{\alpha,k}(x-y)\psi(y)\,d\sigma(y),
    \qquad x\in\mathbb R^d\setminus\partial D .
\]
The same notation is used for its trace on $\partial D$ as a bounded boundary operator. Denote by
\[
    (\mathcal K_D^{-\alpha,k})^*[\psi](x)
    :=
    \int_{\partial D}
    \frac{\partial G^{\alpha,k}(x-y)}{\partial\nu(x)}
    \psi(y)\,d\sigma(y),
    \qquad x\in\partial D .
\]
Away from Rayleigh thresholds, the standard mapping properties yield (see, e.g., \cite{ammari2018mathematical,mclean2000strongly})
\[
    \mathcal S_D^{\alpha,k}:
    H^{-1/2}(\partial D)\to H^{1/2}(\partial D),
    \qquad
    (\mathcal K_D^{-\alpha,k})^*:
    H^{-1/2}(\partial D)\to H^{-1/2}(\partial D).
\]
Moreover, the jump relation holds in $H^{-1/2}(\partial D)$:
\begin{equation}\label{eq:jump-formula}
    \left.
    \frac{\partial}{\partial\nu}\mathcal S_D^{\alpha,k}[\psi]
    \right|_{\pm}
    =
    \left(
        \pm\frac12\mathcal I
        +
        (\mathcal K_D^{-\alpha,k})^*
    \right)[\psi].
\end{equation}

The expansion \eqref{eq:green-low-frequency-expansion} yields the following
operator norm expansions for the boundary operators:
\begin{equation}\label{eq:boundary-operator-asymptotic}
    \mathcal S_D^{\alpha,k}
    =
    \sum_{n=-1}^{\infty}\omega^n\mathcal S_{D,n}^{a,c},
    \qquad
    (\mathcal K_D^{-\alpha,k})^*
    =
    \sum_{n=-1}^{\infty}\omega^n(\mathcal K_{D,n}^{-a,c})^*.
\end{equation}
Here, $\mathcal S_{D,n}^{a,c}$ and
$(\mathcal K_{D,n}^{-a,c})^*$ are defined by replacing the kernel
$G^{\alpha,k}$ by $G_n^{a,c}$. Since $G_{-1}^{a,c}$ is constant,
$(\mathcal K_{D,-1}^{-a,c})^*=0$. We shall use the truncated operators, retaining the singular- and zeroth-order terms,
\[
    \widehat{\mathcal S}_D^{\alpha,k}
    :=
    \omega^{-1}\mathcal S_{D,-1}^{a,c}
    +
    \mathcal S_{D,0}^{a,c},
    \qquad
    (\widehat{\mathcal K}_D^{-\alpha,k})^*
    :=
    (\mathcal K_{D,0}^{-a,c})^* .
\]
Consequently, the full boundary operators admit the following expansions as $\omega\to0$:
\begin{align}\label{eq:Sak-Ksak-expansion}
    \mathcal S_D^{\alpha,k}
    =
    \widehat{\mathcal S}_D^{\alpha,k}
    +
    \omega\mathcal S_{D,1}^{a,c}
    +
    \mathcal O(\omega^2),\quad
    (\mathcal K_D^{-\alpha,k})^*
    =
    (\widehat{\mathcal K}_D^{-\alpha,k})^*
    +
    \omega(\mathcal K_{D,1}^{-a,c})^*
    +
    \mathcal O(\omega^2).
\end{align}
The recurrence \eqref{eq:green-coefficient-recursion}, together with the divergence theorem, yields the following integral identities.

\begin{lemma}[{\cite[Lemma~3.1]{ammari2021bound}}]
\label{lem:ksn-integral}
For any $\psi\in H^{-1/2}(\partial D)$ and $1\le i\le N$,
\[
    \left\langle \left(-\frac12\mathcal I + (\widehat{\mathcal K}_D^{-\alpha,k})^*\right)[\psi], 1 \right\rangle_{\partial D_i}= 0,\quad
    \left\langle (\mathcal K_{D,n}^{-a,c})^*[\psi], 1 \right\rangle_{\partial D_i}
    =
    -\frac{1}{c^2}
    \left( \mathcal S_{D,n-2}^{a,c}[\psi], 1 \right)_{D_i},
    \quad n \ge 1.
\]
\end{lemma}

For $\psi\in H^{-1/2}(\partial D)$, define the total charge and the
periodic/transverse first moments by
\begin{equation}\label{eq:def-moment-functionals}
    \mathfrak m[\psi]
    :=
    \langle\psi,1\rangle_{\partial D},
    \qquad
    \mathfrak m_\ell[\psi]
    :=
    \langle\psi,y_\ell\rangle_{\partial D},
    \qquad
    \mathfrak m_d[\psi]
    :=
    \langle\psi,y_d\rangle_{\partial D}.
\end{equation}
These functionals $\mathfrak m,\mathfrak m_d:
    H^{-1/2}(\partial D)\to\mathbb C$, $\mathfrak m_\ell:
    H^{-1/2}(\partial D)\to\mathbb C^{d-1}$
are linear and bounded. Using the explicit form of $G_0^{a,c}$ in \eqref{eq:green-coefficients}, the truncated operators are rewritten as
\begin{gather}
    \widehat{\mathcal S}_D^{\alpha,k}[\psi]
    =
    -\frac{\mathrm i\,\mathfrak m[\psi]}{2\omega\tau|Y|}
    +
    \mathcal S_D^{0,0}[\psi]
    +
    \frac{
        (a\cdot x_\ell)\,\mathfrak m[\psi]
        -
        a\cdot\mathfrak m_\ell[\psi]
    }{2\tau|Y|},
    \label{eq:hatSak-formula}
    \\
    (\widehat{\mathcal K}_D^{-\alpha,k})^*[\psi]
    =
    (\mathcal K_D^{0,0})^*[\psi]
    +
    \frac{(a\cdot\nu_\ell)\,\mathfrak m[\psi]}{2\tau|Y|},\quad \nu=(\nu_\ell,\nu_d).
    \label{eq:hatKak-formula}
\end{gather}

To study the invertibility of $\widehat{\mathcal S}_D^{\alpha,k}$, we first establish the following auxiliary isomorphism, whose proof is deferred to Appendix \ref{app:auxiliary-proofs}.

\begin{lemma}\label{lem:H0-isomorphism}
Define
$\mathcal H : H^{-1/2}(\partial D)\times\mathbb C
\to H^{1/2}(\partial D)\times\mathbb C$ by
\[
    \mathcal H[\psi,s]
    :=
    \bigl(\mathcal S_D^{0,0}[\psi]+s,\mathfrak m[\psi]\bigr).
\]
Then $\mathcal H$ is a linear isomorphism. Set
$H^{-1/2}_0(\partial D):=\ker\mathfrak m$. The reduced map
\begin{equation}\label{eq:def-H0}
    \mathcal H_0:
    H^{-1/2}_0(\partial D)\times\mathbb C
    \to
    H^{1/2}(\partial D),
    \qquad
    \mathcal H_0[\psi,s]
    :=
    \mathcal S_D^{0,0}[\psi]+s,
\end{equation}
is also a linear isomorphism.
\end{lemma}

The holomorphy of $(\widehat{\mathcal S}_D^{\alpha,k})^{-1}$ was established in
\cite{ammari2021bound,ammari2022exceptional}, where an implicit asymptotic
expansion was obtained. Here, we instead take a direct approach to obtain the
following explicit inverse formula for
$(\widehat{\mathcal S}_D^{\alpha,k})^{-1}$; the proof is given in
Appendix \ref{app:auxiliary-proofs}.

\begin{proposition}[Inverse truncated single-layer operator]
\label{prop:truncated-single-layer-inverse}
Define the auxiliary boundary function
\[
    f_D
    :=
    \mathcal S_D^{0,0}\bigl[\chi_{\partial D}^{\mathrm{av}}\bigr]
    +
    \frac{
        a\cdot x_\ell
        -
        a\cdot\mathfrak m_\ell
        \bigl[\chi_{\partial D}^{\mathrm{av}}\bigr]
    }{2\tau|Y|}
    \in H^{1/2}(\partial D).
\]
For each $f\in H^{1/2}(\partial D)$, define the pairs
$(\psi_f^0,s_f^0)$ and $(\psi_D^0,s_D^0)$ by
\[
    (\psi_f^0,s_f^0):=\mathcal H_0^{-1}[f],
    \qquad
    (\psi_D^0,s_D^0):=\mathcal H_0^{-1}[f_D].
\]
With these pairs, set
\begin{equation}\label{eq:mhatpsi1}
    s_f^1
    :=
    \mathrm i
    \bigl(
        a\cdot\mathfrak m_\ell[\psi_f^0]
        +
        2\tau|Y|\,s_f^0
    \bigr),
    \qquad
    s_D^1
    :=
    \mathrm i
    \bigl(
        a\cdot\mathfrak m_\ell[\psi_D^0]
        +
        2\tau|Y|\,s_D^0
    \bigr).
\end{equation}
Then, for $|\omega|<|s_D^1|^{-1}$, the operator
\[
    \widehat{\mathcal S}_D^{\alpha,k}:
    H^{-1/2}(\partial D)\to H^{1/2}(\partial D)
\]
is invertible and its inverse extends holomorphically to this disk. More
precisely,
\begin{equation}\label{eq:closed-form-inverse-hatSak}
\bigl(\widehat{\mathcal S}_D^{\alpha,k}\bigr)^{-1}[f]
=
\psi_f^0
+
\frac{\omega s_f^1}{1+\omega s_D^1}
\bigl(\chi_{\partial D}^{\mathrm{av}}-\psi_D^0\bigr).
\end{equation}
Consequently, we obtain the asymptotic expansion
\begin{equation}\label{eq:expansion-of-inverse-hatSak}
    \bigl(\widehat{\mathcal S}_D^{\alpha,k}\bigr)^{-1}
    =
    \mathcal S_0^{-1}
    +
    \omega\widehat{\mathcal S}_1^{-1}
    +
    \mathcal O(\omega^2),
    \qquad
    \mathcal S_0^{-1}[f]:=\psi_f^0,
    \qquad
    \widehat{\mathcal S}_1^{-1}[f]
    :=
    s_f^1
    \bigl(
        \chi_{\partial D}^{\mathrm{av}}
        -
        \psi_D^0
    \bigr).
\end{equation}
\end{proposition}

Combining \cref{prop:truncated-single-layer-inverse} with \eqref{eq:Sak-Ksak-expansion}, and expanding the resulting perturbation via a Neumann series, yields the inverse expansion for the full single-layer operator.

\begin{corollary}[Inverse of the single-layer operator]
\label{cor:single-layer-inverse-expansion}
There exists $\omega_0>0$, with $\omega_0<\min\{\omega_{\rm sw},|s_D^1|^{-1}\}$,
such that, for $|\omega|<\omega_0$,
\[
    \mathcal S_D^{\alpha,k}:
    H^{-1/2}(\partial D)\to H^{1/2}(\partial D)
\]
is boundedly invertible. Its inverse extends holomorphically to
$|\omega|<\omega_0$, and
\begin{equation}\label{eq:expansion-of-inverse-Sak}
    \bigl(\mathcal S_D^{\alpha,k}\bigr)^{-1}
    =
    \mathcal S_0^{-1}
    +
    \omega\mathcal S_1^{-1}
    +
    \mathcal O(\omega^2),
\end{equation}
where the first-order correction is given by
\begin{equation}\label{eq:inverse-S1}
    \mathcal S_1^{-1}
    :=
    \widehat{\mathcal S}_1^{-1}
    -
    \mathcal S_0^{-1}
    \mathcal S_{D,1}^{a,c}
    \mathcal S_0^{-1}.
\end{equation}
\end{corollary}

\subsection{Capacitance matrices}

We next define the capacitance matrices and record the identities needed for
the subwavelength reduction.

\begin{definition}\label{def:capacitance-densities}
For $1\le j\le N$, define the pairs $(\psi_j^0,s_j^0)$ and the densities $\widehat\psi_j^1,\psi_j^1$ by
\begin{equation}\label{eq:def-capacitance-densities}
(\psi_j^0,s_j^0):=\mathcal H_0^{-1}[\chi_{\partial D_j}],\qquad
\widehat\psi_j^1:=\widehat{\mathcal S}_1^{-1}[\chi_{\partial D_j}],\qquad
\psi_j^1:=\mathcal S_1^{-1}[\chi_{\partial D_j}],
\end{equation}
where $\widehat{\mathcal S}_1^{-1}$ and $\mathcal S_1^{-1}$ are from \cref{prop:truncated-single-layer-inverse,cor:single-layer-inverse-expansion}. The entries of the leading capacitance matrix $C^0\in\mathbb C^{N\times N}$ and the first-order matrices $\widehat C^1,C^1\in\mathbb C^{N\times N}$ are given by
\begin{equation}\label{eq:capacitance-entries}
C^0_{ij}:=-\langle \psi_j^0, 1\rangle_{\partial D_i},\qquad
\widehat C^1_{ij}:=-\langle \widehat\psi_j^1, 1\rangle_{\partial D_i},\qquad
C^1_{ij}:=-\langle \psi_j^1, 1\rangle_{\partial D_i},\qquad 1\leq i,j\leq N.
\end{equation}
Let $s^0:=(s_1^0,\ldots,s_N^0)^\top\in\mathbb C^N$. We also define the moment vectors $m_\ell^a,m_d^\tau\in\mathbb C^N$ by
\[
m_{\ell,j}^a:=a\cdot\mathfrak m_\ell[\psi_j^0],\qquad
m_{d,j}^\tau:=\tau\,\mathfrak m_d[\psi_j^0],\qquad 1\le j\le N,
\]
where the superscripts denote the dependence on $a$ and $\tau$, respectively.
\end{definition}

We now collect the key properties of these objects.
\begin{lemma}\label{lem:C1-decomposition}
The following statements hold.

\emph{(i)} The vector $s^0$ and the matrix $C^0$ are real-valued. Moreover, for each $1\le j\le N$,
\begin{align}\label{eq:capacitance-density-expansion}
    \psi_j^0\in H^{-1/2}_0(\partial D;\mathbb R),\qquad
    \widehat\psi_j^1=s_j^1
    \bigl(\chi_{\partial D}^{\mathrm{av}}-\psi_D^0\bigr),\qquad
    \psi_j^1=\widehat\psi_j^1-\mathcal S_0^{-1}\mathcal S_{D,1}^{a,c}[\psi_j^0],
\end{align}
with $s_j^1:=s_{\chi_{\partial D_j}}^1=\mathrm i\bigl(a\cdot\mathfrak m_\ell[\psi_j^0]+2\tau|Y|\,s_j^0 \bigr)$ as defined in \eqref{eq:mhatpsi1}.

\emph{(ii)} Summing over all components gives
\[
    \sum_{j=1}^N\psi_j^0=0,
    \qquad
    \sum_{j=1}^N s_j^0=1,
    \qquad
    \sum_{j=1}^N C^0_{ij}=0 .
\]

\emph{(iii)} The first-order matrix splits as
$C^1=\widehat C^1+\widetilde C^1$, where
\[
\widetilde C^1_{ij}
:=
\bigl\langle \mathcal S_0^{-1}\mathcal S_{D,1}^{a,c}[\psi_j^0], 1\bigr\rangle_{\partial D_i}
=
\bigl\langle \mathcal S_{D,1}^{a,c}[\psi_j^0], \psi_i^0\bigr\rangle_{\partial D}
=
\mathrm{i}\,\widetilde C^{1,e}_{ij}+\mathrm{i}\,\widetilde C^{1,o}_{ij}.
\]
Here, the entries of $\widetilde C^{1,e}$ and $\widetilde C^{1,o}$ are defined by
\begin{align*}
    \widetilde C^{1,e}_{ij}
    &:=
    \int_{\partial D}\!\int_{\partial D}
        G_{1,e}^{a,c}(x-y)
        \psi_j^0(y)\psi_i^0(x)
    \,d\sigma(y)d\sigma(x),
    \\
    \widetilde C^{1,o}_{ij}
    &:=
    \int_{\partial D}\!\int_{\partial D}
        G_{1,o}^{a,c}(x-y)
        \psi_j^0(y)\psi_i^0(x)
    \,d\sigma(y)d\sigma(x)
\end{align*}
with $G_{1,e}^{a,c}$ and $G_{1,o}^{a,c}$ as in \eqref{eq:G1-parity-decomposition}. Moreover,
\[
    \widehat C^1
    =
    \frac{\mathrm i}{2\tau|Y|}
    \bigl(m_\ell^a-2\tau|Y|s^0\bigr)
    \bigl(m_\ell^a+2\tau|Y|s^0\bigr)^\top,\quad
    \widetilde C^1
    =
    -\frac{\mathrm i}{2\tau|Y|}
    \left(
        m_\ell^a(m_\ell^a)^\top
        +
        m_d^\tau(m_d^\tau)^\top
    \right)
    +
    \mathrm i\,\widetilde C^{1,o}.
\]
Consequently,
\begin{align}
    C^1
    =
    -\mathrm i
    \left(
        2\tau|Y|\,s^0(s^0)^\top
        +
        \frac{m_d^\tau(m_d^\tau)^\top}{2\tau|Y|}
    \right)
    +
    \mathrm i
    \left(
        m_\ell^a(s^0)^\top
        -
        s^0(m_\ell^a)^\top
        +
        \widetilde C^{1,o}
    \right).
    \label{eq:C1-formula}
\end{align}
\end{lemma}

\begin{proof}
\emph{(i)--(ii).}
The reality of $s^0$ and $C^0$ and the inclusion $\psi_j^0\in H^{-1/2}_0(\partial D)$ follow immediately from \eqref{eq:def-capacitance-densities} and the definition of $\mathcal H_0$.
The formulas for $\widehat\psi_j^1$ and $\psi_j^1$ follow from \eqref{eq:expansion-of-inverse-hatSak} and \eqref{eq:inverse-S1}, respectively.
For the summation identities, since $\mathcal H_0[0,1]=\chi_{\partial D}$, using injectivity gives $\sum_j\psi_j^0=0$ and $\sum_j s_j^0=1$;
the row-sum identity for $C^0$ then follows from \eqref{eq:capacitance-entries}.

\emph{(iii).}
We first compute $\widehat C^1$. From the explicit form of $\widehat\psi_j^1$ in \eqref{eq:capacitance-density-expansion},
\[
\widehat C^1_{ij}=-s_j^1\langle \chi_{\partial D}^{\mathrm{av}}-\psi_D^0,1\rangle_{\partial D_i}.
\]
Using the bilinear symmetry of $\mathcal S_D^{0,0}$, together with
\[
\mathcal S_D^{0,0}[\psi_i^0]+s_i^0=\chi_{\partial D_i},\qquad
\mathcal S_D^{0,0}[\psi_D^0]+s_D^0=f_D,
\]
we obtain
\[
\langle \psi_D^0,1\rangle_{\partial D_i}
=
\langle f_D-s_D^0,\psi_i^0\rangle_{\partial D}
=
\langle \chi_{\partial D}^{\mathrm{av}},1\rangle_{\partial D_i}
-
s_i^0
+
\frac{m_{\ell,i}^a}{2\tau|Y|},
\]
where we have used the zero-charge property of $\psi_i^0$ and $\psi_D^0$. Hence,
\[
\langle \chi_{\partial D}^{\mathrm{av}}-\psi_D^0,1\rangle_{\partial D_i}
=
s_i^0-\frac{m_{\ell,i}^a}{2\tau|Y|}
=
\frac{2\tau|Y|s_i^0-m_{\ell,i}^a}{2\tau|Y|}.
\]
Together with the explicit form of $s_j^1$, this gives the asserted formula for $\widehat C^1$.

It remains to compute $\widetilde C^1$. By the symmetry of $\mathcal S_D^{0,0}$,
\[
\widetilde C^1_{ij}
=
\bigl\langle \mathcal S_0^{-1}\mathcal S_{D,1}^{a,c}[\psi_j^0], \chi_{\partial D_i}\bigr\rangle_{\partial D}
=
\bigl\langle \mathcal S_{D,1}^{a,c}[\psi_j^0], \psi_i^0\bigr\rangle_{\partial D}
=
\int_{\partial D}\!\int_{\partial D}
        G_{1}^{a,c}(x-y)
        \psi_j^0(y)\psi_i^0(x)
    \,d\sigma(y)d\sigma(x),
\]
where $G_1^{a,c}=\mathrm{i}G_{1,e}^{a,c}+\mathrm{i}G_{1,o}^{a,c}$ from \eqref{eq:G1-parity-decomposition}. Expanding the kernel gives $\widetilde C^1=\mathrm{i}\widetilde C^{1,e}+\mathrm{i}\widetilde C^{1,o}$. Since $\psi_i^0,\psi_j^0\in H^{-1/2}_0(\partial D)$, the zero-charge property implies that
\begin{align*}
\int_{\partial D}\int_{\partial D}
(a\cdot x_\ell-a\cdot y_\ell)^2
\psi_j^0(y)\psi_i^0(x)\,d\sigma(y)d\sigma(x)
&=
-2m_{\ell,i}^a m_{\ell,j}^a,\\
\int_{\partial D}\int_{\partial D}
\tau^2(x_d-y_d)^2
\psi_j^0(y)\psi_i^0(x)\,d\sigma(y)d\sigma(x)
&=
-2m_{d,i}^\tau m_{d,j}^\tau .
\end{align*}
Consequently, we have
\[
\widetilde C^{1,e}_{ij}
=
-\frac{1}{2\tau|Y|}
\left(m_{\ell,i}^a m_{\ell,j}^a+m_{d,i}^\tau m_{d,j}^\tau\right),
\]
which yields the asserted formula for $\widetilde C^1$. Combining this with the formula for $\widehat C^1$ gives \eqref{eq:C1-formula}.
\end{proof}

The row-sum identity for $C^0$ in \cref{lem:C1-decomposition}, together with
the spectral theorem for symmetric generalized eigenvalue problems, gives the
following spectral normalization \cite{cbms}.

\begin{lemma}[Spectral normalization of the capacitance matrix]
\label{lem:spectral-capacitance}
The leading capacitance matrix $C^0$ is real symmetric and positive
semidefinite. Moreover,
\[
    C^0\mathbf 1=0,
    \qquad
    \ker C^0=\operatorname{span}\{\mathbf 1\},
\]
where $\mathbf 1:=(1,\ldots,1)^\top$. The generalized eigenvalue problem
$C^0p=\lambda Vp$ admits a $V$-orthonormal basis of real eigenvectors. We denote
the corresponding eigenpairs by $\{(\lambda_j,p_j)\}_{j=1}^N$, ordered so that
\begin{equation}\label{eq:eigenvalue-problem-of-C}
    C^0p_j=\lambda_jVp_j,
    \qquad
    p_i^\top Vp_j=\delta_{ij},
    \qquad
    0=\lambda_1\le\lambda_2\le\cdots\le\lambda_N .
\end{equation}
\end{lemma}

\section{Subwavelength resonances in general configurations}
\label{section:general-subwavelength-resonances}

This section has three objectives. We first prove a spectral location result for
the nonlinear resonance problem. We then eliminate the exterior
field by means of the Dirichlet-to-Neumann (DtN) operator.
Finally, under \cref{ass:subwavelength-low-frequency}, we decompose the
interior field into its componentwise averages and a zero-average correction
and derive a finite-dimensional amplitude equation.

\begin{assumption}
\label{ass:positive-effective-coefficient}
An outgoing solution $(\omega,u)$ of \eqref{eq:periodic-transmission-problem} with $u^i=0$ is assumed
to satisfy
\[
    1+\sigma_D(x)|u(x)|^2>0
    \qquad
    \text{for a.e. }x\in D.
\]
\end{assumption}

\begin{remark}
\label{rem:positive-effective-small-amplitude}
\emph{(i)} The assumption is automatic in the self-defocusing case
$\sigma_D\ge0$.

\emph{(ii)} For sign-indefinite $\sigma_D$, it is an amplitude-smallness
condition. Indeed, whenever an $L^\infty(D)$ bound is available, one has
\[
    1+\sigma_D(x)|u(x)|^2
    \ge
    1-\|\sigma_D\|_{L^\infty(D)}\|u\|_{L^\infty(D)}^2 .
\]
Thus, the assumption holds if
\[
    \|u\|_{L^\infty(D)}^2
    <
    \|\sigma_D\|_{L^\infty(D)}^{-1}
\]
with the convention that the condition is void when $\sigma_D\equiv0$. In this
sense, the assumption is a small-amplitude condition.
\end{remark}

We next show that, for fixed real Bloch fibres, outgoing resonances satisfying this positivity condition cannot lie in the open upper half-plane.

\begin{proposition}\label{prop:no-upper-half-plane}
Let $\alpha\in Y^*$ be fixed and let $(\omega,u)$ be an outgoing
$\alpha$-quasiperiodic solution of \eqref{eq:periodic-transmission-problem} with $u^i=0$ satisfying
\cref{ass:positive-effective-coefficient}. If $\Im\omega>0$, then
$u\equiv0$. Consequently, every nontrivial outgoing resonance satisfying
\cref{ass:positive-effective-coefficient} has $\Im\omega\le0$.
\end{proposition}

\begin{proof}
Let $h>0$ be such that $\overline D\subset\Omega_h$, and set
$\Gamma_h^\pm:=Y\times\{\pm h\}$. Multiplying \eqref{eq:periodic-transmission-problem} by
$\overline u$, integrating over $\Omega_h$, and using the transmission
conditions gives
\begin{equation}\label{eq:omega-h-green}
    \int_{\Omega_h}\frac1\rho|\nabla u|^2\,dx
    -
    \omega^2
    \int_{\Omega_h}
        \frac1\kappa
        \bigl(1+\sigma_D(x)|u|^2\bigr)|u|^2
    \,dx
    =
    \int_{\partial\Omega_h}
        \frac1\rho\overline u\,\partial_\nu u
    \,d\sigma .
\end{equation}
The lateral boundary terms cancel by real quasiperiodicity. On
$\Gamma_h^+\cup\Gamma_h^-$, the outgoing Rayleigh--Bloch expansion yields
\begin{equation}\label{eq:omega-h-rb}
    \int_{\partial\Omega_h}
        \frac1\rho\overline u\,\partial_\nu u
    \,d\sigma
    =
    \frac{\mathrm i |Y|}{\rho_m}
    \sum_{\eta\in\Lambda^*}
        \beta_\eta(k_m)
        \left(
            |u^\alpha_\eta(+h)|^2
            +
            |u^\alpha_\eta(-h)|^2
        \right).
\end{equation}
We distinguish three cases:

\emph{(i)} Suppose first that $\Re\omega>0$ and $\Im\omega>0$. Taking imaginary
parts in \eqref{eq:omega-h-green} gives
\[
    -\Im(\omega^2)
    \int_{\Omega_h}
        \frac1\kappa
        \bigl(1+\sigma_D(x)|u|^2\bigr)|u|^2
    \,dx
    =
    \frac{|Y|}{\rho_m}
    \sum_{\eta\in\Lambda^*}
        \Re\beta_\eta(k_m)
        \left(
            |u^\alpha_\eta(+h)|^2
            +
            |u^\alpha_\eta(-h)|^2
        \right).
\]
Since $\Im(\omega^2)=2\Re\omega\,\Im\omega>0$, the left-hand side is
nonpositive. On the outgoing branch, $\Re\beta_\eta(k_m)\ge0$, so the
right-hand side is nonnegative. Hence, both sides vanish. In particular,
\[
    \int_{\Omega_h}
        \frac1\kappa
        \bigl(1+\sigma_D(x)|u|^2\bigr)|u|^2
    \,dx
    =
    0 .
\]
The coefficient in the integrand is strictly positive in $D$ by assumption and
is equal to one in the exterior. Therefore, $u=0$ in $\Omega_h$.

\emph{(ii)} Next, suppose that $\Re\omega<0$ and $\Im\omega>0$. Taking the
complex conjugate of \eqref{eq:omega-h-green} and using \eqref{eq:omega-h-rb},
together with the identity $\beta_\eta(-\overline{k_m})=-\overline{\beta_\eta(k_m)}$ on the same outgoing sheet, we obtain
\[
    \int_{\Omega_h}\frac1\rho|\nabla u|^2\,dx
    -
    \overline\omega^{\,2}
    \int_{\Omega_h}
        \frac1\kappa
        \bigl(1+\sigma_D(x)|u|^2\bigr)|u|^2
    \,dx
    =
    \frac{\mathrm i |Y|}{\rho_m}
    \sum_{\eta\in\Lambda^*}
        \beta_\eta(-\overline{k_m})
        \left(
            |u^\alpha_\eta(+h)|^2
            +
            |u^\alpha_\eta(-h)|^2
        \right).
\]
Taking imaginary parts gives
\[
    -\Im(\overline\omega^{\,2})
    \int_{\Omega_h}
        \frac1\kappa
        \bigl(1+\sigma_D(x)|u|^2\bigr)|u|^2
    \,dx
    =
    \frac{|Y|}{\rho_m}
    \sum_{\eta\in\Lambda^*}
        \Re\beta_\eta(-\overline{k_m})
        \left(
            |u^\alpha_\eta(+h)|^2
            +
            |u^\alpha_\eta(-h)|^2
        \right).
\]
Here, $\Im(\overline\omega^{\,2})=-2\Re\omega\,\Im\omega>0$, while
$\Re\beta_\eta(-\overline{k_m})\ge0$. The same sign argument gives
$u=0$ in $\Omega_h$.

\emph{(iii)} Finally, let $\Re\omega=0$ and $\Im\omega>0$. Write
$\omega=\mathrm i|\omega|$. Then
\[
    \omega^2=-|\omega|^2,
    \qquad
    \beta_\eta(k_m)=\mathrm i\mu_\eta,
    \qquad
    \mu_\eta:=
    \sqrt{|\omega|^2/c_m^2+|\alpha+\eta|^2}>0.
\]
The energy identity becomes
\[
     \int_{\Omega_h}\frac1\rho|\nabla u|^2\,dx
    +
    |\omega|^2
    \int_{\Omega_h}
        \frac1\kappa
        \bigl(1+\sigma_D(x)|u|^2\bigr)|u|^2
    \,dx
    =
    -\frac{|Y|}{\rho_m}
    \sum_{\eta\in\Lambda^*}
        \mu_\eta
        \left(
            |u^\alpha_\eta(+h)|^2
            +
            |u^\alpha_\eta(-h)|^2
        \right).
\]
The left-hand side is nonnegative and the right-hand side is nonpositive. Hence,
both sides vanish. The positivity assumption again implies $u=0$ in
$\Omega_h$.

In all cases $u$ vanishes in a truncated strip containing $\overline D$. The
outgoing Rayleigh--Bloch expansion then gives $u=0$ for $|x_d|>h$.
Consequently, $u\equiv0$ in $\Omega$.
\end{proof}

\begin{remark}
The nonlinear resonance equation is phase invariant: if $(\omega,u)$ is an
outgoing solution, then $(\omega,e^{\mathrm i\theta}u)$ is also an outgoing
solution for every $\theta\in\mathbb R$. This follows from the fact that
\[
    |e^{\mathrm i\theta}u|^2(e^{\mathrm i\theta}u)
    =
    e^{\mathrm i\theta}|u|^2u .
\]
There is also a conjugation symmetry at the level of the Bloch family. Since the
material parameters and Kerr coefficients are real, conjugating the equation
maps
\[
    (\alpha,\omega,u)
    \longmapsto
    (-\alpha,-\overline\omega,\overline u).
\]
Thus, $(-\overline\omega,\overline u)$ solves the nonlinear resonance problem in
the $-\alpha$ quasiperiodic fibre.
\end{remark}

Next, we introduce the finite-dimensional counterpart of the Kerr nonlinearity
for later use.

\begin{definition}
\label{def:finite-dimensional-kerr-map}
The finite-dimensional Kerr map $V_\sigma:\mathbb C^N\to\mathbb C^N$ is defined
componentwise by
\begin{equation}\label{eq:def-Vsigma-q}
    (V_\sigma(q))_i
    :=
    \sigma_i|D_i||q_i|^2q_i,
    \qquad
    1\le i\le N.
\end{equation}
It satisfies the following elementary identities:
\begin{enumerate}[label=(\roman*)]

    \item \emph{Cubic scaling}: $V_\sigma(tq)=|t|^2t\,V_\sigma(q)$, $t\in\mathbb C$;

    \item \emph{Gauge equivariance}: $V_\sigma(e^{\mathrm i\theta}q)=
        e^{\mathrm i\theta}V_\sigma(q)$, $\theta\in\mathbb R$.
\end{enumerate}
Regarded as a real map under the identification
$\mathbb C^N\simeq\mathbb R^{2N}$, $V_\sigma$ is real analytic.
\end{definition}

\subsection{Variational formulation}

We first introduce the exterior Dirichlet-to-Neumann operator and then use it to
eliminate the exterior field.

\begin{definition}\label{def:DtN-operator}
For $\psi\in H^{1/2}(\partial D)$, consider the outgoing
quasiperiodic exterior Dirichlet problem
\begin{equation}
\label{eq:exterior-boundary-value-problem}
(\Delta+k^2)u^{\mathrm{ext}}=0
\quad\text{in }\Omega\setminus\overline D,
\qquad
u^{\mathrm{ext}}=\psi
\quad\text{on }\partial D,
\end{equation}
subject to the outgoing condition
\eqref{eq:outgoing-radiation-condition}. Assume that, for the given pair $(\alpha,k)$, this exterior problem is uniquely
solvable for every boundary datum $\psi\in H^{1/2}(\partial D)$. The exterior
Dirichlet-to-Neumann operator is defined by
\[
    \mathcal T_D^{\alpha,k}[\psi]
    :=
    \left.\partial_\nu u^{\mathrm{ext}}\right|_+,
\]
where $\nu$ denotes the unit normal pointing out of $D$.
\end{definition}

The single-layer representation of the exterior field and the jump relation
\eqref{eq:jump-formula} give the following expression for
$\mathcal T_D^{\alpha,k}$, together with its low-frequency expansion.

\begin{lemma}\label{lem:DtN-low-frequency-expansion}
Assume $\alpha=\omega a$ and $k=\omega/c$, with fixed
$a\in\mathbb R^{d-1}$, $c>0$, and $|a|<c^{-1}$. For
$|\omega|<\omega_0$, where $\omega_0$ is chosen as in
\cref{cor:single-layer-inverse-expansion}, the DtN operator is
\begin{equation}\label{eq:DtN-representation}
    \mathcal T_D^{\alpha,k}
    =
    \left(
        \frac12\mathcal I
        +
        (\mathcal K_D^{-\alpha,k})^*
    \right)
    \left(\mathcal S_D^{\alpha,k}\right)^{-1}.
\end{equation}
Moreover, the operator-valued map $\omega \mapsto \mathcal T_D^{\omega a,\omega/c}$ is holomorphic for $|\omega|<\omega_0$. In particular, as
$\omega\to0$,
\begin{equation}\label{eq:DtN-operator-expansion}
    \mathcal T_D^{\alpha,k}
    =
    \mathcal T_0^{a,c}
    +
    \omega\mathcal T_1^{a,c}
    +
    \omega^2\mathcal R_D^{\alpha,k}
    \quad
    \text{in }
    \mathcal L\bigl(H^{1/2}(\partial D),H^{-1/2}(\partial D)\bigr),
\end{equation}
where $\mathcal R_D^{\alpha,k}$ is uniformly bounded for
$|\omega|<\omega_0$, and
$\omega\mapsto\mathcal R_D^{\omega a,\omega/c}$ is holomorphic with values in
the same operator space.
The first two coefficients are
\begin{equation}\label{eq:T0-T1-definition}
\begin{aligned}
    \mathcal T_0^{a,c}
    :=
    \left(
        \frac12\mathcal I
        +
        (\mathcal K_{D,0}^{-a,c})^*
    \right)
    \mathcal S_0^{-1},
    \qquad
    \mathcal T_1^{a,c}
    :=
    \left(
        \frac12\mathcal I
        +
        (\mathcal K_{D,0}^{-a,c})^*
    \right)
    \mathcal S_1^{-1}
    +
    (\mathcal K_{D,1}^{-a,c})^*
    \mathcal S_0^{-1}.
\end{aligned}
\end{equation}
Finally, for $1\le i,j\le N$,
\begin{equation}\label{eq:DtN-pairing-with-capacitance}
    \left\langle
        \mathcal T_0^{a,c}[\chi_{\partial D_j}],
        \chi_{\partial D_i}
    \right\rangle_{\partial D}
    =
    -C^0_{ij},
    \qquad
    \left\langle
        \mathcal T_1^{a,c}[\chi_{\partial D_j}],
        \chi_{\partial D_i}
    \right\rangle_{\partial D}
    =
    -C^1_{ij}.
\end{equation}
\end{lemma}

Under the solvability assumption in \cref{def:DtN-operator}, the outgoing
exterior field can be eliminated by $\mathcal T_D^{\alpha,k_m}$. Using
$\delta=\rho_b/\rho_m$ in the transmission condition, the homogeneous nonlinear resonance
problem \eqref{eq:periodic-transmission-problem} is equivalent to the interior nonlinear resonance problem
\begin{equation}
\label{eq:interior-nonlinear-resonance-problem}
    \Delta u+k_b^2\bigl(u+\mathcal N_\sigma[u]\bigr)=0
    \quad\text{in }D,
    \qquad
    \left.\partial_\nu u\right|_-
    =
    \delta\mathcal T_D^{\alpha,k_m}[u]
    \quad\text{on }\partial D.
\end{equation}
Here, $u$ is identified with its trace on $\partial D$ in the DtN term.
Since $d\in\{2,3\}$, \cref{lem:kerr-local-estimate} implies that
$\mathcal N_\sigma:H^1(D)\to L^2(D)$ is well defined. Hence, the weak
formulation is: find $(\omega,u)\in \mathbb{C} \times H^1(D)$ such that
$a^{\mathrm{non}}_{\omega,\delta}(u;v)=0$ for every $v\in H^1(D)$, where
\begin{equation}
\label{eq:def-nonlinear-variational-form}
    a^{\mathrm{non}}_{\omega,\delta}(u;v)
    :=
    (\nabla u,\nabla v)_D
    -
    k_b^2\bigl(u+\mathcal N_\sigma[u],v\bigr)_D
    -
    \delta
    \langle
        \mathcal T_D^{\alpha,k_m}[u],
        v
    \rangle_{\partial D}.
\end{equation}
Here, $a^{\mathrm{non}}_{\omega,\delta}$ denotes the nonlinear variational form:
it is nonlinear in the first argument and conjugate-linear in the test
function. For fixed $u$, the map
$v\mapsto a^{\mathrm{non}}_{\omega,\delta}(u;v)$ is the corresponding weak
residual.
In the linear case $\sigma_j=0$ for all $j$, the associated sesquilinear form is
\begin{equation}
\label{eq:def-linear-variational-form}
    a^{\mathrm{lin}}_{\omega,\delta}(u,v)
    :=
    (\nabla u,\nabla v)_D
    -
    k_b^2(u,v)_D
    -
    \delta
    \langle
        \mathcal T_D^{\alpha,k_m}[u],
        v
    \rangle_{\partial D}.
\end{equation}

\subsection{Reduced equations}

Decompose $u=u_q+z$ according to \eqref{eq:XZ-decomposition}, with $u_q\in\mathcal X(D)$ being the lift of the componentwise
averages $q$, and $z\in\mathcal Z(D)$ being the zero-average correction. We
project the weak form \eqref{eq:def-nonlinear-variational-form} first onto $\mathcal Z(D)$ and then onto
$\mathcal X(D)$.

For the $\mathcal Z(D)$-projection, since $u_q$ is componentwise constant
and $\sigma_D$ is also componentwise constant on each $D_j$, we have
\[
    \mathcal N_\sigma[u_q]
    =
    \sigma_D|u_q|^2u_q
    \in \mathcal X(D).
\]
Thus, for every $v\in\mathcal Z(D)$,
\[
    (u_q,v)_D=0,
    \qquad
    (\mathcal N_\sigma[u_q],v)_D=0,
    \qquad
    (\nabla u_q,\nabla v)_D=0.
\]
Define
\begin{equation}
\label{eq:def-tilde-Nsigma}
    \widetilde{\mathcal N}_\sigma[q,z]
    :=
    \mathcal N_\sigma[u_q+z]-\mathcal N_\sigma[u_q],
\end{equation}
which isolates the part of the nonlinearity involving the zero-average
correction $z$. Substituting $u=u_q+z$ into
\eqref{eq:def-nonlinear-variational-form} and testing against
$v\in\mathcal Z(D)$, we obtain
\[
\begin{aligned}
0
=
a^{\mathrm{non}}_{\omega,\delta}(u_q+z;v)
=
a^{\mathrm{lin}}_{\omega,\delta}(z,v)
-
\delta
\langle
    \mathcal T_D^{\alpha,k_m}[u_q],
    v
\rangle_{\partial D}
-
k_b^2
\bigl(
    \widetilde{\mathcal N}_\sigma[q,z],
    v
\bigr)_D .
\end{aligned}
\]
Hence, the $\mathcal Z(D)$-projected equation is
\begin{equation}
\label{eq:projected-Z-before-scaling}
a^{\mathrm{lin}}_{\omega,\delta}(z,v)
=
k_b^2
\bigl(
    \widetilde{\mathcal N}_\sigma[q,z],
    v
\bigr)_D
+
\delta
\langle
    \mathcal T_D^{\alpha,k_m}[u_q],
    v
\rangle_{\partial D},
\qquad
v\in\mathcal Z(D).
\end{equation}

Next, we project onto $\mathcal X(D)$. For $1\le i\le N$, take
$v=\chi_{D_i}$, the
gradient term vanishes. Moreover, $z\in\mathcal Z(D)$ implies
that $(z,1)_{D_i}=0$. The weak formulation
\eqref{eq:def-nonlinear-variational-form} therefore gives
\begin{equation}
\label{eq:X-test-before-expansion}
    k_b^2
    \bigl(
        u_q+\mathcal N_\sigma[u_q+z],
        1
    \bigr)_{D_i}
    +
    \delta
    \langle
        \mathcal T_D^{\alpha,k_m}[u_q+z],
        1
    \rangle_{\partial D_i}
    =
    0 .
\end{equation}
Using $\mathcal N_\sigma[u_q+z]
    =
    \mathcal N_\sigma[u_q]
    +
    \widetilde{\mathcal N}_\sigma[q,z]$,
together with
\[
    (u_q,1)_{D_i}=|D_i|q_i=(Vq)_i,
    \qquad
    (\mathcal N_\sigma[u_q],1)_{D_i}
    =
    \sigma_i|D_i||q_i|^2q_i
    =
    (V_\sigma(q))_i,
\]
we find that
\[
    \bigl(
        u_q+\mathcal N_\sigma[u_q+z],
        1
    \bigr)_{D_i}
    =
    (Vq)_i
    +
    (V_\sigma(q))_i
    +
    \bigl(
        \widetilde{\mathcal N}_\sigma[q,z],
        1
    \bigr)_{D_i}.
\]
We expand the DtN contribution as follows. Since the trace of $u_q$ on
$\partial D$ is piecewise constant, the linearity of
$\mathcal T_D^{\alpha,k_m}$, the DtN expansion
\eqref{eq:DtN-operator-expansion} with $c=c_m$, and the capacitance identities
\eqref{eq:DtN-pairing-with-capacitance} imply that
\[
\begin{aligned}
    \langle
        \mathcal T_D^{\alpha,k_m}[u_q+z],
        1
    \rangle_{\partial D_i}
    =
    -(C^0q)_i
    -
    \omega(C^1q)_i
    +
    \omega^2
    \langle
        \mathcal R_D^{\alpha,k_m}[u_q],
        1
    \rangle_{\partial D_i}
    +
    \langle
        \mathcal T_D^{\alpha,k_m}[z],
        1
    \rangle_{\partial D_i}.
\end{aligned}
\]
Substituting this identity into \eqref{eq:X-test-before-expansion} gives
\begin{equation}
\label{eq:projected-X-finite-dimensional}
    k_b^2
    \bigl(Vq+V_\sigma(q)\bigr)
    -
    \delta
    \bigl(C^0+\omega C^1\bigr)q
    +
    \delta
    \mathcal R_{\mathcal X}(q,z,\omega,\delta)
    =
    0
    \quad \text{in }\mathbb C^N,
\end{equation}
where the remainder is defined componentwise by
\begin{equation}
\label{eq:projected-X-explicit-remainder}
\begin{aligned}
    \bigl(\mathcal R_{\mathcal X}(q,z,\omega,\delta)\bigr)_i
    :=
    \frac{k_b^2}{\delta}
    \bigl(
        \widetilde{\mathcal N}_\sigma[q,z],
        1
    \bigr)_{D_i}
    +
    \omega^2
    \langle
        \mathcal R_D^{\alpha,k_m}[u_q],
        1
    \rangle_{\partial D_i}
    +
    \langle
        \mathcal T_D^{\alpha,k_m}[z],
        1
    \rangle_{\partial D_i}.
\end{aligned}
\end{equation}
Since $H^1(D)=\mathcal X(D)\oplus\mathcal Z(D)$, the nonlinear resonance
problem \eqref{eq:interior-nonlinear-resonance-problem} is equivalent to the coupled system consisting of
\eqref{eq:projected-Z-before-scaling} and
\eqref{eq:projected-X-finite-dimensional}.

We now introduce the high-contrast scaling. Set
$\varepsilon:=\sqrt{\delta}$ and
$\hat\omega:=\omega/(c_b\varepsilon)$. Then
\begin{align}
\label{eq:low-frequency-scaling}
    \omega
    =
    \varepsilon c_b\hat\omega,
    \qquad
    \alpha
    =
    \omega a
    =
    \varepsilon c_b\hat\omega a,
    \qquad
    k_b
    =
    \varepsilon\hat\omega,
    \qquad
    k_m
    =
    \varepsilon \hat\omega c_b/c_m.
\end{align}
In what follows, all occurrences of $\alpha$ and $k_m$ in
$\mathcal T_D^{\alpha,k_m}$ are understood through the scaling
\eqref{eq:low-frequency-scaling}.
We define the scaled linear and nonlinear variational forms by
\begin{align}
\label{eq:scaled-weak-form}
    a^{\mathrm{lin}}_{\hat\omega,\varepsilon}(u,v)
    &:=
    a^{\mathrm{lin}}_{\omega,\delta}(u,v)
    \big|_{\omega=\varepsilon c_b\hat\omega,\ \delta=\varepsilon^2},
    \qquad
    a^{\mathrm{non}}_{\hat\omega,\varepsilon}(u;v)
    :=
    a^{\mathrm{non}}_{\omega,\delta}(u;v)
    \big|_{\omega=\varepsilon c_b\hat\omega,\ \delta=\varepsilon^2}.
\end{align}
We also introduce the two functionals on $\mathcal Z(D)$ that enter the
$\mathcal Z(D)$-projection:
\begin{equation}
\label{eq:def-z-projection-functionals}
\begin{aligned}
    \mathcal B_{q,\hat\omega,\varepsilon}(v)
    :=
    \left\langle
        \mathcal T_D^{\alpha,k_m}[u_q],
        v
    \right\rangle_{\partial D},
    \qquad
    \mathcal H_q[z](v)
    :=
    \bigl(
        \widetilde{\mathcal N}_\sigma[q,z],
        v
    \bigr)_D,
    \qquad
    v\in\mathcal Z(D).
\end{aligned}
\end{equation}
Since $k_b^2=\varepsilon^2\hat\omega^2$, the scaled
$\mathcal Z(D)$-equation \eqref{eq:projected-Z-before-scaling} becomes
\begin{equation}
\label{eq:projected-Z-nonlinear-form}
    a^{\mathrm{lin}}_{\hat\omega,\varepsilon}(z,v)
    =
    \varepsilon^2\mathcal B_{q,\hat\omega,\varepsilon}(v)
    +
    \varepsilon^2\hat\omega^2\mathcal H_q[z](v),
    \qquad
    v\in\mathcal Z(D).
\end{equation}
Dividing \eqref{eq:projected-X-finite-dimensional} by
$\delta=\varepsilon^2$ gives
\begin{equation}
\label{eq:projected-X-finite-dimensional-system}
    \hat\omega^2
    \bigl(Vq+V_\sigma(q)\bigr)
    -
    \bigl(C^0+c_b\varepsilon\hat\omega C^1\bigr)q
    +
    \mathcal R_{\mathcal X}(q,z,\hat\omega,\varepsilon)
    =
    0
    \quad \text{in }\mathbb C^N.
\end{equation}
Here, after scaling, the remainder is defined componentwise by
\begin{equation}
\label{eq:scaled-X-remainder}
\begin{aligned}
    \bigl(\mathcal R_{\mathcal X}(q,z,\hat\omega,\varepsilon)\bigr)_i
    :=
    \hat\omega^2
    \bigl(
        \widetilde{\mathcal N}_\sigma[q,z],
        1
    \bigr)_{D_i}
    +
    \varepsilon^2c_b^2\hat\omega^2
    \langle
        \mathcal R_D^{\alpha,k_m}[u_q],
        1
    \rangle_{\partial D_i}
    +
    \langle
        \mathcal T_D^{\alpha,k_m}[z],
        1
    \rangle_{\partial D_i}.
\end{aligned}
\end{equation}

\begin{remark}\label{rem:lyapunov-schmidt-strategy}
Equivalently, after the scaling $\delta=\varepsilon^2$ and
$\omega=c_b\varepsilon\hat\omega$, finding a nontrivial interior resonance
$(\omega,u)$ of \eqref{eq:interior-nonlinear-resonance-problem} for fixed $\delta$ is the same as
finding $(\hat\omega,q,z)$ for fixed $\varepsilon$, with $u=u_q+z$, satisfying
\eqref{eq:projected-Z-nonlinear-form} and
\eqref{eq:projected-X-finite-dimensional-system}. The subsequent local analysis follows a Lyapunov--Schmidt reduction near a
simple mode:
\begin{enumerate}[label=(\roman*)]
    \item For fixed $(q,\hat\omega,\varepsilon)$, solve the
    $\mathcal Z(D)$-equation \eqref{eq:projected-Z-nonlinear-form} for the
    zero-average  $z=z(q,\hat\omega,\varepsilon)$;

    \item Substitute $z(q,\hat\omega,\varepsilon)$ into
    \eqref{eq:projected-X-finite-dimensional-system} to obtain a
    finite-dimensional resonance equation for $(\hat\omega,q)$;

    \item Impose the appropriate normalization and solve the finite-dimensional
    equation locally in modal coordinates, with the relevant small parameters;

    \item Reconstruct the interior field by combining the piecewise-constant
    part determined by $q$ with the zero-average correction.
\end{enumerate}

\end{remark}

Both linear and nonlinear reductions continue from a simple mode
$(\hat\omega_j^0,p_j)$, where $\hat\omega_j^0=\sqrt{\lambda_j}$. We therefore
fix the modal complement used below.

\begin{definition}[$V$-orthogonal modal complement]
\label{def:modal-complement}
Let $(\lambda_j,p_j)$ be a simple positive generalized eigenpair from
\cref{lem:spectral-capacitance}, normalized by $p_j^\top Vp_j=1$. We define the
$V$-orthogonal modal complement of $p_j$, with respect to the bilinear
$V$-pairing, by
\begin{equation}
\label{eq:def-modal-complement-space}
    E_j^\perp
    :=
    \left\{
        \xi\in\mathbb C^N:
        p_j^\top V\xi=0
    \right\}.
\end{equation}
Equivalently, every $q\in\mathbb C^N$ admits the unique decomposition
\[
    q
    =
    (p_j^\top Vq)p_j
    +
    \xi,
    \qquad
    \xi\in E_j^\perp .
\]
For vector $q\in\mathbb C^N$, we also define the modal projections
\[
    \mathcal P_j^\parallel q
    :=
    (p_j^\top q)Vp_j,
    \qquad
    \mathcal P_j^\perp q
    :=
    q-\mathcal P_j^\parallel q .
\]
\end{definition}

\subsection{Linear subwavelength resonances}

We consider the linear medium under
\cref{ass:subwavelength-low-frequency}. Following the Lyapunov--Schmidt
strategy described in \cref{rem:lyapunov-schmidt-strategy}, we first solve the
$\mathcal Z(D)$-projected equation for the zero-average correction
$z^{\rm lin}=z^{\rm lin}(q,\hat\omega,\varepsilon)$ and then substitute this
correction into the $\mathcal X(D)$-equation. This yields a finite-dimensional
reduced equation for $(\hat\omega,q)$, from which we construct a local branch
near a simple linear mode and compute its first coefficients.

\begin{proposition}[Projected reduction of the linear resonance problem]
\label{prop:linear-projected-reduction}
Fix $M_{\hat\omega}>0$, and choose
$\varepsilon_0=\varepsilon_0(M_{\hat\omega})>0$ sufficiently small that
$c_b\varepsilon_0M_{\hat\omega}<\omega_0$. Then, for every
$q\in\mathbb C^N$, $|\hat\omega|\le M_{\hat\omega}$, and
$0<\varepsilon<\varepsilon_0$, there exists a unique
$z^{\rm lin}=z^{\rm lin}(q,\hat\omega,\varepsilon)\in\mathcal Z(D)$ such that
\begin{equation}
\label{eq:linear-Z-equation-for-z}
    a^{\mathrm{lin}}_{\hat\omega,\varepsilon}(z^{\rm lin},v)
    =
    \varepsilon^2\mathcal B_{q,\hat\omega,\varepsilon}(v),
    \qquad v\in\mathcal Z(D).
\end{equation}
Here, the scaling \eqref{eq:low-frequency-scaling} is understood. This solution satisfies
\begin{equation}
\label{eq:linear-basic-z-estimate}
    \|z^{\rm lin}(q,\hat\omega,\varepsilon)\|_{H^1(D)}
    \le C\varepsilon^2\|q\|.
\end{equation}
It admits the expansion
\begin{equation}
\label{eq:linear-z-expansion-epsilon}
    z^{\rm lin}(q,\hat\omega,\varepsilon)
    =
    \varepsilon^2z_q^0
    +
    c_b\varepsilon^3\hat\omega z_q^1
    +
    \mathfrak r_z^{\rm lin}(q,\hat\omega,\varepsilon)
    \qquad\text{in }H^1(D),
\end{equation}
where $z_q^0,z_q^1\in\mathcal Z(D)$ are uniquely determined by
\begin{equation}
\label{eq:linear-corrector-equation}
    (\nabla z_q^n,\nabla v)_D
    =
    \left\langle
        \mathcal T_n^{a,c_m}[u_q],v
    \right\rangle_{\partial D},
    \qquad v\in\mathcal Z(D),\qquad n=0,1.
\end{equation}
The associated estimates are
\begin{align}
\label{eq:linear-corrector-estimate}
    \|z_q^n\|_{H^1(D)}
    &\le C\|q\|,
    \qquad n=0,1,
    \\
\label{eq:linear-z-remainder-estimate}
    \|\mathfrak r_z^{\rm lin}(q,\hat\omega,\varepsilon)\|_{H^1(D)}
    &\le
    C\varepsilon^4(1+|\hat\omega|^2)\|q\|.
\end{align}
Substituting this $z^{\rm lin}$ into the $\mathcal X(D)$-projection gives
the residual equation $\mathcal F^{\rm lin}(q,\hat\omega,\varepsilon)=0$,
where the linear reduced residual is defined by
\begin{equation}
\label{eq:reduced-linear-resonance-equation}
    \mathcal F^{\rm lin}(q,\hat\omega,\varepsilon)
    :=
    \hat\omega^2Vq
    -
    \bigl(C^0+c_b\varepsilon\hat\omega C^1\bigr)q
    +
    R^{\rm lin}(q,\hat\omega,\varepsilon).
\end{equation}
Here, $R^{\rm lin}(q,\hat\omega,\varepsilon)$ is defined componentwise by
\begin{equation}
\label{eq:def-linear-R}
\begin{aligned}
    \bigl(R^{\rm lin}(q,\hat\omega,\varepsilon)\bigr)_i
    :=
    c_b^2\varepsilon^2\hat\omega^2
    \langle
        \mathcal R_D^{\alpha,k_m}[u_q],1
    \rangle_{\partial D_i}
    +
    \langle
        \mathcal T_D^{\alpha,k_m}
        [z^{\rm lin}(q,\hat\omega,\varepsilon)],1
    \rangle_{\partial D_i}.
\end{aligned}
\end{equation}
The reduced remainder satisfies
\begin{equation}
\label{eq:linear-R-estimate}
    \|R^{\rm lin}(q,\hat\omega,\varepsilon)\|
    \le
    C\varepsilon^2(1+|\hat\omega|^2)\|q\|.
\end{equation}
The maps $z^{\rm lin}$, $R^{\rm lin}$, and $\mathcal F^{\rm lin}$ are linear
in $q$ for fixed $(\hat\omega,\varepsilon)$. Moreover, after extending the
scaled problem to $\varepsilon=0$ by the limiting gradient form on
$\mathcal Z(D)$, these maps extend smoothly to $\varepsilon=0$. All constants are
uniform for $|\hat\omega|\le M_{\hat\omega}$,
$0<\varepsilon<\varepsilon_0$, and are independent of $q$, $\hat\omega$, and
$\varepsilon$.
\end{proposition}

\begin{proof}
We first solve the $\mathcal Z(D)$-equation. Since the medium is linear, the
nonlinear terms in
\eqref{eq:projected-Z-nonlinear-form} vanish. Hence, the linear
$\mathcal Z(D)$-projected equation is exactly
\eqref{eq:linear-Z-equation-for-z}. Its right-hand side is
$F_q:=\varepsilon^2\mathcal B_{q,\hat\omega,\varepsilon}\in\mathcal Z(D)'$.
By the uniform boundedness of $\mathcal T_D^{\alpha,k_m}$, the trace theorem,
and the estimate $\|u_q\|_{H^1(D)}=\|u_q\|_{L^2(D)}\le C\|q\|$, we have
\[
\begin{aligned}
    |F_q(v)|
    \le
    C\varepsilon^2
    \|u_q\|_{H^{1/2}(\partial D)}
    \|v\|_{H^{1/2}(\partial D)}
    \le
    C\varepsilon^2\|q\|\|v\|_{H^1(D)},
    \qquad
    v\in\mathcal Z(D).
\end{aligned}
\]
Applying \cref{lem:linear-Z-lax-milgram}, after decreasing $\varepsilon_0$ if
necessary, gives a unique solution
$z^{\rm lin}=z^{\rm lin}(q,\hat\omega,\varepsilon)$ and the estimate
\eqref{eq:linear-basic-z-estimate}. For fixed $(\hat\omega,\varepsilon)$, the
operator on the left-hand side of \eqref{eq:linear-Z-equation-for-z} is fixed,
while the right-hand side is linear in $u_q$, hence in $q$. By uniqueness,
$q\mapsto z^{\rm lin}(q,\hat\omega,\varepsilon)$ is linear.

Next, we derive the expansion of $z^{\rm lin}$. We use the componentwise
Poincar\'e inequality
(see, e.g. \cite[Chapter~5, Section~5.8.1]{evans2010partial}), which yields
$\|v\|_{H^1(D)}\le C\|\nabla v\|_{L^2(D)}$ for
$v\in\mathcal Z(D)$. Hence, the gradient form is coercive on
$\mathcal Z(D)$, so \eqref{eq:linear-corrector-equation} uniquely determines
$z_q^0,z_q^1$ and gives \eqref{eq:linear-corrector-estimate}. Set
\[
    z_{\rm app}
    :=
    \varepsilon^2 z_q^0
    +
    c_b\varepsilon^3\hat\omega z_q^1,
    \qquad
    \mathfrak r_z^{\rm lin}
    :=
    z^{\rm lin}-z_{\rm app}.
\]
Using the DtN expansion and subtracting the equation satisfied by
$z_{\rm app}$ from \eqref{eq:linear-Z-equation-for-z}, we obtain
\[
\begin{aligned}
    a^{\mathrm{lin}}_{\hat\omega,\varepsilon}
    (\mathfrak r_z^{\rm lin},v)
    =
    \varepsilon^2\hat\omega^2(z_{\rm app},v)_D
    +
    \varepsilon^2
    \langle
        \mathcal T_D^{\alpha,k_m}[z_{\rm app}],v
    \rangle_{\partial D}
    +
    c_b^2\varepsilon^4\hat\omega^2
    \langle
        \mathcal R_D^{\alpha,k_m}[u_q],v
    \rangle_{\partial D}.
\end{aligned}
\]
Since $\|z_{\rm app}\|_{H^1(D)}
\le C\varepsilon^2(1+\varepsilon|\hat\omega|)\|q\|$, the right-hand side is
bounded by
$C\varepsilon^4(1+|\hat\omega|^2)\|q\|\|v\|_{H^1(D)}$.
A second application of \cref{lem:linear-Z-lax-milgram} proves
\eqref{eq:linear-z-remainder-estimate}.

It remains to identify the finite-dimensional residual. Since $V_\sigma=0$ and
$\widetilde{\mathcal N}_\sigma=0$ in the linear case, substituting
$z^{\rm lin}=z^{\rm lin}(q,\hat\omega,\varepsilon)$ into
\eqref{eq:projected-X-finite-dimensional-system} gives the residual equation
$\mathcal F^{\rm lin}(q,\hat\omega,\varepsilon)=0$, with
$\mathcal F^{\rm lin}$ defined by
\eqref{eq:reduced-linear-resonance-equation} and $R^{\rm lin}$ by
\eqref{eq:def-linear-R}. The two terms in $R^{\rm lin}$ are estimated
separately. The first term in
\eqref{eq:def-linear-R} is bounded
by $C\varepsilon^2|\hat\omega|^2\|q\|$. The trace theorem, the uniform
boundedness of the DtN operator, and \eqref{eq:linear-basic-z-estimate} show
that the second term is bounded by $C\varepsilon^2\|q\|$. This proves
\eqref{eq:linear-R-estimate}.

Finally, we note the linearity and smooth dependence. For fixed
$(\hat\omega,\varepsilon)$, \eqref{eq:def-linear-R} shows that
$q\mapsto R^{\rm lin}(q,\hat\omega,\varepsilon)$ is linear, because
$q\mapsto u_q$ and $q\mapsto z^{\rm lin}(q,\hat\omega,\varepsilon)$ are
linear and the DtN operators are linear. Then the linearity of
$q\mapsto \mathcal F^{\rm lin}(q,\hat\omega,\varepsilon)$ follows from
\eqref{eq:reduced-linear-resonance-equation}. The smooth dependence of
$z^{\rm lin}$ on $(q,\hat\omega,\varepsilon)$ follows from the smooth
dependence of $a^{\rm lin}_{\hat\omega,\varepsilon}$ and
$\mathcal T_D^{\alpha,k_m}$ in the low-frequency regime. At
$\varepsilon=0$, the limiting form on $\mathcal Z(D)$ is
$(\nabla u,\nabla v)_D$, which is coercive by the componentwise
Poincar\'e inequality. Hence, the extended operator family remains uniformly
invertible for small $\varepsilon$, and its inverse depends smoothly on
$(\hat\omega,\varepsilon)$. The corresponding smooth extension of
$R^{\rm lin}$ follows from \eqref{eq:def-linear-R} and the smooth
low-frequency remainder in the DtN expansion. The smooth extension of
$\mathcal F^{\rm lin}$ follows then from
\eqref{eq:reduced-linear-resonance-equation}.
\end{proof}

We now solve the projected linear residual near a simple positive capacitance
mode and compute the first coefficients of the branch.

\begin{theorem}[Asymptotics of a simple linear resonance branch]
\label{thm:linear-simple-resonance-branch}
Let $(\lambda_j,p_j)$ be a simple positive generalized eigenpair from
\cref{lem:spectral-capacitance}. Set
$\hat\omega_j^0:=\sqrt{\lambda_j}$. Fix
$M_{\hat\omega}>\hat\omega_j^0$, and work in the validity regime of
\cref{prop:linear-projected-reduction}. Then there exist $\gamma_j>0$,
$\varepsilon_j>0$, and a smooth branch of normalized solutions
$(\hat\omega_j^{\rm lin}(\varepsilon),q_j^{\rm lin}(\varepsilon))$ such that,
for $0\le\varepsilon<\varepsilon_j$, each pair solves the residual equation
$\mathcal F^{\rm lin}(q,\hat\omega,\varepsilon)=0$ from
\cref{prop:linear-projected-reduction} and satisfies
\[
    |\hat\omega_j^{\rm lin}(\varepsilon)-\hat\omega_j^0|<\gamma_j,
    \qquad
    \|q_j^{\rm lin}(\varepsilon)-p_j\|<\gamma_j,
    \qquad
    p_j^\top Vq_j^{\rm lin}(\varepsilon)=1.
\]
This branch is locally unique among solutions satisfying the same normalization
and the two inequalities in the display. The constants $\gamma_j$ and
$\varepsilon_j$ are chosen so that the branch lies in the validity regime of
\cref{prop:linear-projected-reduction}. Define
\begin{align}
\label{eq:linear-radiative-coefficient}
    \hat\omega_j^1
    &:=
    c_b
    \left[
        \tau_m|Y|\,(s^0\cdot p_j)^2
        +
        \frac{(m_d^{\tau_m}\cdot p_j)^2}{4\tau_m|Y|}
    \right],
    \\
\label{eq:linear-vector-first-correction}
    q_j^0
    &:=
    p_j,
    \qquad
    q_j^1
    :=
    \mathrm i\,c_b\sqrt{\lambda_j}
    \sum_{i\ne j}
    \frac{p_i^\top C^1p_j}{\lambda_j-\lambda_i}\,p_i .
\end{align}
Let $U_j^0$ and $U_j^1$ be the corresponding piecewise constant lifts,
\begin{equation}
\label{eq:linear-leading-field-lifts}
    U_j^0
    :=
    \sum_{i=1}^N(q_j^0)_i\chi_{D_i},
    \qquad
    U_j^1
    :=
    \sum_{i=1}^N(q_j^1)_i\chi_{D_i}.
\end{equation}
Then, as $\varepsilon\to0$, we have
\begin{equation}
\label{eq:linear-scaled-branch-expansion}
    \hat\omega_j^{\rm lin}(\varepsilon)
    =
    \hat\omega_j^0
    -
    \mathrm i\,\hat\omega_j^1\varepsilon
    +
    \mathcal O(\varepsilon^2),
    \qquad
    q_j^{\rm lin}(\varepsilon)
    =
    q_j^0
    -
    \mathrm i\,\varepsilon q_j^1
    +
    \mathcal O(\varepsilon^2).
\end{equation}
Equivalently, with $\delta=\varepsilon^2$, the physical resonance frequency
satisfies
\begin{equation}
\label{eq:linear-resonance-expansion}
    \omega_j^{\rm lin}(\delta)
    =
    c_b\hat\omega_j^0\sqrt\delta
    -
    \mathrm i\,c_b\hat\omega_j^1\,\delta
    +
    \mathcal O(\delta^{3/2}).
\end{equation}
The reconstructed interior field satisfies
\begin{equation}
\label{eq:linear-physical-field-expansion}
    u_j^{\rm lin}(\delta)
    =
    U_j^0
    -
    \mathrm i\,\sqrt\delta\,U_j^1
    +
    \mathcal O(\delta).
\end{equation}
\end{theorem}

\begin{proof}
By \cref{prop:linear-projected-reduction}, the maps $R^{\rm lin}$ and
$\mathcal F^{\rm lin}$ extend smoothly to $\varepsilon=0$, and
$R^{\rm lin}$ is linear in $q$ for fixed $(\hat\omega,\varepsilon)$. Together
with \eqref{eq:linear-R-estimate}, this gives
\[
    R^{\rm lin}(q,\hat\omega,0)=0,
    \qquad
    \partial_\varepsilon R^{\rm lin}(q,\hat\omega,0)=0.
\]
Taylor's formula with integral remainder therefore yields
\begin{equation}
\label{eq:linear-R-factorization}
    R^{\rm lin}(q,\hat\omega,\varepsilon)
    =
    \varepsilon^2
    \widetilde R^{\rm lin}(q,\hat\omega,\varepsilon),
\end{equation}
where $\widetilde R^{\rm lin}$ is smooth and linear in $q$. The normalization
$p_j^\top Vq=1$ allows us to write uniquely $q=p_j+\xi$, where
$\xi\in E_j^\perp$; see \eqref{eq:def-modal-complement-space}. Define
\[
    \mathcal F_j^{\rm lin}(\xi,\hat\omega,\varepsilon)
    :=
    \mathcal F^{\rm lin}(p_j+\xi,\hat\omega,\varepsilon).
\]
Then $\mathcal F_j^{\rm lin}(0,\hat\omega_j^0,0)=0$. We apply the finite-dimensional implicit function theorem to
$\mathcal F_j^{\rm lin}=0$ in the variables $(\xi,\hat\omega)$, viewing all
complex spaces as real finite-dimensional spaces. By
\eqref{eq:linear-R-factorization}, the derivative at
$(0,\hat\omega_j^0,0)$ is
\[
\begin{aligned}
    L_j[\widetilde\xi,\widetilde{\hat\omega}]
    :=
    D_{(\xi,\hat\omega)}
    \mathcal F_j^{\rm lin}(0,\hat\omega_j^0,0)
    [\widetilde\xi,\widetilde{\hat\omega}]
    =
    (\lambda_jV-C^0)\widetilde\xi
    +
    2\hat\omega_j^0\widetilde{\hat\omega} Vp_j.
\end{aligned}
\]
This map is an isomorphism from $E_j^\perp\times\mathbb C$ onto
$\mathbb C^N$. Indeed, if $\widetilde\xi=\sum_{i\ne j}c_ip_i$, then the modal
projections give
\[
\begin{aligned}
    \mathcal P_j^\parallel L_j[\widetilde\xi,\widetilde{\hat\omega}]
     =
    2\hat\omega_j^0\widetilde{\hat\omega} Vp_j,
    \qquad
    \mathcal P_j^\perp L_j[\widetilde\xi,\widetilde{\hat\omega}]
     =
    \sum_{i\ne j}(\lambda_j-\lambda_i)c_iVp_i .
\end{aligned}
\]
Since $\hat\omega_j^0>0$ and $\lambda_j$ is simple, the kernel is trivial; the
domain and the codomain have the same real dimension, so $L_j$ is invertible. The
implicit function theorem gives unique smooth functions $\xi_j(\varepsilon)$
and $\hat\omega_j^{\rm lin}(\varepsilon)$ for $|\varepsilon|$ sufficiently
small, with $\xi_j(0)=0$ and
$\hat\omega_j^{\rm lin}(0)=\hat\omega_j^0$. Setting
$q_j^{\rm lin}(\varepsilon):=p_j+\xi_j(\varepsilon)$ gives the normalized
branch. Restricting to $\varepsilon\ge0$ and decreasing $\gamma_j$ and
$\varepsilon_j$ if necessary keeps the branch in the validity regime of
\cref{prop:linear-projected-reduction} and gives the stated local uniqueness.

Next, we compute the first derivatives at $\varepsilon=0$. Set
\[
    \omega_1
    :=
    \frac{d}{d\varepsilon}
    \hat\omega_j^{\rm lin}(0),
    \qquad
    \xi_1
    :=
    \frac{d}{d\varepsilon}
    \xi_j(0).
\]
Along the branch, \eqref{eq:linear-R-factorization} implies that the total
$\varepsilon$-derivative of $R^{\rm lin}$ at $\varepsilon=0$ vanishes.
Differentiating the residual identity $\mathcal F^{\rm lin}=0$ and using
\eqref{eq:reduced-linear-resonance-equation} therefore give
\begin{equation}
\label{eq:linear-first-order-coefficient-equation}
    (\lambda_jV-C^0)\xi_1
    +
    2\hat\omega_j^0\omega_1Vp_j
    -
    c_b\hat\omega_j^0C^1p_j
    =
    0.
\end{equation}
Since $\xi_j(\varepsilon)\in E_j^\perp$, one has
$p_j^\top V\xi_1=0$. Multiplying
\eqref{eq:linear-first-order-coefficient-equation} from the left by
$p_j^\top$ gives $2\hat\omega_j^0\omega_1
-c_b\hat\omega_j^0p_j^\top C^1p_j=0$, and therefore,
$\omega_1=(c_b/2)p_j^\top C^1p_j$. By \eqref{eq:C1-formula}, the real
skew-symmetric part of $C^1$ has zero quadratic form against the real vector
$p_j$. Therefore
\[
    p_j^\top C^1p_j
    =
    -\mathrm i
    \left(
        2\tau_m|Y|\,(s^0\cdot p_j)^2
        +
        \frac{(m_d^{\tau_m}\cdot p_j)^2}{2\tau_m|Y|}
    \right),
\]
which, together with \eqref{eq:linear-radiative-coefficient}, yields
$\omega_1=-\mathrm i\,\hat\omega_j^1$. Next, we write
$\xi_1=\sum_{i\ne j}c_ip_i$. Multiplying
\eqref{eq:linear-first-order-coefficient-equation} from the left by
$p_i^\top$, $i\ne j$, gives
$(\lambda_j-\lambda_i)c_i=c_b\hat\omega_j^0p_i^\top C^1p_j$. Since
$\hat\omega_j^0=\sqrt{\lambda_j}$, comparison with
\eqref{eq:linear-vector-first-correction} gives $\xi_1=-\mathrm i\,q_j^1$. Taylor's
formula now proves \eqref{eq:linear-scaled-branch-expansion}.

Finally, we pass from the reduced coefficients to the physical frequency and
field. Since $\omega=c_b\varepsilon\hat\omega$ and $\delta=\varepsilon^2$,
\eqref{eq:linear-scaled-branch-expansion} gives
\eqref{eq:linear-resonance-expansion}. The bounded linear lift
$q\mapsto u_q$ from $\mathbb C^N$ to $H^1(D)$ gives
\[
    u_{q_j^{\rm lin}(\varepsilon)}
    =
    U_j^0
    -
    \mathrm i\,\varepsilon U_j^1
    +
    \mathcal O(\varepsilon^2).
\]
Moreover, \eqref{eq:linear-basic-z-estimate} and the boundedness of the
normalized branch give
\[
    z^{\rm lin}
    \bigl(
        q_j^{\rm lin}(\varepsilon),
        \hat\omega_j^{\rm lin}(\varepsilon),
        \varepsilon
    \bigr)
    =
    \mathcal O(\varepsilon^2).
\]
Since $u_j^{\rm lin}=u_{q_j^{\rm lin}}+z^{\rm lin}$, adding the piecewise
constant and zero-average contributions proves
\begin{equation*}
\label{eq:linear-scaled-field-expansion}
    u_j^{\rm lin}(\varepsilon)
    =
    U_j^0
    -
    \mathrm i\,\varepsilon U_j^1
    +
    \mathcal O(\varepsilon^2).
\end{equation*}
Replacing $\varepsilon$ by
$\sqrt\delta$ then gives \eqref{eq:linear-physical-field-expansion}.
\end{proof}

\begin{remark}[The zero capacitance mode]
\label{rem:zero-capacitance-mode}
The zero eigenvalue in \cref{lem:spectral-capacitance} is excluded from
\cref{thm:linear-simple-resonance-branch}. Set
$p_0:=\mathbf 1/\sqrt{|D|}$, where $|D|:=\sum_{i=1}^N|D_i|$; then
$C^0p_0=0$ and $p_0^\top Vp_0=1$. Since $\hat\omega_0^0=0$, the frequency
linearization used above degenerates. Instead, one uses
\[
    \hat\omega=\varepsilon\eta,
    \qquad
    q=p_0+\xi,
    \qquad
    p_0^\top V\xi=0.
\]
For $q=p_0$, the static exterior trace is constant and therefore has zero normal
derivative. The leading scalar projection of the reduced equation is therefore
$\eta^2
    -
    c_b(p_0^\top C^1p_0)\eta
    =
    0$.
The root $\eta=0$ is static. For the nonstatic root, \eqref{eq:C1-formula} and
\cref{lem:C1-decomposition} give
\[
\begin{aligned}
    p_0^\top C^1p_0
    =
    -\mathrm i\,\frac{2\tau_m|Y|}{|D|},\quad
    \hat\omega_0(\varepsilon)
    =
    -\mathrm i\,\frac{2c_b\tau_m|Y|}{|D|}\,\varepsilon
    +
    \mathcal O(\varepsilon^2),\quad
    \omega_0(\delta)
    =
    -\mathrm i\,\frac{2c_b^2\tau_m|Y|}{|D|}\,\delta
    +
    \mathcal O(\delta^{3/2}).
\end{aligned}
\]
Thus, the zero-mode physical frequency is of order $\delta$, rather than
$\sqrt\delta$.
\end{remark}

\subsection{Nonlinear subwavelength resonances}

We now turn to the nonlinear Kerr medium \eqref{eq:def-Nsigma}. The reduction
follows the same Lyapunov--Schmidt framework as in the linear subsection: under
the scaling \eqref{eq:low-frequency-scaling} and
\cref{ass:subwavelength-low-frequency}, we first solve the
$\mathcal Z(D)$-projected equation for a zero-average correction
$z^{\rm non}=z^{\rm non}(q,\hat\omega,\varepsilon)$, and then substitute this
correction into the $\mathcal X(D)$-projection to obtain a finite-dimensional
residual equation.

The leading linear operator and the first correctors are the same as in the
linear reduction. The new points are that the $\mathcal Z(D)$-equation contains
a $z$-dependent Kerr term, so it is solved by a contraction argument rather
than directly by Lax--Milgram, and that the finite-dimensional residual gains
the cubic modal term $V_\sigma(q)$. Since the Kerr map contains complex
conjugation, all complex Banach spaces are regarded as real Banach spaces when
applying implicit-function arguments.

\begin{proposition}[Projected reduction with piecewise constant nonlinearity]
\label{prop:nonlinear-projected-reduction}
Fix $M_q,M_{\hat\omega}>0$, and choose
$\varepsilon_0=\varepsilon_0(M_q,M_{\hat\omega})>0$ sufficiently small that
$c_b\varepsilon_0M_{\hat\omega}<\omega_0$. Then there exist $K,C>0$ such
that, whenever $\|q\|\le M_q$, $|\hat\omega|\le M_{\hat\omega}$, and
$0<\varepsilon<\varepsilon_0$, the $\mathcal Z(D)$-projected nonlinear equation
\eqref{eq:projected-Z-nonlinear-form} has a unique solution in the ball
\[
    B_{q,\varepsilon}
    :=
    \left\{
        z\in\mathcal Z(D):
        \|z\|_{H^1(D)}
        \le
        K\varepsilon^2\|q\|
    \right\}.
\]
We denote this solution by
$z^{\rm non}=z^{\rm non}(q,\hat\omega,\varepsilon)$. In particular,
\begin{equation}
\label{eq:nonlinear-basic-z-estimate}
    \|z^{\rm non}(q,\hat\omega,\varepsilon)\|_{H^1(D)}
    \le
    C\varepsilon^2\|q\|.
\end{equation}
After setting $z^{\rm non}(q,\hat\omega,0)=0$, the solution extends
real-analytically to $\varepsilon=0$, where the complex variables are regarded
as real variables. It also admits the expansion
\begin{equation}
\label{eq:nonlinear-z-expansion}
    z^{\rm non}(q,\hat\omega,\varepsilon)
    =
    \varepsilon^2 z_q^0
    +
    c_b\varepsilon^3\hat\omega z_q^1
    +
    \mathfrak r_z^{\rm non}(q,\hat\omega,\varepsilon)
    \qquad
    \text{in }H^1(D),
\end{equation}
where $z_q^0,z_q^1$ are the linear correctors defined by
\eqref{eq:linear-corrector-equation}. The nonlinear $z$-remainder satisfies
\begin{equation}
\label{eq:nonlinear-z-remainder-estimate}
    \|\mathfrak r_z^{\rm non}(q,\hat\omega,\varepsilon)\|_{H^1(D)}
    \le
    C\varepsilon^4(1+|\hat\omega|^2)(\|q\|+\|q\|^3).
\end{equation}
Substituting this $z^{\rm non}$ into the $\mathcal X(D)$-projection gives the
residual equation $\mathcal F^{\rm non}(q,\hat\omega,\varepsilon)=0$, where the
nonlinear reduced residual is defined by
\begin{equation}
\label{eq:nonlinear-reduced-equation}
    \mathcal F^{\rm non}(q,\hat\omega,\varepsilon)
    :=
    \hat\omega^2
    \bigl(Vq+V_\sigma(q)\bigr)
    -
    \bigl(C^0+c_b\varepsilon\hat\omega C^1\bigr)q
    +
    R^{\rm non}(q,\hat\omega,\varepsilon).
\end{equation}
Here, $R^{\rm non}$ is defined componentwise by
\begin{equation}
\label{eq:def-nonlinear-R}
\begin{aligned}
    \bigl(R^{\rm non}(q,\hat\omega,\varepsilon)\bigr)_i
    :=
    &\hat\omega^2
    \bigl(
        \widetilde{\mathcal N}_\sigma
        [q,z^{\rm non}],
        1
    \bigr)_{D_i}
    +
    c_b^2\varepsilon^2\hat\omega^2
    \langle
        \mathcal R_D^{\alpha,k_m}[u_q],
        1
    \rangle_{\partial D_i}
    +
    \langle
        \mathcal T_D^{\alpha,k_m}
        [z^{\rm non}],
        1
    \rangle_{\partial D_i}.
\end{aligned}
\end{equation}
The reduced remainder satisfies
\begin{equation}
\label{eq:nonlinear-R-estimate}
    \|R^{\rm non}(q,\hat\omega,\varepsilon)\|
    \le
    C\varepsilon^2(1+|\hat\omega|^2)\|q\|
    +
    C\varepsilon^4|\hat\omega|^2\|q\|^3.
\end{equation}
The maps $R^{\rm non}$ and $\mathcal F^{\rm non}$ extend real-smoothly to
$\varepsilon=0$. Moreover,
\[
    R^{\rm non}(0,\hat\omega,\varepsilon)=0,
    \qquad
    R^{\rm non}(q,\hat\omega,0)=0,
    \qquad
    \partial_\varepsilon R^{\rm non}(q,\hat\omega,0)=0.
\]
All constants may depend on $M_q$, $M_{\hat\omega}$, and the fixed geometric
and material parameters, but are independent of $q$, $\hat\omega$, and
$\varepsilon$ in the above range.
\end{proposition}

\begin{proof}
We first solve the $\mathcal Z(D)$-equation by a fixed-point argument. This is
the standard contraction method used in nonlinear PDE existence theory; see,
e.g., \cite[Chapter~9, Section~9.2.1]{evans2010partial}. Define
$\mathcal A_{\hat\omega,\varepsilon}:\mathcal Z(D)\to\mathcal Z(D)'$ by
\[
    \left\langle
        \mathcal A_{\hat\omega,\varepsilon}w,
        v
    \right\rangle
    :=
    a^{\mathrm{lin}}_{\hat\omega,\varepsilon}(w,v).
\]
By \cref{lem:linear-Z-lax-milgram}, after decreasing $\varepsilon_0$ if
necessary, $\mathcal A_{\hat\omega,\varepsilon}$ is invertible and
\[
    \|\mathcal A_{\hat\omega,\varepsilon}^{-1}F\|_{H^1(D)}
    \le
    C\|F\|_{\mathcal Z(D)'}
\]
uniformly for $|\hat\omega|\le M_{\hat\omega}$ and
$0<\varepsilon<\varepsilon_0$. Then \eqref{eq:projected-Z-nonlinear-form} is equivalent to
\[
    z
    =
    \Phi_{q,\hat\omega,\varepsilon}[z]
    :=
    \mathcal A_{\hat\omega,\varepsilon}^{-1}
    \left(
        \varepsilon^2\mathcal B_{q,\hat\omega,\varepsilon}
        +
        \varepsilon^2\hat\omega^2\mathcal H_q[z]
    \right).
\]
The trace theorem and the uniform boundedness of the DtN operator give
\[
    \|\mathcal B_{q,\hat\omega,\varepsilon}\|_{\mathcal Z(D)'}
    \le
    C\|q\|.
\]

Next, we prove that $\Phi_{q,\hat\omega,\varepsilon}$ is a contraction on
$B_{q,\varepsilon}$. For $z\in B_{q,\varepsilon}$,
\cref{lem:kerr-local-estimate} gives
\[
\begin{aligned}
    \|\mathcal H_q[z]\|_{\mathcal Z(D)'}
    &\le
    C
    \bigl(
        \|q\|^2\|z\|_{H^1(D)}
        +
        \|q\|\|z\|_{H^1(D)}^2
        +
        \|z\|_{H^1(D)}^3
    \bigr)
    \le
    C\varepsilon^2\|q\|^3 .
\end{aligned}
\]
Consequently,
\[
    \|\Phi_{q,\hat\omega,\varepsilon}[z]\|_{H^1(D)}
    \le
    C\varepsilon^2\|q\|
    +
    C\varepsilon^4|\hat\omega|^2\|q\|^3 .
\]
Choose $K$ larger than the constant in the first term, and then decrease
$\varepsilon_0=\varepsilon_0(M_q,M_{\hat\omega})$ so that the second term is
absorbed into $(K-C)\varepsilon^2\|q\|$ for all
$\|q\|\le M_q$ and $|\hat\omega|\le M_{\hat\omega}$. Thus,
$\Phi_{q,\hat\omega,\varepsilon}$ maps $B_{q,\varepsilon}$ into itself. For
$z_1,z_2\in B_{q,\varepsilon}$, we use
\eqref{eq:pc-nonlinear-Hq-lipschitz} and
$\|z_\ell\|_{H^1(D)}\le K\varepsilon^2\|q\|$, $\ell=1,2$, to obtain
\[
\begin{aligned}
    &\|\Phi_{q,\hat\omega,\varepsilon}[z_1]
        -
        \Phi_{q,\hat\omega,\varepsilon}[z_2]\|_{H^1(D)}
    \\
    &\le
    C\varepsilon^2|\hat\omega|^2
    \|\mathcal H_q[z_1]-\mathcal H_q[z_2]\|_{\mathcal Z(D)'}
    \\
    &\le
    C\varepsilon^2|\hat\omega|^2
    \bigl(
        \|q\|^2+\varepsilon^2\|q\|^2+\varepsilon^4\|q\|^2
    \bigr)
    \|z_1-z_2\|_{H^1(D)}
    \\
    &\le
    C\varepsilon^2|\hat\omega|^2\|q\|^2
    \|z_1-z_2\|_{H^1(D)}.
\end{aligned}
\]
After further decreasing $\varepsilon_0=\varepsilon_0(M_q,M_{\hat\omega})$, the
prefactor is at most $1/2$. Banach's fixed-point theorem yields a unique fixed
point in $B_{q,\varepsilon}$. If $q=0$, then $B_{0,\varepsilon}=\{0\}$ and the
unique solution is $z=0$. This proves existence, uniqueness in the small ball,
and \eqref{eq:nonlinear-basic-z-estimate}.

The same formulation gives the parameter dependence. Define
\[
    \mathfrak F(q,\hat\omega,\varepsilon,z)
    :=
    \mathcal A_{\hat\omega,\varepsilon}z
    -
    \varepsilon^2\mathcal B_{q,\hat\omega,\varepsilon}
    -
    \varepsilon^2\hat\omega^2\mathcal H_q[z].
\]
As a map between real Banach spaces, $\mathfrak F$ is real analytic: the Kerr
term is polynomial in $(q,\overline q,z,\overline z)$, and the DtN operator is
holomorphic in the low-frequency parameter. Moreover,
\[
    D_z\mathfrak F
    =
    \mathcal A_{\hat\omega,\varepsilon}
    -
    \varepsilon^2\hat\omega^2D_z\mathcal H_q[z]
\]
is invertible in the above neighbourhood, since the second term is a small
perturbation of $\mathcal A_{\hat\omega,\varepsilon}$. The real analytic
implicit function theorem gives the asserted real-analytic dependence. At
$\varepsilon=0$, the equation reduces to the limiting gradient problem
$(\nabla z,\nabla v)_D=0$ on $\mathcal Z(D)$, and hence
$z^{\rm non}(q,\hat\omega,0)=0$. We then derive the $z$-expansion. Let
\[
    z_{\rm app}
    :=
    \varepsilon^2z_q^0
    +
    c_b\varepsilon^3\hat\omega z_q^1,
    \qquad
    \mathfrak r_z^{\rm non}
    :=
    z^{\rm non}-z_{\rm app},
\]
where $z_q^0,z_q^1$ are the correctors from
\eqref{eq:linear-corrector-equation}. Using the DtN expansion and subtracting
the equation satisfied by $z_{\rm app}$ from
\eqref{eq:projected-Z-nonlinear-form}, we obtain, for every
$v\in\mathcal Z(D)$,
\[
\begin{aligned}
    a^{\mathrm{lin}}_{\hat\omega,\varepsilon}
    (\mathfrak r_z^{\rm non},v)
    &=
    \varepsilon^2\hat\omega^2(z_{\rm app},v)_D
    +
    \varepsilon^2
    \langle
        \mathcal T_D^{\alpha,k_m}[z_{\rm app}],
        v
    \rangle_{\partial D}
    \\
    &\quad
    +
    c_b^2\varepsilon^4\hat\omega^2
    \langle
        \mathcal R_D^{\alpha,k_m}[u_q],
        v
    \rangle_{\partial D}
    +
    \varepsilon^2\hat\omega^2
    \bigl(
        \widetilde{\mathcal N}_\sigma[q,z^{\rm non}],
        v
    \bigr)_D .
\end{aligned}
\]
The first three terms are estimated exactly as in the proof of
\cref{prop:linear-projected-reduction}, giving
$C\varepsilon^4(1+|\hat\omega|^2)\|q\|\|v\|_{H^1(D)}$. By
\eqref{eq:nonlinear-basic-z-estimate} and \cref{lem:kerr-local-estimate}, $\|\mathcal H_q[z^{\rm non}]\|_{\mathcal Z(D)'}
    \le
    C\varepsilon^2\|q\|^3$.
Thus, the nonlinear term is bounded by
$C\varepsilon^4|\hat\omega|^2\|q\|^3\|v\|_{H^1(D)}$. Applying
\cref{lem:linear-Z-lax-milgram} proves
\eqref{eq:nonlinear-z-remainder-estimate}.

It remains to identify the finite-dimensional residual. Substituting
$z^{\rm non}=z^{\rm non}(q,\hat\omega,\varepsilon)$ into
\eqref{eq:projected-X-finite-dimensional-system} gives
$\mathcal F^{\rm non}(q,\hat\omega,\varepsilon)=0$, with
$\mathcal F^{\rm non}$ defined by \eqref{eq:nonlinear-reduced-equation} and
$R^{\rm non}$ by \eqref{eq:def-nonlinear-R}. The second term in
\eqref{eq:def-nonlinear-R} is bounded by
$C\varepsilon^2|\hat\omega|^2\|q\|$, and the third term is bounded by the
uniform boundedness of the DtN operator and
\eqref{eq:nonlinear-basic-z-estimate}. We now estimate the first term. On each
component $D_i$, since $\sigma_D=\sigma_i$ and $u_q=q_i$, we have
\[
\begin{aligned}
    |q_i+z|^2(q_i+z)-|q_i|^2q_i
    ={}&
    2|q_i|^2z
    +
    q_i^2\overline z
    +
    2q_i|z|^2
    +
    \overline q_i z^2
    +
    |z|^2z .
\end{aligned}
\]
Since $z=z^{\rm non}\in\mathcal Z(D)$, both the mean of $z$ and the mean of
$\overline z$ vanish on each component. Hence, the two linear terms in $z$ do
not contribute after testing against $1$, and
\[
\begin{aligned}
    \left|
        \bigl(
            \widetilde{\mathcal N}_\sigma[q,z^{\rm non}],
            1
        \bigr)_{D_i}
    \right|
    \le
    C
    \left(
        |q_i|\|z^{\rm non}\|_{L^2(D_i)}^2
        +
        \|z^{\rm non}\|_{L^3(D_i)}^3
    \right)
    \le
    C\varepsilon^4\|q\|^3 .
\end{aligned}
\]
Combining the three bounds proves \eqref{eq:nonlinear-R-estimate}.

Finally, we establish the smoothness and vanishing properties of the residual.
The real-smooth extension of $R^{\rm non}$ follows from
\eqref{eq:def-nonlinear-R}, the real-analytic extension of $z^{\rm non}$, and
the low-frequency smoothness of the DtN expansion. Since $z^{\rm non}=0$ when
$q=0$, each term in \eqref{eq:def-nonlinear-R} vanishes and
$R^{\rm non}(0,\hat\omega,\varepsilon)=0$. Since
$z^{\rm non}(q,\hat\omega,0)=0$, the definition also gives
$R^{\rm non}(q,\hat\omega,0)=0$. Moreover, the expansion
\eqref{eq:nonlinear-z-expansion} shows that the first and third terms in
\eqref{eq:def-nonlinear-R} are $\mathcal O(\varepsilon^2)$, while the second term has an
explicit factor $\varepsilon^2$; hence
$\partial_\varepsilon R^{\rm non}(q,\hat\omega,0)=0$. The real-smooth extension
of $\mathcal F^{\rm non}$ follows from \eqref{eq:nonlinear-reduced-equation}.
\end{proof}

\begin{remark}[Comparison with the linear reduction]\label{rem:comparison-with-linear-reduction}
Although the nonlinear $\mathcal Z(D)$-equation is solved by a fixed-point
argument, its first two correctors coincide with those in the linear
reduction. The reason is that, for componentwise constant $\sigma_D$,
$\mathcal N_\sigma[u_q]\in\mathcal X(D)$, and hence
\[
    (\mathcal N_\sigma[u_q],v)_D=0,
    \qquad
    v\in\mathcal Z(D).
\]
Thus, the leading Kerr contribution does not change the zero-average
correctors. It enters the finite-dimensional residual through the cubic modal
vector $V_\sigma(q)$, while the $z$-dependent nonlinear terms are absorbed into
$R^{\rm non}$.
\end{remark}

Therefore, the projected reduction has produced the residual equation
$\mathcal F^{\rm non}(q,\hat\omega,\varepsilon)=0$, with
$\mathcal F^{\rm non}$ defined in \eqref{eq:nonlinear-reduced-equation}. We now
solve this equation near a simple positive linear mode. As in the linear
theorem, the modal coordinates separate the distinguished eigendirection from its
$V$-orthogonal complement. In the nonlinear case, the branch also carries an
amplitude parameter. The cubic modal term then produces the leading
amplitude-dependent correction to the resonance frequency.

\begin{theorem}[Small-amplitude nonlinear continuation in modal coordinates]
\label{thm:nonlinear-modal-continuation}
Let $(\lambda_j,p_j)$ and
$\hat\omega_j^0,\hat\omega_j^1,q_j^1,U_j^0,U_j^1$ be as in
\cref{thm:linear-simple-resonance-branch}. Define
\begin{align*}
    \beta_j
    :=
    p_j^\top V_\sigma(p_j)
    =
    \sum_{i=1}^N|D_i|\sigma_i(p_j)_i^4,
    \qquad
    \xi_j^3
    :=
    \lambda_j
    \sum_{i\ne j}
    \frac{p_i^\top V_\sigma(p_j)}
    {\lambda_i-\lambda_j}\,p_i,
    \qquad
    U_j^3
    :=
    \sum_{i=1}^N(\xi_j^3)_i\chi_{D_i}.
\end{align*}
Then there exist $t_j>0$, $\varepsilon_j>0$, $\gamma_j>0$, and real-smooth
functions
\[
    \hat\omega_j^{\rm non}(t,\varepsilon),
    \qquad
    \xi_j(t,\varepsilon)\in E_j^\perp
\]
defined for $|t|<t_j$ and $0\le\varepsilon<\varepsilon_j$, where
$E_j^\perp$ is defined in
\eqref{eq:def-modal-complement-space}. For each $0<|t|<t_j$, impose the phase
normalization $p_j^\top Vq=t\in\mathbb R$. Consider local solutions of
$\mathcal F^{\rm non}(q,\hat\omega,\varepsilon)=0$, with
$\mathcal F^{\rm non}$ defined in \eqref{eq:nonlinear-reduced-equation}, of the
form
\[
    q=tp_j+\xi,
    \qquad
    \xi\in E_j^\perp,
    \qquad
    \|\xi\|<\gamma_j,
    \qquad
    |\hat\omega-\hat\omega_j^0|<\gamma_j .
\]
In this class, there is a unique local nontrivial branch, given by
\[
    \hat\omega
    =
    \hat\omega_j^{\rm non}(t,\varepsilon),
    \qquad
    q
    =
    q_j^{\rm non}(t,\varepsilon)
    :=
    tp_j+\xi_j(t,\varepsilon).
\]
Without the phase normalization, the branch is
locally unique up to the gauge transformation
$q\mapsto e^{\mathrm i\theta}q$. As $(t,\varepsilon)\to(0,0)$,
\begin{align}
    \hat\omega_j^{\rm non}(t,\varepsilon)
    &=
    \hat\omega_j^0
    -
    \frac12\hat\omega_j^0\beta_jt^2
    -
    \mathrm i\,\hat\omega_j^1\varepsilon
    +
    \mathcal O\bigl(
        |t|^4+|t|^2\varepsilon+\varepsilon^2
    \bigr),
    \label{eq:modal-scaled-frequency-expansion}
    \\
    \xi_j(t,\varepsilon)
    &=
    t^3\xi_j^3
    -
    \mathrm i\,t\varepsilon q_j^1
    +
    \mathcal O\bigl(
        |t|^5+|t|^3\varepsilon+|t|\varepsilon^2
    \bigr).
    \label{eq:modal-xi-expansion}
\end{align}
Equivalently, with $\delta=\varepsilon^2$, we have
\begin{align}
    \omega_j^{\rm non}(t,\delta)
    &=
    c_b\hat\omega_j^0\sqrt\delta
    -
    \frac12 c_b\hat\omega_j^0\beta_jt^2\sqrt\delta
    -
    \mathrm i\,c_b\hat\omega_j^1\delta
    +
    \mathcal O\bigl(
        |t|^4\sqrt\delta+|t|^2\delta+\delta^{3/2}
    \bigr),
    \label{eq:modal-physical-frequency-expansion}
    \\
    q_j^{\rm non}(t,\delta)
    &=
    tp_j
    +
    t^3\xi_j^3
    -
    \mathrm i\,tq_j^1\sqrt\delta
    +
    \mathcal O\bigl(
        |t|^5+|t|^3\sqrt\delta+|t|\delta
    \bigr).
    \label{eq:modal-q-expansion}
\end{align}
The corresponding reconstructed interior field satisfies
\begin{equation}
\label{eq:modal-nonlinear-field-expansion}
    u_j^{\rm non}(t,\delta)
    =
    tU_j^0
    +
    t^3U_j^3
    -
    \mathrm i\,tU_j^1\sqrt\delta
    +
    \mathcal O
    (
        |t|^5+|t|^3\sqrt\delta+|t|\delta
    ).
\end{equation}
\end{theorem}

\begin{proof}
Recall the linear and nonlinear residuals $\mathcal F^{\rm lin}$ and
$\mathcal F^{\rm non}$ from
\cref{prop:linear-projected-reduction,prop:nonlinear-projected-reduction}.

\smallskip
\noindent\emph{Step 1: Gauge symmetry and linearization at zero amplitude.}
By \cref{prop:nonlinear-projected-reduction}, the map $R^{\rm non}$ is real
smooth and satisfies
\[
    R^{\rm non}(0,\hat\omega,\varepsilon)=0,
    \qquad
    R^{\rm non}(q,\hat\omega,0)=0.
\]
The fixed-point equation defining $z^{\rm non}$ is gauge covariant. Hence,
by uniqueness in the contraction ball,
\[
    z^{\rm non}(e^{\mathrm i\theta}q,\hat\omega,\varepsilon)
    =
    e^{\mathrm i\theta}
    z^{\rm non}(q,\hat\omega,\varepsilon).
\]
Consequently, $R^{\rm non}$ and $\mathcal F^{\rm non}$ are gauge equivariant.
Next, we identify the linearization at $q=0$. Differentiate
\eqref{eq:projected-Z-nonlinear-form} with respect to $q$ at $q=0$ in the
direction $\widetilde q$. At $q=0$, the correction satisfies
$z^{\rm non}(0,\hat\omega,\varepsilon)=0$. The first real Fr\'echet derivative
of the cubic Kerr term also vanishes at the origin. Therefore, every
chain-rule contribution containing
$D_qz^{\rm non}(0,\hat\omega,\varepsilon)[\widetilde q]$ in the Kerr term is
zero. Hence,
\[
    a^{\mathrm{lin}}_{\hat\omega,\varepsilon}
    \bigl(
        D_qz^{\rm non}(0,\hat\omega,\varepsilon)[\widetilde q],v
    \bigr)
    =
    \varepsilon^2
    \langle
        \mathcal T_D^{\alpha,k_m}[u_{\widetilde q}],v
    \rangle_{\partial D},
    \qquad
    v\in\mathcal Z(D).
\]
By uniqueness in the linear projected problem \eqref{eq:linear-Z-equation-for-z},
\begin{equation}
\label{eq:Dq-z-nonlinear-linear}
    D_qz^{\rm non}(0,\hat\omega,\varepsilon)[\widetilde q]
    =
    z^{\rm lin}(\widetilde q,\hat\omega,\varepsilon).
\end{equation}
The first real derivative at $q=0$ of the nonlinear term in
\eqref{eq:def-nonlinear-R} is zero. Differentiating the remaining two terms
and using \eqref{eq:Dq-z-nonlinear-linear}, we obtain
\[
\begin{aligned}
    \bigl(
        D_qR^{\rm non}(0,\hat\omega,\varepsilon)[\widetilde q]
    \bigr)_i
    =
    c_b^2\varepsilon^2\hat\omega^2
    \langle
        \mathcal R_D^{\alpha,k_m}[u_{\widetilde q}],1
    \rangle_{\partial D_i}
    +
    \langle
        \mathcal T_D^{\alpha,k_m}
        [z^{\rm lin}(\widetilde q,\hat\omega,\varepsilon)],1
    \rangle_{\partial D_i}.
\end{aligned}
\]
Combining this identity with the definition of the reduced residuals gives
\begin{equation}
\label{eq:nonlinear-linearization-identity}
\begin{aligned}
    D_qR^{\rm non}(0,\hat\omega,\varepsilon)[\widetilde q]
    =
    R^{\rm lin}(\widetilde q,\hat\omega,\varepsilon),
    \qquad
    D_q\mathcal F^{\rm non}(0,\hat\omega,\varepsilon)[\widetilde q]
    =
    \mathcal F^{\rm lin}(\widetilde q,\hat\omega,\varepsilon).
\end{aligned}
\end{equation}

\smallskip
\noindent\emph{Step 2: Transverse equation and divided scalar equation.}
Write $q=tp_j+\xi$, where $t\in\mathbb R$ and
$\xi\in E_j^\perp$. The transverse equation is
\begin{equation}
\label{eq:modal-transverse-equations}
    \mathcal P_j^\perp
    \mathcal F^{\rm non}(tp_j+\xi,\hat\omega,\varepsilon)
    =
    0.
\end{equation}
At $(t,\xi,\hat\omega,\varepsilon)=(0,0,\hat\omega_j^0,0)$, the derivative
with respect to $\xi$ is
\[
    \widetilde\xi
    \longmapsto
    \mathcal P_j^\perp\bigl((\lambda_jV-C^0)\widetilde\xi\bigr).
\]
If $\widetilde\xi=\sum_{i\ne j}c_ip_i$, then
\[
    \mathcal P_j^\perp\bigl((\lambda_jV-C^0)\widetilde\xi\bigr)
    =
    \sum_{i\ne j}(\lambda_j-\lambda_i)c_iVp_i .
\]
Since $\lambda_j$ is simple, this map is an isomorphism from
$E_j^\perp$ onto $\operatorname{Ran}\mathcal P_j^\perp$.
Viewing the finite-dimensional complex spaces as real spaces, the real implicit
function theorem gives a unique real-smooth map
\[
    \xi
    =
    \Xi^{\rm non}(t,\hat\omega,\varepsilon)
    \in E_j^\perp
\]
solving \eqref{eq:modal-transverse-equations}. Since
$\mathcal F^{\rm non}(0,\hat\omega,\varepsilon)=0$, uniqueness gives
$\Xi^{\rm non}(0,\hat\omega,\varepsilon)=0$. Thus, the complement variable has
been solved as a function of $(t,\hat\omega,\varepsilon)$. We now form the
scalar equation along this graph:
\[
    \phi_j^{\rm non}(t,\hat\omega,\varepsilon)
    :=
    p_j^\top
    \mathcal F^{\rm non}
    \bigl(
        tp_j+\Xi^{\rm non}(t,\hat\omega,\varepsilon),
        \hat\omega,\varepsilon
    \bigr).
\]
Since $\phi_j^{\rm non}(0,\hat\omega,\varepsilon)=0$, the scalar equation
contains the trivial zero-amplitude factor. We divide out this factor by
defining
\begin{equation}
\label{eq:def-nonlinear-Gj}
    G_j^{\rm non}(t,\hat\omega,\varepsilon)
    :=
    \int_0^1
    \partial_t
    \phi_j^{\rm non}(st,\hat\omega,\varepsilon)\,ds.
\end{equation}
Indeed,
\[
\begin{aligned}
    \phi_j^{\rm non}(t,\hat\omega,\varepsilon)
    =
    \int_0^t
    \partial_s\phi_j^{\rm non}(s,\hat\omega,\varepsilon)\,ds
    =
    tG_j^{\rm non}(t,\hat\omega,\varepsilon).
\end{aligned}
\]
In particular,
$G_j^{\rm non}(0,\hat\omega,\varepsilon)
=\partial_t\phi_j^{\rm non}(0,\hat\omega,\varepsilon)$.
Next, we compute the zero-amplitude value of the divided equation at
$\varepsilon=0$. Here, $R^{\rm non}(q,\hat\omega,0)=0$ and
$DV_\sigma(0)=0$ for the first real derivative of the Kerr map. Thus, the first
$t$-derivative of the transverse equation contains only the linear reduced
part.
Differentiating \eqref{eq:modal-transverse-equations} with respect to $t$ at
$t=0$ and applying $p_i^\top$, $i\ne j$, gives
\[
    p_i^\top(\hat\omega^2V-C^0)
    \left(
        p_j+\partial_t\Xi^{\rm non}(0,\hat\omega,0)
    \right)
    =
    0,
    \qquad i\ne j.
\]
The contribution of $p_j$ vanishes. Writing
$\partial_t\Xi^{\rm non}(0,\hat\omega,0)=\sum_{i\ne j}d_ip_i$, we obtain
$(\hat\omega^2-\lambda_i)d_i=0$ for $i\ne j$. After shrinking the
neighbourhood of $\hat\omega_j^0$, one has
$\hat\omega^2\ne\lambda_i$ for every $i\ne j$, and hence,
$\partial_t\Xi^{\rm non}(0,\hat\omega,0)=0$. It follows that
\[
    G_j^{\rm non}(0,\hat\omega,0)
    =
    p_j^\top(\hat\omega^2V-C^0)p_j
    =
    \hat\omega^2-\lambda_j.
\]
Thus,
\[
    G_j^{\rm non}(0,\hat\omega_j^0,0)=0,
    \qquad
    \partial_{\hat\omega}
    G_j^{\rm non}(0,\hat\omega_j^0,0)
    =
    2\hat\omega_j^0\ne0.
\]
The real implicit function theorem gives a unique real-smooth function
$\hat\omega_j^{\rm non}(t,\varepsilon)$. Set
\[
    \xi_j(t,\varepsilon)
    :=
    \Xi^{\rm non}
    \bigl(
        t,\hat\omega_j^{\rm non}(t,\varepsilon),\varepsilon
    \bigr),
    \qquad
    q_j^{\rm non}(t,\varepsilon)
    :=
    tp_j+\xi_j(t,\varepsilon).
\]
For $t\ne0$, the equation $G_j^{\rm non}=0$ is equivalent to the original
scalar equation, so the transverse and scalar equations yield the full
reduced equation.

It remains to pass from normalized uniqueness to uniqueness up to gauge. Let
$(q,\hat\omega)$ be a nearby nonzero solution. If
$p_j^\top Vq=0$, then $q\in E_j^\perp$ and the transverse uniqueness at
$t=0$ gives $q=0$, a contradiction after shrinking the neighbourhood. Hence
$p_j^\top Vq\ne0$.

Choose $\theta$ so that $p_j^\top V(e^{\mathrm i\theta}q)$ is real. By gauge
equivariance, $(e^{\mathrm i\theta}q,\hat\omega)$ is another solution in the
normalized class. Normalized uniqueness identifies it with the constructed
branch. Thus the original solution differs from that branch only by a gauge
factor.

\smallskip
\noindent\emph{Step 3: Zero-amplitude limit and parity.}
At $t=0$, one has $q_j^{\rm non}(0,\varepsilon)=0$. Since
$\mathcal F^{\rm non}(0,\hat\omega,\varepsilon)=0$ for all nearby
$\hat\omega$, differentiating this identity in $\hat\omega$ gives
$\partial_{\hat\omega}\mathcal F^{\rm non}(0,\hat\omega,\varepsilon)=0$.
This removes the frequency-derivative term from the chain rule. Therefore,
when differentiating
\[
    \mathcal F^{\rm non}
    \bigl(
        q_j^{\rm non}(t,\varepsilon),
        \hat\omega_j^{\rm non}(t,\varepsilon),
        \varepsilon
    \bigr)
    =
    0
\]
with respect to $t$ at $t=0$, the chain-rule term involving
$\partial_t\hat\omega_j^{\rm non}(0,\varepsilon)$ vanishes. Using
\eqref{eq:nonlinear-linearization-identity}, the differentiated equation
becomes
\[
\begin{aligned}
0
=
D_q\mathcal F^{\rm non}
\bigl(
    0,\hat\omega_j^{\rm non}(0,\varepsilon),\varepsilon
\bigr)
\bigl[
    \partial_tq_j^{\rm non}(0,\varepsilon)
\bigr]
=
\mathcal F^{\rm lin}
\bigl(
    \partial_tq_j^{\rm non}(0,\varepsilon),
    \hat\omega_j^{\rm non}(0,\varepsilon),
    \varepsilon
\bigr).
\end{aligned}
\]
The normalization identity $p_j^\top Vq_j^{\rm non}(t,\varepsilon)=t$
gives, after differentiating at $t=0$,
\[
    p_j^\top V
    \partial_tq_j^{\rm non}(0,\varepsilon)
    =
    1.
\]
The local uniqueness in \cref{thm:linear-simple-resonance-branch} therefore yields
\begin{equation}
\label{eq:nonlinear-zero-amplitude-linear-branch}
    \hat\omega_j^{\rm non}(0,\varepsilon)
    =
    \hat\omega_j^{\rm lin}(\varepsilon),
    \qquad
    \partial_tq_j^{\rm non}(0,\varepsilon)
    =
    q_j^{\rm lin}(\varepsilon).
\end{equation}
Gauge equivariance with $\theta=\pi$ shows that
$(\hat\omega_j^{\rm non}(t,\varepsilon),
 -q_j^{\rm non}(t,\varepsilon))$
is the normalized solution with amplitude $-t$. By local uniqueness,
\[
    \hat\omega_j^{\rm non}(-t,\varepsilon)
    =
    \hat\omega_j^{\rm non}(t,\varepsilon),
    \qquad
    q_j^{\rm non}(-t,\varepsilon)
    =
    -q_j^{\rm non}(t,\varepsilon).
\]
Thus, the frequency is even and the amplitude vector is odd in $t$.

\smallskip
\noindent\emph{Step 4: Amplitude derivatives at $\varepsilon=0$.}
Set
$q(t):=q_j^{\rm non}(t,0)$ and
$\hat\omega(t):=\hat\omega_j^{\rm non}(t,0)$. Since
$R^{\rm non}(q,\hat\omega,0)=0$, the reduced equation becomes
\begin{equation}
\label{eq:nonlinear-reduced-at-epsilon-zero}
    \hat\omega(t)^2
    \bigl(
        Vq(t)+V_\sigma(q(t))
    \bigr)
    -
    C^0q(t)
    =
    0.
\end{equation}
At $t=0$, we have $q(0)=0$, $q_t(0)=p_j$, and
$\hat\omega(0)=\hat\omega_j^0$. Differentiating
\eqref{eq:nonlinear-reduced-at-epsilon-zero} once gives
$(\lambda_jV-C^0)q_t(0)=0$, consistently with $q_t(0)=p_j$.
Differentiating twice gives
\[
    (\lambda_jV-C^0)q_{tt}(0)
    +
    4\hat\omega_j^0\hat\omega_t(0)Vp_j
    =
    0.
\]
Projection onto $p_j$ yields $\hat\omega_t(0)=0$. The second derivative of the
normalization gives $p_j^\top Vq_{tt}(0)=0$, and hence $q_{tt}(0)=0$.
For the third derivative, these identities remove all lower-order mixed terms,
and the product rule gives
\[
    \left.
    \frac{d^3}{dt^3}
    \bigl(
        \hat\omega^2Vq
    \bigr)
    \right|_{t=0}
    =
    \lambda_jVq_{ttt}(0)
    +
    6\hat\omega_j^0\hat\omega_{tt}(0)Vp_j.
\]
For the Kerr term, we use its real cubic structure. Since
$q(0)=0$, $q_t(0)=p_j$, and $q_{tt}(0)=0$, we have
\[
    q(t)=tp_j+\mathcal O(t^3).
\]
By the cubic homogeneity of $V_\sigma$,
\[
    V_\sigma(q(t))
    =
    V_\sigma(tp_j+\mathcal O(t^3))
    =
    t^3V_\sigma(p_j)+\mathcal O(t^5).
\]
Hence,
\[
    \left.
    \frac{d^3}{dt^3}
    V_\sigma(q(t))
    \right|_{t=0}
    =
    6V_\sigma(p_j).
\]
Moreover, $V_\sigma(q(t))$ has a zero of order three at $t=0$. Thus, when
differentiating $\hat\omega^2(t)V_\sigma(q(t))$, all terms containing
derivatives of $\hat\omega^2(t)$ vanish at $t=0$. Since
$\hat\omega(0)^2=\lambda_j$, we obtain
\[
    \left.
    \frac{d^3}{dt^3}
    \bigl(
        \hat\omega^2V_\sigma(q)
    \bigr)
    \right|_{t=0}
    =
    6\lambda_jV_\sigma(p_j).
\]
Therefore,
\begin{equation}
\label{eq:nonlinear-third-t-derivative-equation}
    (\lambda_jV-C^0)q_{ttt}(0)
    +
    6\hat\omega_j^0\hat\omega_{tt}(0)Vp_j
    +
    6\lambda_jV_\sigma(p_j)
    =
    0.
\end{equation}
Projecting onto $p_j$ gives $\hat\omega_{tt}(0)
    =
    -\hat\omega_j^0\beta_j $.
The third derivative of the normalization gives
$p_j^\top Vq_{ttt}(0)=0$. Projecting
\eqref{eq:nonlinear-third-t-derivative-equation} onto $p_i$, $i\ne j$, gives
\[
    \frac16q_{ttt}(0)
    =
    \lambda_j
    \sum_{i\ne j}
    \frac{
        p_i^\top V_\sigma(p_j)
    }{
        \lambda_i-\lambda_j
    }p_i
    =
    \xi_j^3.
\]

\smallskip
\noindent\emph{Step 5: Expansions, physical frequency, and reconstructed field.}
By \eqref{eq:nonlinear-zero-amplitude-linear-branch} and
\eqref{eq:linear-scaled-branch-expansion},
\[
    \hat\omega_j^{\rm non}(0,\varepsilon)
    =
    \hat\omega_j^0
    -
    \mathrm i\,\hat\omega_j^1\varepsilon
    +
    \mathcal O(\varepsilon^2),\qquad
    \partial_tq_j^{\rm non}(0,\varepsilon)
    =
    p_j
    -
    \mathrm i\,\varepsilon q_j^1
    +
    \mathcal O(\varepsilon^2).
\]
The preceding computation gives the second and third $t$-derivatives at
$\varepsilon=0$. Since the branch is smooth, these derivatives depend smoothly
on $\varepsilon$. Taylor expansion in $\varepsilon$ gives
\[
    \partial_{tt}
    \hat\omega_j^{\rm non}(0,\varepsilon)
    =
    -\hat\omega_j^0\beta_j+\mathcal O(\varepsilon),
    \qquad
    \frac16
    \partial_{ttt}
    q_j^{\rm non}(0,\varepsilon)
    =
    \xi_j^3+\mathcal O(\varepsilon).
\]
Since $\hat\omega_j^{\rm non}$ is even in $t$, Taylor's formula gives
\[
    \hat\omega_j^{\rm non}(t,\varepsilon)
    =
    \hat\omega_j^{\rm non}(0,\varepsilon)
    +
    \frac{t^2}{2}
    \partial_{tt}
    \hat\omega_j^{\rm non}(0,\varepsilon)
    +
    \mathcal O(t^4),
\]
uniformly for small $\varepsilon$. Hence,
\[
    \hat\omega_j^{\rm non}(t,\varepsilon)
    =
    \hat\omega_j^0
    -
    \frac12\hat\omega_j^0\beta_jt^2
    -
    \mathrm i\,\hat\omega_j^1\varepsilon
    +
    \mathcal O\bigl(
        |t|^4+|t|^2\varepsilon+\varepsilon^2
    \bigr).
\]
Similarly, oddness of $q_j^{\rm non}$ gives
\[
    q_j^{\rm non}(t,\varepsilon)
    =
    t\,\partial_tq_j^{\rm non}(0,\varepsilon)
    +
    \frac{t^3}{6}
    \partial_{ttt}q_j^{\rm non}(0,\varepsilon)
    +
    \mathcal O(t^5).
\]
Substitution yields
\[
    q_j^{\rm non}(t,\varepsilon)
    =
    tp_j
    +
    t^3\xi_j^3
    -
    \mathrm i\,t\varepsilon q_j^1
    +
    \mathcal O\bigl(
        |t|^5+|t|^3\varepsilon+|t|\varepsilon^2
    \bigr).
\]
Since $q_j^{\rm non}=tp_j+\xi_j$, the preceding two displays prove
\eqref{eq:modal-scaled-frequency-expansion} and
\eqref{eq:modal-xi-expansion}. Now set $\varepsilon=\sqrt\delta$. Using
$\omega=c_b\varepsilon\hat\omega$ gives
\eqref{eq:modal-physical-frequency-expansion} and \eqref{eq:modal-q-expansion}.

Finally, reconstruct the interior field. The linearity of $q\mapsto u_q$ gives
the expansion of the piecewise constant lift, and
\[
    u_j^{\rm non}
    =
    u_{q_j^{\rm non}}
    +
    z^{\rm non}
    \bigl(
        q_j^{\rm non},
        \hat\omega_j^{\rm non},
        \varepsilon
    \bigr).
\]
The zero-average correction is smaller. Since
$\|q_j^{\rm non}(t,\varepsilon)\|=\mathcal O(|t|)$,
\eqref{eq:nonlinear-basic-z-estimate} gives
\[
    \left\|
        z^{\rm non}
        \bigl(
            q_j^{\rm non},
            \hat\omega_j^{\rm non},
            \varepsilon
        \bigr)
    \right\|_{H^1(D)}
    \le
    C|t|\varepsilon^2
    =
    \mathcal O(|t|\delta).
\]
This contribution is contained in the stated remainder. This proves
\eqref{eq:modal-nonlinear-field-expansion}.
\end{proof}

\begin{remark}[Comparison with the linear branch]
Comparing \eqref{eq:modal-physical-frequency-expansion} with
\eqref{eq:linear-resonance-expansion}, the nonlinear branch has the same
leading subwavelength frequency and the same leading radiative correction as
the linear branch. More precisely,
\[
    \omega_j^{\rm non}(t,\delta)
    =
    \omega_j^{\rm lin}(\delta)
    -
    \frac12 c_b\hat\omega_j^0\beta_jt^2\sqrt\delta
    +
    \mathcal O\bigl(
        |t|^4\sqrt\delta+|t|^2\delta+\delta^{3/2}
    \bigr).
\]
Thus, the first nonlinear effect is the real frequency shift generated by the
cubic modal term $V_\sigma(q)$; the leading imaginary part remains the linear
radiative term $-\mathrm i\,c_b\hat\omega_j^1\delta$.

The field expansion has the analogous structure. From
\eqref{eq:modal-nonlinear-field-expansion} and
\eqref{eq:linear-physical-field-expansion},
\[
    u_j^{\rm non}(t,\delta)
    =
    t\,u_j^{\rm lin}(\delta)
    +
    t^3U_j^3
    +
    \mathcal O\bigl(
        |t|^5+|t|^3\sqrt\delta+|t|\delta
    \bigr).
\]
Hence, the nonlinear field contains the amplitude-scaled linear field, while
the first genuinely nonlinear correction is the cubic modal component
$t^3U_j^3$.
\end{remark}

\section{Bound states in the continuum in symmetric configurations}
\label{section:symmetric-configurations}

We study reflection-protected bound states in the continuum at the
high-symmetry point $\alpha=0$. Under the low-frequency scaling
\eqref{eq:low-frequency-scaling}, this corresponds to $a=0$.
The linear symmetry decomposition produces exact antisymmetric BIC branches,
and the same mechanism persists under reflection-symmetric Kerr nonlinearities.

\subsection{Reflection symmetry}

We first introduce the reflection notation used below. The physical reflection
$R_\ell$ induces a component permutation $\pi$, fixed and paired component
indices, and symmetric and antisymmetric subspaces for reduced amplitudes and
functions. We then prove that the leading reduced capacitance problem
decomposes accordingly.

\begin{assumption}[Reflection symmetry]
\label{ass:periodic-reflection-symmetry}
For $x=(x_\ell,x_d)$, set $R_\ell x:=(-x_\ell,x_d)$. We assume that the
resonator configuration is invariant under this reflection: $R_\ell D=D$.

\end{assumption}

The symmetry $R_\ell D=D$ also determines the component permutation used in the following. Since
$R_\ell$ is an isometry and the components $D_i$ are connected and pairwise
disjoint, each reflected component $R_\ell D_i$ is again a component of $D$.
Thus, there is a unique permutation $\pi$ of $\{1,\ldots,N\}$ such that
\begin{equation}
\label{eq:component-permutation}
    R_\ell D_i=D_{\pi(i)},
    \qquad
    i=1,\ldots,N .
\end{equation}
Since $R_\ell^2=I$, this permutation is an involution: $\pi^2=I$.

We now define the fixed components and the reflected pairs determined by this
permutation.

\begin{definition}[Fixed indices and reflected pairs]
\label{def:reflection-fixed-pairs}
Under \cref{ass:periodic-reflection-symmetry}, let $\pi$ be the induced
permutation from \eqref{eq:component-permutation}. The fixed-index set and the
reflected-pair set are
\begin{equation}
\label{eq:reflection-fixed-and-pair-sets}
    \mathcal I_{\pi}^{\mathrm f}
    :=
    \{i:1\le i\le N,\ \pi(i)=i\},
    \qquad
    \mathcal I_{\pi}^{\mathrm p}
    :=
    \bigl\{\{i,\pi(i)\}:1\le i\le N,\ i<\pi(i)\bigr\}.
\end{equation}
Their counts are
\begin{equation}
\label{eq:reflection-fixed-and-pair-numbers}
    n_{\pi}^{\mathrm f}:=|\mathcal I_{\pi}^{\mathrm f}|,
    \qquad
    n_{\pi}^{\mathrm p}:=|\mathcal I_{\pi}^{\mathrm p}|.
\end{equation}
\end{definition}

\begin{figure}[htbp]
\centering
\begin{tikzpicture}[scale=0.9, every node/.style={font=\small}]
    \def\cellw{5.8}
    \def\cellh{3.0}

    \draw[black, thick] (-0.5*\cellw,-0.5*\cellh) rectangle (0.5*\cellw,0.5*\cellh);
    \node[below] at (0,-0.5*\cellh-0.18) {$x_\ell$};
    \node[left] at (-0.5*\cellw-0.18,0) {$x_d$};

    \filldraw[fill=blue!15, draw=blue!60!black, thick]
        (-1.48,0) ellipse (0.34 and 0.56);
    \node[font=\scriptsize] at (-1.48,0) {$D_1$};

    \filldraw[fill=blue!15, draw=blue!60!black, thick]
        (1.48,0) ellipse (0.34 and 0.56);
    \node[font=\scriptsize] at (1.48,0) {$D_3$};

    \filldraw[fill=blue!12, draw=blue!60!black, thick]
        (0,0) circle (0.34);
    \node[font=\scriptsize] at (0,0) {$D_2$};

    \filldraw[fill=blue!10, draw=blue!60!black, thick]
        (0,0.82) ellipse (0.54 and 0.26);
    \node[font=\scriptsize] at (0,0.82) {$D_4$};

    \filldraw[fill=blue!10, draw=blue!60!black, thick]
        (0,-0.82) ellipse (0.54 and 0.26);
    \node[font=\scriptsize] at (0,-0.82) {$D_5$};
\end{tikzpicture}
\caption{Reflection symmetry of five resonators in one periodic cell. The
induced component permutation is $(\pi(1),\ldots,\pi(5))=(3,2,1,4,5)$. Hence,
$\mathcal I_{\pi}^{\mathrm f}=\{2,4,5\}$,
$\mathcal I_{\pi}^{\mathrm p}=\{\{1,3\}\}$, and
$n_{\pi}^{\mathrm f}=3$, $n_{\pi}^{\mathrm p}=1$.}
\label{fig:reflection-symmetric-cell}
\end{figure}
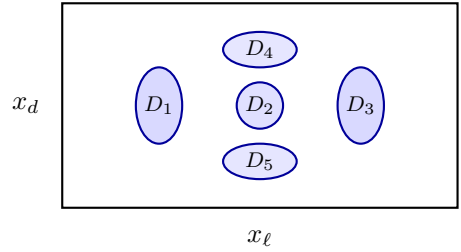

We next define the symmetric and antisymmetric subspaces for reduced amplitudes
and functions.

\begin{definition}[Reflection subspaces]
\label{def:reflection-subspaces}
Under \cref{ass:periodic-reflection-symmetry}, let $\pi$ be the induced
permutation from \eqref{eq:component-permutation}, and let
$\Pi:\mathbb C^N\to\mathbb C^N$ be the corresponding permutation matrix,
defined by
\begin{equation}
\label{eq:reflection-vector-action}
    (\Pi q)_i:=q_{\pi(i)},
    \qquad
    q\in\mathbb C^N .
\end{equation}
Equivalently, $\Pi_{ij}=1$ exactly when $j=\pi(i)$, and $\Pi_{ij}=0$
otherwise. Define two vector subspaces
\begin{equation}
\label{eq:reflection-complex-vector-subspaces}
    \mathbb C^N_{\mathrm{sym}}
    :=
    \{q\in\mathbb C^N:\Pi q=q\},
    \qquad
    \mathbb C^N_{\mathrm{ant}}
    :=
    \{q\in\mathbb C^N:\Pi q=-q\}.
\end{equation}
The corresponding real vector subspaces are
\begin{equation}
\label{eq:reflection-real-vector-subspaces}
    \mathbb R^N_{\mathrm{sym}}
    :=
    \mathbb C^N_{\mathrm{sym}}\cap\mathbb R^N,
    \qquad
    \mathbb R^N_{\mathrm{ant}}
    :=
    \mathbb C^N_{\mathrm{ant}}\cap\mathbb R^N .
\end{equation}
For functions on an $R_\ell$-invariant set, define the reflected pull-back by
\begin{equation}
\label{eq:function-reflection-action}
    \mathcal R_\ell[u](x):=u(R_\ell x).
\end{equation}
The corresponding function subspaces are
\begin{equation}
\label{eq:reflection-complex-function-spaces}
\begin{aligned}
    H^1_{\mathrm{sym}}(D;\mathbb C)
    :=
    \{u\in H^1(D;\mathbb C):\mathcal R_\ell[u]=u\},
    \quad
    H^1_{\mathrm{ant}}(D;\mathbb C)
    :=
    \{u\in H^1(D;\mathbb C):\mathcal R_\ell[u]=-u\}.
\end{aligned}
\end{equation}
The real-valued function subspaces are
\begin{equation}
\label{eq:reflection-real-function-spaces}
    H^1_{\mathrm{sym}}(D;\mathbb R)
    :=
    H^1_{\mathrm{sym}}(D;\mathbb C)\cap H^1(D;\mathbb R),
    \qquad
    H^1_{\mathrm{ant}}(D;\mathbb R)
    :=
    H^1_{\mathrm{ant}}(D;\mathbb C)\cap H^1(D;\mathbb R).
\end{equation}
The piecewise-constant lift is reflection-equivariant: for every
$q\in\mathbb C^N$,
\begin{equation}
\label{eq:reflection-uq}
    \mathcal R_\ell[u_q]=u_{\Pi q}.
\end{equation}
Consequently, $u_q$ is symmetric if and only if
$q\in\mathbb C^N_{\mathrm{sym}}$, and antisymmetric if and only if
$q\in\mathbb C^N_{\mathrm{ant}}$.
\end{definition}

Since $\Pi^2=I$ and $\mathcal R_\ell^2=I$, every vector and every function
decomposes uniquely into symmetric and antisymmetric parts. For example,
\[
    q_{\mathrm{sym}}:=\frac{q+\Pi q}{2},
    \quad
    q_{\mathrm{ant}}:=\frac{q-\Pi q}{2},
    \qquad
    u_{\mathrm{sym}}:=\frac{u+\mathcal R_\ell[u]}{2},
    \quad
    u_{\mathrm{ant}}:=\frac{u-\mathcal R_\ell[u]}{2}.
\]

This reflection decomposition passes to the reduced capacitance problem.

\begin{proposition}[Reflection decomposition of the reduced capacitance problem]
\label{prop:reduced-reflection-decomposition}
Under \cref{ass:periodic-reflection-symmetry}, the matrices $V$ and $C^0$
commute with the reflection permutation matrix:
\[
    V\Pi=\Pi V,
    \qquad
    C^0\Pi=\Pi C^0 .
\]
Consequently, the generalized eigenvalue problem
\eqref{eq:eigenvalue-problem-of-C} splits on
$\mathbb R^N_{\mathrm{sym}}$ and $\mathbb R^N_{\mathrm{ant}}$ as follows.
\begin{enumerate}[label=(\roman*)]
    \item The symmetric restriction has dimension $n_{\pi}^{\mathrm f}+n_{\pi}^{\mathrm p}$
     and admits real $V$-orthonormal eigenpairs
    \[
        \{(\lambda_j^{\mathrm{sym}},p_j^{\mathrm{sym}})\}
        _{j=1}^{n_{\pi}^{\mathrm f}+n_{\pi}^{\mathrm p}}
        \subset
        [0,\infty)\times\mathbb R^N_{\mathrm{sym}} ;
    \]

    \item The antisymmetric restriction has dimension
    $n_{\pi}^{\mathrm p}$ and admits real $V$-orthonormal eigenpairs
    \[
        \{(\lambda_j^{\mathrm{ant}},p_j^{\mathrm{ant}})\}
        _{j=1}^{n_{\pi}^{\mathrm p}}
        \subset
        (0,\infty)\times\mathbb R^N_{\mathrm{ant}} .
    \]
\end{enumerate}
In both cases, the eigenvalues are counted with their multiplicities.
\end{proposition}

\begin{proof}
We first prove the commutation relations stated above. Since
$R_\ell D_i=D_{\pi(i)}$ and $R_\ell$ is an isometry,
$|D_i|=|D_{\pi(i)}|$. Thus, for every $q\in\mathbb C^N$,
\[
    (V\Pi q)_i
    =
    |D_i|q_{\pi(i)}
    =
    |D_{\pi(i)}|q_{\pi(i)}
    =
    (\Pi Vq)_i .
\]
This proves $V\Pi=\Pi V$. Next, we prove the commutation relation for $C^0$. We use the same symbol
$\mathcal R_\ell$ for the induced reflection on boundary traces and densities.
The periodic Laplace Green kernel is reflection invariant:
$G^{0,0}(R_\ell z)=G^{0,0}(z)$. Hence, for
$Q_0(x,y):=G^{0,0}(x-y)$,
\[
    Q_0(R_\ell x,y)
    =
    G^{0,0}(R_\ell(x-R_\ell y))
    =
    Q_0(x,R_\ell y).
\]
By \cref{lem:reflection-covariance-kernel},
\begin{equation}
\label{eq:S00-reflection-covariance}
    \mathcal S_D^{0,0}\mathcal R_\ell[\psi]
    =
    \mathcal R_\ell\mathcal S_D^{0,0}[\psi],
    \qquad
    \psi\in H^{-1/2}(\partial D).
\end{equation}
Moreover, $\mathfrak m[\mathcal R_\ell[\psi]]=\mathfrak m[\psi]$, so
$\mathcal R_\ell$ preserves $H^{-1/2}_0(\partial D)$. Therefore,
\begin{equation}
\label{eq:H0-reflection-covariance}
    \mathcal H_0[\mathcal R_\ell[\psi],s]
    =
    \mathcal R_\ell\mathcal H_0[\psi,s],
    \qquad
    \psi\in H^{-1/2}_0(\partial D),
    \quad
    s\in\mathbb C .
\end{equation}
Using $(\psi_j^0,s_j^0)=\mathcal H_0^{-1}[\chi_{\partial D_j}]$ and
\eqref{eq:H0-reflection-covariance}, we obtain
\[
    \mathcal H_0[\mathcal R_\ell[\psi_j^0],s_j^0]
    =
    \chi_{\partial D_{\pi(j)}} .
\]
The uniqueness in \cref{lem:H0-isomorphism} gives
\begin{equation}
\label{eq:reflected-capacitance-density}
    \psi_{\pi(j)}^0=\mathcal R_\ell[\psi_j^0],
    \qquad
    s_{\pi(j)}^0=s_j^0 .
\end{equation}
Using \eqref{eq:capacitance-entries},
\eqref{eq:reflected-capacitance-density}, and the change of variables
$x=R_\ell y$, we get
\[
\begin{aligned}
    C^0_{\pi(i)\pi(j)}
    =
    -\int_{\partial D_{\pi(i)}}
        \psi_{\pi(j)}^0(x)
    \,d\sigma(x)
    =
    -\int_{\partial D_{\pi(i)}}
        \psi_j^0(R_\ell x)
    \,d\sigma(x)
    =
    -\int_{\partial D_i}
        \psi_j^0(y)
    \,d\sigma(y)
    =
    C^0_{ij}.
\end{aligned}
\]
Since $\Pi_{ij}=1$ when $j=\pi(i)$ and $\Pi_{ij}=0$
otherwise, the corresponding matrix
entries satisfy
\[
    (\Pi C^0\Pi)_{ij}
    =
    \sum_{m,n=1}^N
        (\Pi)_{im}C^0_{mn}(\Pi)_{nj}
    =
    C^0_{\pi(i)\pi(j)}
    =
    C^0_{ij}.
\]
Hence, $\Pi C^0\Pi=C^0$. Since $\Pi^2=I$, this is equivalent to
$C^0\Pi=\Pi C^0$.

The commutation relations show that the generalized eigenvalue problem preserves the real reflection
subspaces:
\[
    (C^0-\lambda V)\mathbb R^N_{\mathrm{sym}}
    \subset
    \mathbb R^N_{\mathrm{sym}},
    \qquad
    (C^0-\lambda V)\mathbb R^N_{\mathrm{ant}}
    \subset
    \mathbb R^N_{\mathrm{ant}}.
\]
It remains to count dimensions and signs. Since $\pi^2=I$, the orbits of
$\pi$ consist of fixed indices and two-element reflected pairs. Hence,
$N=n_{\pi}^{\mathrm f}+2n_{\pi}^{\mathrm p}$.
If $q\in\mathbb R^N_{\mathrm{ant}}$, then $\Pi q=-q$. For
$i\in\mathcal I_{\pi}^{\mathrm f}$, this gives $q_i=-q_i$, so $q_i=0$.
For a reflected pair $\{i,\pi(i)\}\in\mathcal I_{\pi}^{\mathrm p}$, it gives
$q_{\pi(i)}=-q_i$. Thus an antisymmetric vector is determined by one real
parameter on each reflected pair, and
\[
    \dim\mathbb R^N_{\mathrm{ant}}
    =
    n_{\pi}^{\mathrm p}.
\]
If $q\in\mathbb R^N_{\mathrm{sym}}$, then $\Pi q=q$. Each fixed index gives
one free real component, and each reflected pair satisfies
$q_{\pi(i)}=q_i$, and hence gives one free real parameter. Therefore,
\[
    \dim\mathbb R^N_{\mathrm{sym}}
    =
    n_{\pi}^{\mathrm f}
    +
    n_{\pi}^{\mathrm p}.
\]
On each invariant real subspace $E=\mathbb R^N_{\mathrm{sym}}$ or
$E=\mathbb R^N_{\mathrm{ant}}$, the restrictions of $C^0$ and $V$ are real
symmetric, and $V|_E$ is positive definite. The generalized spectral theorem
therefore gives a $V$-orthonormal eigenbasis of $E$, hence exactly $\dim E$
generalized eigenpairs, counted with their multiplicities.

Finally, \cref{lem:spectral-capacitance} gives $C^0\ge0$ and
$\ker C^0=\operatorname{span}\{\mathbf 1\}$. Since
$\Pi\mathbf 1=\mathbf 1$, this kernel lies in
$\mathbb R^N_{\mathrm{sym}}$, so
$\ker C^0\cap\mathbb R^N_{\mathrm{ant}}=\{0\}$.
Therefore, the symmetric restricted eigenvalues are nonnegative, while the
antisymmetric restricted eigenvalues are strictly positive. Together with the
dimension count, this gives the asserted eigenvalue counts, with their multiplicities.
\end{proof}

\begin{lemma}[First radiative correction on the antisymmetric subspace]
\label{lem:antisymmetric-C1-vanishing}
Under \cref{ass:periodic-reflection-symmetry} and at $a=0$, the first-order
capacitance correction $C^1$ satisfies $C^1q=0$ for every
$q\in\mathbb C^N_{\mathrm{ant}}$.
\end{lemma}

\begin{proof}
By \cref{def:capacitance-densities},
$m_{\ell,j}^a=a\cdot\mathfrak m_\ell[\psi_j^0]$, so $m_\ell^a=0$ at
$a=0$. Also, \eqref{eq:G1-parity-decomposition} gives
$G_{1,o}^{0,c}=0$, hence $\widetilde C^{1,o}=0$. Therefore,
\eqref{eq:C1-formula} reduces to
\[
    C^1
    =
    -\mathrm i
    \left(
        2\tau|Y|\,s^0(s^0)^\top
        +
        \frac{m_d^\tau(m_d^\tau)^\top}{2\tau|Y|}
    \right).
\]
By \eqref{eq:reflected-capacitance-density},
$s_{\pi(j)}^0=s_j^0$ and $m_{d,\pi(j)}^\tau=m_{d,j}^\tau$ for every $j$.
Thus, $s^0$ and $m_d^\tau$ are symmetric vectors. Hence,
\[
    (s^0)^\top q=0,
    \qquad
    (m_d^\tau)^\top q=0.
\]
The displayed formula for $C^1$ gives $C^1q=0$ for
$q\in\mathbb C^N_{\mathrm{ant}}$.
\end{proof}

\begin{remark}[Antisymmetric candidates for embedded eigenvalues]\label{rem:antisymmetric-candidates}
By \cref{prop:reduced-reflection-decomposition}, the leading reduced
eigenmodes split into symmetric and antisymmetric sectors. For an antisymmetric
mode $p_j^{\mathrm{ant}}$, \cref{lem:antisymmetric-C1-vanishing} gives
$C^1p_j^{\mathrm{ant}}=0$. Hence, the first radiative coefficient
$\hat\omega_j^1$ vanishes in the corresponding antisymmetric subspace. Consequently, the order-$\delta$ imaginary correction in
\eqref{eq:linear-resonance-expansion} is absent. The same cancellation removes
the corresponding order-$\delta$ imaginary correction in
\eqref{eq:modal-physical-frequency-expansion}. Thus, antisymmetric modes are
natural candidates for embedded eigenvalues and BICs.
\end{remark}

\subsection{Linear bound states in the continuum}
We now specialize to the linear medium at the symmetry point $\alpha=0$. The linear
variational form \eqref{eq:def-linear-variational-form} then becomes
\begin{equation}\label{eq:linear-reflection-form}
    a^{\rm lin}_{\omega,\delta}(u,v)
    =
    (\nabla u,\nabla v)_D
    -
    k_b^2(u,v)_D
    -
    \delta
    \langle
        \mathcal T_D^{0,k_m}[u],
        v
    \rangle_{\partial D}.
\end{equation}
The reduced antisymmetric eigenpairs from
\cref{prop:reduced-reflection-decomposition} give candidate BIC locations.
To make these candidates exact, we first work with the unscaled linear
resonance problem in $(\omega,\delta)$. Complex antisymmetric restricted
solutions lift to the full linear resonance problem, and real antisymmetric
solutions complexify to the complex restricted linear resonance problem. After these lifting
facts are established, we introduce the scaling
\eqref{eq:low-frequency-scaling} and construct the linear BIC branches.

\begin{lemma}[Reflection decomposition of the linear resonance problem]
\label{lem:exact-linear-reflection-decomposition}
Under \cref{ass:periodic-reflection-symmetry} and at $\alpha=0$, the
linear variational form \eqref{eq:linear-reflection-form} is reflection invariant:
\[
    a^{\rm lin}_{\omega,\delta}(\mathcal R_\ell[u],\mathcal R_\ell[v])
    =
    a^{\rm lin}_{\omega,\delta}(u,v),
    \qquad
    u,v\in H^1(D;\mathbb C).
\]
It is therefore block diagonal with respect to the reflection decomposition:
\[
    a^{\rm lin}_{\omega,\delta}(u_{\mathrm{sym}},v_{\mathrm{ant}})
    =
    a^{\rm lin}_{\omega,\delta}(u_{\mathrm{ant}},v_{\mathrm{sym}})
    =
    0
\]
for all
$u_{\mathrm{sym}},v_{\mathrm{sym}}\in H^1_{\mathrm{sym}}(D;\mathbb C)$
and
$u_{\mathrm{ant}},v_{\mathrm{ant}}\in H^1_{\mathrm{ant}}(D;\mathbb C)$.
In terms of solutions, if the pair $(\omega,u_{\mathrm{ant}})\in \mathbb C \times H^1_{\mathrm{ant}}(D;\mathbb C)$ solves the complex
antisymmetric restricted linear resonance problem
\begin{equation}\label{eq:exact-linear-antisymmetric-problem}
    a^{\rm lin}_{\omega,\delta}(u_{\mathrm{ant}},v_{\mathrm{ant}})=0,
    \qquad
    v_{\mathrm{ant}}\in H^1_{\mathrm{ant}}(D;\mathbb C),
\end{equation}
then the same pair $(\omega,u_{\mathrm{ant}})$ solves the full linear resonance problem:
\begin{equation}\label{eq:exact-linear-full-problem}
    a^{\rm lin}_{\omega,\delta}(u_{\mathrm{ant}},v)=0,
    \qquad
    v\in H^1(D;\mathbb C).
\end{equation}
\end{lemma}

\begin{proof}
We first prove the reflection invariance of the form. For the volume terms,
$(\mathcal R_\ell[u])(x)=u(R_\ell x)$, and the chain rule gives
\[
    \nabla(\mathcal R_\ell[u])(x)
    =
    R_\ell^\top \nabla u(R_\ell x).
\]
Since $R_\ell^\top R_\ell=I$ and $R_\ell D=D$, the change of variables
$x\mapsto R_\ell x$ yields
\[
    (\mathcal R_\ell[u],\mathcal R_\ell[v])_D=(u,v)_D,
    \qquad
    (\nabla\mathcal R_\ell[u],\nabla\mathcal R_\ell[v])_D
    =
    (\nabla u,\nabla v)_D .
\]
It remains to check the DtN term. At $\alpha=0$, the quasiperiodic Green
function satisfies $G^{0,k}(R_\ell z)=G^{0,k}(z)$. Using
$R_\ell^2=I$, the single-layer kernel $Q_S(x,y):=G^{0,k}(x-y)$
satisfies
\[
    Q_S(R_\ell x,y)
    =
    G^{0,k}(R_\ell x-y)
    =
    G^{0,k}(R_\ell(x-R_\ell y))
    =
    Q_S(x,R_\ell y).
\]
By \cref{lem:reflection-covariance-kernel}, we have
\begin{equation}\label{eq:Sk-reflection-covariance}
    \mathcal S_D^{0,k}[\mathcal R_\ell[\psi]]
    =
    \mathcal R_\ell[\mathcal S_D^{0,k}[\psi]],
    \qquad
    \psi\in H^{-1/2}(\partial D).
\end{equation}
For the Neumann--Poincare kernel
$Q_K(x,y):=\partial_{\nu(x)}G^{0,k}(x-y)$, the identities
$\nu(R_\ell x)=R_\ell\nu(x)$ and
$\nabla G^{0,k}(R_\ell z)=R_\ell\nabla G^{0,k}(z)$ give
\[
\begin{aligned}
    Q_K(R_\ell x,y)
    =
    (R_\ell\nu(x))\cdot\nabla G^{0,k}(R_\ell x-y)
    =
    (R_\ell\nu(x))\cdot R_\ell\nabla G^{0,k}(x-R_\ell y)
    =
    Q_K(x,R_\ell y).
\end{aligned}
\]
Applying \cref{lem:reflection-covariance-kernel} in the principal-value sense gives
\begin{equation}\label{eq:Kstar-reflection-covariance}
    (\mathcal K_D^{0,k})^*[\mathcal R_\ell[\psi]]
    =
    \mathcal R_\ell[(\mathcal K_D^{0,k})^*[\psi]],
    \qquad
    \psi\in H^{-1/2}(\partial D).
\end{equation}
When $\mathcal S_D^{0,k}$ is invertible,
\eqref{eq:Sk-reflection-covariance} also gives the covariance of its inverse.
Combining this with the DtN representation \eqref{eq:DtN-representation} and
\eqref{eq:Kstar-reflection-covariance}, we obtain
\begin{equation}\label{eq:DtN-reflection-covariance}
    \mathcal T_D^{0,k}[\mathcal R_\ell[\psi]]
    =
    \mathcal R_\ell[\mathcal T_D^{0,k}[\psi]],
    \qquad
    \psi\in H^{1/2}(\partial D).
\end{equation}
Therefore,
\[
    \langle
        \mathcal T_D^{0,k}[\mathcal R_\ell[u]],
        \mathcal R_\ell[v]
    \rangle_{\partial D}
    =
    \langle
        \mathcal T_D^{0,k}[u],
        v
    \rangle_{\partial D}.
\]
Together with the volume identities above, this proves the reflection
invariance of $a^{\rm lin}_{\omega,\delta}$.

The block diagonal property follows from this invariance. If
$u_{\mathrm{sym}}$ is symmetric and $v_{\mathrm{ant}}$ is antisymmetric,
then
\[
\begin{aligned}
    a^{\rm lin}_{\omega,\delta}(u_{\mathrm{sym}},v_{\mathrm{ant}})
    =
    a^{\rm lin}_{\omega,\delta}
    (
        \mathcal R_\ell[u_{\mathrm{sym}}],
        \mathcal R_\ell[v_{\mathrm{ant}}]
    )
    =
    a^{\rm lin}_{\omega,\delta}(u_{\mathrm{sym}},-v_{\mathrm{ant}})
    =
    -a^{\rm lin}_{\omega,\delta}(u_{\mathrm{sym}},v_{\mathrm{ant}}).
\end{aligned}
\]
Thus, this mixed term is zero. The same argument gives
$a^{\rm lin}_{\omega,\delta}(u_{\mathrm{ant}},v_{\mathrm{sym}})=0$.

Finally, let $v\in H^1(D;\mathbb C)$ and decompose
$v=v_{\mathrm{sym}}+v_{\mathrm{ant}}$. If $u_{\mathrm{ant}}$ solves the
restricted antisymmetric linear resonance problem \eqref{eq:exact-linear-antisymmetric-problem}, then
\[
    a^{\rm lin}_{\omega,\delta}(u_{\mathrm{ant}},v)
    =
    a^{\rm lin}_{\omega,\delta}(u_{\mathrm{ant}},v_{\mathrm{ant}})
    +
    a^{\rm lin}_{\omega,\delta}(u_{\mathrm{ant}},v_{\mathrm{sym}})
    =
    0.
\]
Thus, the pair $(\omega,u_{\mathrm{ant}})$
solves the full linear resonance problem \eqref{eq:exact-linear-full-problem}.
\end{proof}

The previous lemma lifts complex antisymmetric restricted solutions to the full
linear resonance problem. To construct BIC branches, we next restrict ourselves to real
frequencies $\omega\in(0,\omega_0)$ and real antisymmetric fields. Then
$k_m=\omega/c_m$ is real. At $\alpha=0$, the subwavelength regime leaves only
the constant $\eta=0$ Rayleigh--Bloch order as a propagating channel; the
antisymmetry of the exact branch will later force this channel to vanish. The
next proposition gives the real antisymmetric restriction of the linear
resonance problem.

\begin{proposition}[Real antisymmetric linear resonance problem]
\label{prop:real-antisymmetric-linear-resonance-problem}
Under \cref{ass:periodic-reflection-symmetry}, let
$\omega\in(0,\omega_0)$, $a=0$, and $k_m=\omega/c_m$. Then
\[
    \mathcal T_D^{0,k_m}:
    H^{1/2}_{\mathrm{ant}}(\partial D;\mathbb R)
    \longrightarrow
    H^{-1/2}_{\mathrm{ant}}(\partial D;\mathbb R).
\]
Consequently,
\[
    a^{\rm lin}_{\omega,\delta}(u_{\mathrm{ant}},v_{\mathrm{ant}})
    \in\mathbb R,
    \qquad
    u_{\mathrm{ant}},v_{\mathrm{ant}}
    \in H^1_{\mathrm{ant}}(D;\mathbb R).
\]
If $(\omega,u_{\mathrm{ant}})\in
\mathbb R\times H^1_{\mathrm{ant}}(D;\mathbb R)$ solves the real
antisymmetric restricted linear resonance problem
\begin{equation}\label{eq:exact-linear-real-antisymmetric-problem}
    a^{\rm lin}_{\omega,\delta}(u_{\mathrm{ant}},v_{\mathrm{ant}})=0,
    \qquad
    v_{\mathrm{ant}}\in H^1_{\mathrm{ant}}(D;\mathbb R),
\end{equation}
then the same pair $(\omega,u_{\mathrm{ant}})$ also solves the complex
antisymmetric restricted linear resonance problem
\eqref{eq:exact-linear-antisymmetric-problem}. Consequently, by
\cref{lem:exact-linear-reflection-decomposition}, the same pair solves the linear resonance problem \eqref{eq:exact-linear-full-problem}.
\end{proposition}

\begin{proof}
In the subwavelength regime at $a=0$, the outgoing periodic Green function
splits into the single propagating term and a real evanescent series:
\[
    G^{0,k_m}(x-y)
    =
    \frac{e^{\mathrm i k_m|x_d-y_d|}}{2\mathrm i|Y|k_m}
    -
    \sum_{\eta\in\Lambda^*\setminus\{0\}}
    \frac{
        e^{\mathrm i\eta\cdot(x_\ell-y_\ell)}
        e^{-\sqrt{|\eta|^2-k_m^2}\,|x_d-y_d|}
    }{
        2|Y|\sqrt{|\eta|^2-k_m^2}
    } .
\]
Let $\varphi\in H^{-1/2}_{\mathrm{ant}}(\partial D;\mathbb R)$. The
$\eta=0$ term is even in the reflected variable and therefore has zero
pairing with the antisymmetric density. In the remaining sum, the terms
corresponding to $\eta$ and $-\eta$ combine into a real kernel. Hence,
\[
    \mathcal S_D^{0,k_m}:
    H^{-1/2}_{\mathrm{ant}}(\partial D;\mathbb R)
    \longrightarrow
    H^{1/2}_{\mathrm{ant}}(\partial D;\mathbb R).
\]
The same argument applies to $(\mathcal K_D^{0,k_m})^*$, giving
\[
    (\mathcal K_D^{0,k_m})^*:
    H^{-1/2}_{\mathrm{ant}}(\partial D;\mathbb R)
    \longrightarrow
    H^{-1/2}_{\mathrm{ant}}(\partial D;\mathbb R).
\]
Combining these mapping properties with the DtN representation
\eqref{eq:DtN-representation} yields the asserted mapping property of
$\mathcal T_D^{0,k_m}$.

Now, let
$u_{\mathrm{ant}},v_{\mathrm{ant}}\in
H^1_{\mathrm{ant}}(D;\mathbb R)$. The volume terms in
\eqref{eq:linear-reflection-form} are real. By the mapping property just
proved, $\mathcal T_D^{0,k_m}[u_{\mathrm{ant}}]$ is a real antisymmetric
Neumann trace, and its duality pairing with the real trace of
$v_{\mathrm{ant}}$ is real. Therefore,
$a^{\rm lin}_{\omega,\delta}(u_{\mathrm{ant}},v_{\mathrm{ant}})\in\mathbb R$.

Finally, if $u_{\mathrm{ant}}$ satisfies
\eqref{eq:exact-linear-real-antisymmetric-problem} and
$v_{\mathrm{ant}}\in H^1_{\mathrm{ant}}(D;\mathbb C)$, write
\[
    v_{\mathrm{ant}}=v_1+\mathrm i v_2,
    \qquad
    v_1,v_2\in H^1_{\mathrm{ant}}(D;\mathbb R).
\]
Since the form is conjugate-linear in the test variable,
\[
    a^{\rm lin}_{\omega,\delta}(u_{\mathrm{ant}},v_{\mathrm{ant}})
    =
    a^{\rm lin}_{\omega,\delta}(u_{\mathrm{ant}},v_1)
    -
    \mathrm i\,a^{\rm lin}_{\omega,\delta}(u_{\mathrm{ant}},v_2)
    =
    0.
\]
This proves the complexification statement. The assertion for the full linear
resonance problem follows from \cref{lem:exact-linear-reflection-decomposition}.
\end{proof}

The preceding lemma and proposition separate the exact symmetry mechanism from
the small-contrast construction. They show that, for real frequencies, a real
antisymmetric solution of the restricted linear resonance problem is already a
solution of the full linear resonance problem. We now return to the reduced subwavelength
equation and introduce the scaling
$\delta=\varepsilon^2$, $\omega=c_b\varepsilon\hat\omega$. The next theorem
uses a real Lyapunov--Schmidt argument in the antisymmetric subspace to construct
such solutions, and then applies the restricted-to-full lifting results to identify them as
linear BICs.

\begin{theorem}[Symmetry classification of linear subwavelength branches]
\label{thm:linear-bic-symmetry-classification}
Assume \cref{ass:periodic-reflection-symmetry} and work in
\cref{ass:subwavelength-low-frequency} with $a=0$. By
\cref{prop:reduced-reflection-decomposition}, the linear subwavelength
branches in \cref{thm:linear-simple-resonance-branch} split into symmetric and
antisymmetric classes. Assume also that the modes considered below are simple
in their respective reflection subspaces.

\textbf{Symmetric class:} There are
$n_{\pi}^{\mathrm f}+n_{\pi}^{\mathrm p}$ symmetric branches such that, for
$1\le j\le n_{\pi}^{\mathrm f}+n_{\pi}^{\mathrm p}$,
\[
    (\hat\omega_{j,\mathrm{sym}}^{\rm lin},q_{j,\mathrm{sym}}^{\rm lin})
    \in
    \mathbb C\times\mathbb C^N_{\mathrm{sym}},
    \qquad
    (\hat\omega_{j,\mathrm{sym}}^{\rm lin},q_{j,\mathrm{sym}}^{\rm lin})
    \quad\text{near}\quad
    (\sqrt{\lambda_j^{\mathrm{sym}}},p_j^{\mathrm{sym}}),
    \qquad
    u_{j,\mathrm{sym}}^{\rm lin}\in H^1_{\mathrm{sym}}(D;\mathbb C).
\]
The corresponding frequency and interior-field expansions are those in
\eqref{eq:linear-resonance-expansion} and
\eqref{eq:linear-physical-field-expansion}.

\textbf{Antisymmetric class:} There are $n_{\pi}^{\mathrm p}$ antisymmetric
branches such that, for $1\le j\le n_{\pi}^{\mathrm p}$,
\[
    (\hat\omega_{j,\mathrm{ant}}^{\rm lin},q_{j,\mathrm{ant}}^{\rm lin})
    \in
    \mathbb R\times\mathbb R^N_{\mathrm{ant}},
    \qquad
    (\hat\omega_{j,\mathrm{ant}}^{\rm lin},q_{j,\mathrm{ant}}^{\rm lin})
    \quad\text{near}\quad
    (\sqrt{\lambda_j^{\mathrm{ant}}},p_j^{\mathrm{ant}}),
    \qquad
    u_{j,\mathrm{ant}}^{\rm lin}\in H^1_{\mathrm{ant}}(D;\mathbb R).
\]
For these antisymmetric branches, with $\delta=\varepsilon^2$,
\[
    \omega_{j,\mathrm{ant}}^{\rm lin}(\delta)
    =
    c_b\sqrt{\lambda_j^{\mathrm{ant}}}\sqrt\delta
    +
    \mathcal O(\delta^{3/2})
    \in\mathbb R,
    \qquad
    u_{j,\mathrm{ant}}^{\rm lin}(\delta)
    =
    u_{p_j^{\mathrm{ant}}}
    +
    \mathcal O(\delta)
    \quad\text{in }H^1(D).
\]
Consequently, for all sufficiently small $\delta>0$, each antisymmetric
branch yields an exact linear BIC. Its propagating Rayleigh coefficients vanish,
so the outgoing continuation has no propagating Rayleigh--Bloch mode and is
exponentially localized; the frequency
$\omega_{j,\mathrm{ant}}^{\rm lin}(\delta)$ is real and embedded in the
radiation continuum.
\end{theorem}

\begin{proof}
We first indicate the strategy. The symmetric class is obtained by the same
finite-dimensional reduction restricted to $\mathbb C^N_{\mathrm{sym}}$, so
we omit the details. We prove the antisymmetric embedded branches by restricting
the reduced problem to $\mathbb R^N_{\mathrm{ant}}$. On this real subspace, the
residual is real, the first radiative term is absent, and the branch follows
from the real implicit function theorem. Fix an antisymmetric mode and set, for
this proof,
\[
    p_j:=p_j^{\mathrm{ant}},
    \qquad
    \lambda_j:=\lambda_j^{\mathrm{ant}},
    \qquad
    \hat\omega_j^0:=\sqrt{\lambda_j}.
\]
We work with real variables in $\mathbb R^N_{\mathrm{ant}}$ and define
\[
    E_{j,\mathrm{ant}}^\perp
    :=
    \{\xi\in\mathbb R^N_{\mathrm{ant}}:p_j^\top V\xi=0\}.
\]
For $q\in\mathbb R^N_{\mathrm{ant}}$, real $\hat\omega$, and
$\alpha=0$, the linear $\mathcal Z(D)$-equation in
\cref{prop:linear-projected-reduction} is invariant under reflection and
complex conjugation. By uniqueness, its solution
$z^{\rm lin}(q,\hat\omega,\varepsilon)$ belongs to
$\mathcal Z(D)\cap H^1_{\mathrm{ant}}(D;\mathbb R)$. Thus, the real
antisymmetric subspace is invariant for the reduced residual, and
$\mathcal F^{\rm lin}$ restricts to a real map
\[
    \mathcal F_{\mathrm{ant}}^{\rm lin}
    :
    \mathbb R^N_{\mathrm{ant}}\times\mathbb R\times[0,\varepsilon_0)
    \longrightarrow
    \mathbb R^N_{\mathrm{ant}},
    \qquad
    \mathcal F_{\mathrm{ant}}^{\rm lin}
    :=
    \mathcal F^{\rm lin}\big|_{\mathbb R^N_{\mathrm{ant}}\times\mathbb R}.
\]
By \cref{lem:antisymmetric-C1-vanishing}, $C^1q=0$ for every
$q\in\mathbb R^N_{\mathrm{ant}}$. Thus, the
$c_b\varepsilon\hat\omega C^1q$ term in
\eqref{eq:reduced-linear-resonance-equation} is absent. Combining this with
\eqref{eq:linear-R-factorization}, we may write
\[
    \mathcal F_{\mathrm{ant}}^{\rm lin}(q,\hat\omega,\varepsilon)
    =
    \hat\omega^2Vq-C^0q
    +
    \varepsilon^2
    \widetilde R_{\mathrm{ant}}^{\rm lin}(q,\hat\omega,\varepsilon),
\]
where $\widetilde R_{\mathrm{ant}}^{\rm lin}$ is real-smooth and linear in
$q$ for fixed $(\hat\omega,\varepsilon)$.

We now solve the restricted finite-dimensional equation near
$(p_j,\hat\omega_j^0,0)$. Define
\[
    \mathcal G_j(\xi,\hat\omega,\varepsilon)
    :=
    \mathcal F_{\mathrm{ant}}^{\rm lin}(p_j+\xi,\hat\omega,\varepsilon),
    \qquad
    \xi\in E_{j,\mathrm{ant}}^\perp .
\]
Then $\mathcal G_j(0,\hat\omega_j^0,0)=0$. Its derivative with respect to
$(\xi,\hat\omega)$ at this point is
\[
    L_{\mathrm{ant}}[\widetilde\xi,\widetilde{\hat\omega}]
    =
    (\lambda_jV-C^0)\widetilde\xi
    +
    2\hat\omega_j^0\widetilde{\hat\omega}Vp_j .
\]
Consider this linear map from
$E_{j,\mathrm{ant}}^\perp\times\mathbb R$ into
$\mathbb R^N_{\mathrm{ant}}$. If
$L_{\mathrm{ant}}[\widetilde\xi,\widetilde{\hat\omega}]=0$, then
$p_j^\top(\lambda_jV-C^0)=0$, because
$C^0p_j=\lambda_jVp_j$ and $C^0$ is symmetric. Using also
$p_j^\top Vp_j=1$, we obtain
\[
\begin{aligned}
    0
    =
    p_j^\top
    L_{\mathrm{ant}}[\widetilde\xi,\widetilde{\hat\omega}]
    =
    p_j^\top(\lambda_jV-C^0)\widetilde\xi
    +
    2\hat\omega_j^0\widetilde{\hat\omega}\,p_j^\top Vp_j
    =
    2\hat\omega_j^0\widetilde{\hat\omega}.
\end{aligned}
\]
Hence, $\widetilde{\hat\omega}=0$, since $\hat\omega_j^0>0$. Then
$(\lambda_jV-C^0)\widetilde\xi=0$. Since $\lambda_j$ is simple in the
antisymmetric subspace, $\widetilde\xi\in\operatorname{span}\{p_j\}$, and
the constraint $p_j^\top V\widetilde\xi=0$ gives $\widetilde\xi=0$.
Since the domain and codomain have the same finite dimension, injectivity
implies that $L_{\mathrm{ant}}$ is an isomorphism. The real implicit function
theorem gives real-smooth functions
\[
    \xi_j(\varepsilon)\in E_{j,\mathrm{ant}}^\perp,
    \qquad
    \hat\omega_{j,\mathrm{ant}}^{\rm lin}(\varepsilon)\in\mathbb R
\]
with $\xi_j(0)=0$ and
$\hat\omega_{j,\mathrm{ant}}^{\rm lin}(0)=\hat\omega_j^0$, such that
\[
    \mathcal F_{\mathrm{ant}}^{\rm lin}
    \bigl(
        p_j+\xi_j(\varepsilon),
        \hat\omega_{j,\mathrm{ant}}^{\rm lin}(\varepsilon),
        \varepsilon
    \bigr)
    =0.
\]
Because the residual differs from $\hat\omega^2Vq-C^0q$ by
$\varepsilon^2$ times a smooth term, differentiating this identity at
$\varepsilon=0$ gives
$L_{\mathrm{ant}}[\xi_j'(0),(\hat\omega_{j,\mathrm{ant}}^{\rm lin})'(0)]=0$.
Thus,
\[
    \xi_j(\varepsilon)=\mathcal O(\varepsilon^2),
    \qquad
    \hat\omega_{j,\mathrm{ant}}^{\rm lin}(\varepsilon)
    =
    \hat\omega_j^0+\mathcal O(\varepsilon^2).
\]
Setting
$q_{j,\mathrm{ant}}^{\rm lin}(\varepsilon):=p_j+\xi_j(\varepsilon)$,
$\omega_{j,\mathrm{ant}}^{\rm lin}(\varepsilon)
:=c_b\varepsilon\hat\omega_{j,\mathrm{ant}}^{\rm lin}(\varepsilon)$, and
$\delta=\varepsilon^2$, we obtain
\[
    \omega_{j,\mathrm{ant}}^{\rm lin}(\delta)
    =
    c_b\sqrt{\lambda_j^{\mathrm{ant}}}\sqrt\delta
    +
    \mathcal O(\delta^{3/2}).
\]

Finally, we reconstruct the field and verify the BIC property. Define the
interior field by
\[
    u_{j,\mathrm{ant}}^{\rm lin}
    :=
    u_{q_{j,\mathrm{ant}}^{\rm lin}}
    +
    z^{\rm lin}
    \bigl(
        q_{j,\mathrm{ant}}^{\rm lin},
        \hat\omega_{j,\mathrm{ant}}^{\rm lin},
        \varepsilon
    \bigr).
\]
The boundedness of the piecewise-constant lift, the estimate
\eqref{eq:linear-basic-z-estimate}, and
$q_{j,\mathrm{ant}}^{\rm lin}=p_j+\mathcal O(\varepsilon^2)$ give
\[
    u_{j,\mathrm{ant}}^{\rm lin}
    =
    u_{p_j^{\mathrm{ant}}}
    +
    \mathcal O(\varepsilon^2)
    =
    u_{p_j^{\mathrm{ant}}}
    +
    \mathcal O(\delta)
    \quad\text{in }H^1(D).
\]
The $\mathcal Z(D)$-equation together with
$\mathcal F_{\mathrm{ant}}^{\rm lin}=0$ gives the real antisymmetric
variational equation. By \cref{prop:real-antisymmetric-linear-resonance-problem}, the same interior
field solves the full linear resonance problem
\eqref{eq:exact-linear-full-problem}.
Let the same symbol denote its outgoing exterior continuation. Reflection
covariance of the exterior Dirichlet problem gives
$\mathcal R_\ell[u_{j,\mathrm{ant}}^{\rm lin}]
=-u_{j,\mathrm{ant}}^{\rm lin}$. Therefore,
\[
    \bigl(u_{j,\mathrm{ant}}^{\rm lin}\bigr)^0_0(\pm h)
    =
    \frac1{|Y|}\int_Y
    u_{j,\mathrm{ant}}^{\rm lin}(x_\ell,\pm h)\,dx_\ell
    =
    -\bigl(u_{j,\mathrm{ant}}^{\rm lin}\bigr)^0_0(\pm h).
\]
Thus, the propagating $\eta=0$ Rayleigh--Bloch coefficient vanishes. In the
subwavelength regime, this is the only propagating channel, so no propagating
mode remains. Set
$k_{m,j}^{\rm lin}(\delta):=
\omega_{j,\mathrm{ant}}^{\rm lin}(\delta)/c_m$. The outgoing exterior
continuation is therefore
\[
    u_{j,\mathrm{ant}}^{\rm lin}(x_\ell,x_d)
    =
    \begin{cases}
        \displaystyle
        \sum_{\eta\in\Lambda^*\setminus\{0\}}
        \bigl(u_{j,\mathrm{ant}}^{\rm lin}\bigr)^0_\eta(+h)
        e^{\mathrm i\eta\cdot x_\ell}
        e^{-\sqrt{|\eta|^2-\bigl(k_{m,j}^{\rm lin}(\delta)\bigr)^2}\,(x_d-h)},
        & x_d\ge h,\\[2mm]
        \displaystyle
        \sum_{\eta\in\Lambda^*\setminus\{0\}}
        \bigl(u_{j,\mathrm{ant}}^{\rm lin}\bigr)^0_\eta(-h)
        e^{\mathrm i\eta\cdot x_\ell}
        e^{+\sqrt{|\eta|^2-\bigl(k_{m,j}^{\rm lin}(\delta)\bigr)^2}\,(x_d+h)},
        & x_d\le -h .
    \end{cases}
\]
Since $k_{m,j}^{\rm lin}(\delta)<|\eta|$ for every
$\eta\in\Lambda^*\setminus\{0\}$ in the subwavelength regime, the exterior
field decays exponentially as $|x_d|\to\infty$. Also
$\omega_{j,\mathrm{ant}}^{\rm lin}(\varepsilon)>0$ for all sufficiently small
$\varepsilon>0$, so the frequency is embedded in the continuum. The
construction applies to each simple antisymmetric mode, giving the stated
$n_\pi^{\mathrm p}$ antisymmetric branches.
\end{proof}

\subsection{Nonlinear bound states in the continuum}

We now construct nonlinear BICs. The only additional symmetry
assumption is that the Kerr coefficient respects the reflection symmetry. Under this
assumption, real antisymmetric solutions of the restricted nonlinear resonance problem
solve the full nonlinear resonance problem, and the modal continuation theorem can be applied
inside $\mathbb R^N_{\mathrm{ant}}$.

\begin{assumption}[Reflection-symmetric Kerr coefficient]
\label{ass:reflection-symmetric-kerr}
In addition to \cref{ass:periodic-reflection-symmetry}, assume that the Kerr
coefficient is invariant under the component permutation:
\[
    \sigma_{\pi(i)}=\sigma_i,
    \qquad
    i=1,\ldots,N.
\]
Equivalently, $\sigma_D\circ R_\ell=\sigma_D$.
\end{assumption}

\begin{proposition}[Antisymmetric lifting for the nonlinear resonance problem]
\label{prop:nonlinear-antisymmetric-lifting}
Under \cref{ass:periodic-reflection-symmetry,ass:reflection-symmetric-kerr},
let $a=0$ and $\omega\in(0,\omega_0)$. If
$(\omega,u_{\mathrm{ant}})\in
\mathbb R\times H^1_{\mathrm{ant}}(D;\mathbb R)$ satisfies
\begin{equation}
\label{eq:exact-nonlinear-real-antisymmetric-problem}
    a^{\mathrm{non}}_{\omega,\delta}
    (u_{\mathrm{ant}};v_{\mathrm{ant}})
    =0,
    \qquad
    v_{\mathrm{ant}}\in H^1_{\mathrm{ant}}(D;\mathbb R),
\end{equation}
then the same equality holds for all
$v_{\mathrm{ant}}\in H^1_{\mathrm{ant}}(D;\mathbb C)$, and the full nonlinear resonance identity
\begin{equation}
\label{eq:exact-nonlinear-full-problem}
    a^{\mathrm{non}}_{\omega,\delta}
    (u_{\mathrm{ant}};v)
    =0,
    \qquad
    v\in H^1(D;\mathbb C).
\end{equation}
\end{proposition}

\begin{proof}
We first prove reflection invariance of the nonlinear residual. The linear part
is reflection invariant by \cref{lem:exact-linear-reflection-decomposition}. The
Kerr coefficient satisfies $\sigma_D\circ R_\ell=\sigma_D$, and hence
$\mathcal N_\sigma\bigl[\mathcal R_\ell[u]\bigr]
    =
    \mathcal R_\ell\bigl[\mathcal N_\sigma[u]\bigr]$.
The change of variables $x\mapsto R_\ell x$ gives
\[
    \bigl(
        \mathcal N_\sigma\bigl[\mathcal R_\ell[u]\bigr],
        \mathcal R_\ell[v]
    \bigr)_D
    =
    \bigl(
        \mathcal N_\sigma[u],v
    \bigr)_D.
\]
Together with the linear invariance, this proves that
\[
    a^{\mathrm{non}}_{\omega,\delta}
    (\mathcal R_\ell[u];\mathcal R_\ell[v])
    =
    a^{\mathrm{non}}_{\omega,\delta}(u;v).
\]

Next, we derive the mixed cancellation. Let
$u_{\mathrm{ant}}\in H^1_{\mathrm{ant}}(D;\mathbb C)$. Since
$\mathcal R_\ell[u_{\mathrm{ant}}]=-u_{\mathrm{ant}}$, the first equality below
follows from $\sigma_D\circ R_\ell=\sigma_D$, while the last follows from
the oddness of the cubic Kerr map:
\[
    \mathcal R_\ell\bigl[\mathcal N_\sigma[u_{\mathrm{ant}}]\bigr]
    =
    \mathcal N_\sigma\bigl[\mathcal R_\ell[u_{\mathrm{ant}}]\bigr]
    =
    \mathcal N_\sigma[-u_{\mathrm{ant}}]
    =
    -\mathcal N_\sigma[u_{\mathrm{ant}}].
\]
Thus, $\mathcal N_\sigma[u_{\mathrm{ant}}]$ is antisymmetric. Hence, for every
$v_{\mathrm{sym}}\in H^1_{\mathrm{sym}}(D;\mathbb C)$,
\[
    \bigl(
        \mathcal N_\sigma[u_{\mathrm{ant}}],
        v_{\mathrm{sym}}
    \bigr)_D
    =0,
\]
because the integrand is odd under $R_\ell$. The linear mixed term vanishes by
the block diagonal property in \cref{lem:exact-linear-reflection-decomposition}.
Thus,
\[
    a^{\mathrm{non}}_{\omega,\delta}
    (u_{\mathrm{ant}};v_{\mathrm{sym}})
    =
    a^{\rm lin}_{\omega,\delta}
    (u_{\mathrm{ant}},v_{\mathrm{sym}})
    -
    k_b^2
    \bigl(
        \mathcal N_\sigma[u_{\mathrm{ant}}],
        v_{\mathrm{sym}}
    \bigr)_D
    =
    0.
\]

We now complexify the restricted nonlinear resonance equation. Suppose that
$(\omega,u_{\mathrm{ant}})\in
\mathbb R\times H^1_{\mathrm{ant}}(D;\mathbb R)$ satisfies
\eqref{eq:exact-nonlinear-real-antisymmetric-problem}. The linear part
complexifies by \cref{prop:real-antisymmetric-linear-resonance-problem}. The nonlinear term
complexifies in the same way because
$\mathcal N_\sigma[u_{\mathrm{ant}}]$ is real antisymmetric. Thus, writing
$v_{\mathrm{ant}}=v_1+\mathrm i v_2$, with
$v_1,v_2\in H^1_{\mathrm{ant}}(D;\mathbb R)$, gives
\[
    a^{\mathrm{non}}_{\omega,\delta}
    (u_{\mathrm{ant}};v_{\mathrm{ant}})
    =
    a^{\mathrm{non}}_{\omega,\delta}(u_{\mathrm{ant}};v_1)
    -
    \mathrm i\,
    a^{\mathrm{non}}_{\omega,\delta}(u_{\mathrm{ant}};v_2)
    =
    0.
\]
Hence, the same equality holds for every complex antisymmetric test function.

Finally, let $v\in H^1(D;\mathbb C)$ and write
$v=v_{\mathrm{sym}}+v_{\mathrm{ant}}$. The antisymmetric part vanishes by the
complexified restricted nonlinear resonance equation, and the symmetric part vanishes by the mixed
cancellation above. This proves
\eqref{eq:exact-nonlinear-full-problem}.
\end{proof}

We now combine this lifting result with the modal continuation theorem.
Symmetric branches are the usual nonlinear modal continuations. Antisymmetric
branches are constructed in $\mathbb R^N_{\mathrm{ant}}$ and then lifted to
full nonlinear BICs by \cref{prop:nonlinear-antisymmetric-lifting}.

\begin{theorem}[Symmetry classification of nonlinear subwavelength branches]
\label{thm:nonlinear-bic-symmetry-classification}
Assume
\cref{ass:periodic-reflection-symmetry,ass:reflection-symmetric-kerr,ass:subwavelength-low-frequency}
and set $a=0$. By the reflection decomposition and
\cref{thm:nonlinear-modal-continuation}, the nonlinear subwavelength branches
at $\alpha=0$ split into symmetric and antisymmetric classes. Assume that the
modes considered below are simple in their respective reflection subspaces.

All branches below are locally parametrized by the amplitude $t$ and the
scale $\varepsilon$: for each indicated index $j$, there are
$t_j,\varepsilon_j>0$ such that the branch is defined for
$0<|t|<t_j$ and $0<\varepsilon<\varepsilon_j$.

\textbf{Symmetric class:} There are
$n_{\pi}^{\mathrm f}+n_{\pi}^{\mathrm p}$ symmetric branches such that, for
$1\le j\le n_{\pi}^{\mathrm f}+n_{\pi}^{\mathrm p}$,
\[
    (\hat\omega_{j,\mathrm{sym}}^{\rm non},q_{j,\mathrm{sym}}^{\rm non})
    \in\mathbb C\times\mathbb C^N_{\mathrm{sym}},
    \qquad
    (\hat\omega_{j,\mathrm{sym}}^{\rm non},q_{j,\mathrm{sym}}^{\rm non})
    \quad\text{near}\quad
    (\sqrt{\lambda_j^{\mathrm{sym}}},0),
    \qquad
    u_{j,\mathrm{sym}}^{\rm non}\in H^1_{\mathrm{sym}}(D;\mathbb C).
\]
These are the nonlinear modal resonance branches from
\cref{thm:nonlinear-modal-continuation}. For the positive symmetric branches,
the scaled-frequency, reduced-vector, and interior-field expansions are those in
\eqref{eq:modal-scaled-frequency-expansion}, \eqref{eq:modal-q-expansion},
and \eqref{eq:modal-nonlinear-field-expansion}.

\textbf{Antisymmetric class:} There are $n_{\pi}^{\mathrm p}$ real
antisymmetric branches such that, for $1\le j\le n_{\pi}^{\mathrm p}$,
\[
    (\hat\omega_{j,\mathrm{ant}}^{\rm non},q_{j,\mathrm{ant}}^{\rm non})
    \in\mathbb R\times\mathbb R^N_{\mathrm{ant}},
    \qquad
    (\hat\omega_{j,\mathrm{ant}}^{\rm non},q_{j,\mathrm{ant}}^{\rm non})
    \quad\text{near}\quad
    (\sqrt{\lambda_j^{\mathrm{ant}}},0),
    \qquad
    u_{j,\mathrm{ant}}^{\rm non}\in H^1_{\mathrm{ant}}(D;\mathbb R).
\]
Equivalently, with $\delta=\varepsilon^2$, as $(t,\delta)\to(0,0)$,
\begin{align}
\label{eq:nonlinear-bic-frequency}
    \omega_{j,\mathrm{ant}}^{\rm non}(t,\delta)
    &=
    c_b\sqrt{\lambda_j^{\mathrm{ant}}}\sqrt\delta
    \left(
        1-
        \frac12\beta_j^{\mathrm{ant}}t^2
    \right)
    +
    \mathcal O\left(
        |t|^4\sqrt\delta+\delta^{3/2}
    \right),
    \\
\label{eq:nonlinear-bic-field-expansion}
    u_{j,\mathrm{ant}}^{\rm non}(t,\delta)
    &=
    tu_{p_j^{\mathrm{ant}}}
    +
    t^3U_{j,\mathrm{ant}}^3
    +
    \mathcal O\left(
        |t|^5+|t|\delta
    \right)
    \quad\text{in }H^1(D).
\end{align}
Here, $\beta_j^{\mathrm{ant}}$ and $U_{j,\mathrm{ant}}^3$ are defined as in
\cref{thm:nonlinear-modal-continuation}, with $p_j=p_j^{\mathrm{ant}}$.
The outgoing exterior continuation has no propagating $\eta=0$
Rayleigh--Bloch mode and therefore decays exponentially as
$|x_d|\to\infty$. Hence, each antisymmetric branch is a nonlinear BIC:
$\omega_{j,\mathrm{ant}}^{\rm non}(t,\delta)$ is a real embedded
eigenfrequency, and $u_{j,\mathrm{ant}}^{\rm non}(t,\delta)$ is the
corresponding real antisymmetric embedded eigenfunction.
\end{theorem}

\begin{proof}
The symmetric class follows by applying
\cref{thm:nonlinear-modal-continuation} in
$\mathbb C^N_{\mathrm{sym}}$. For the antisymmetric class, we run the same
modal argument in $\mathbb R^N_{\mathrm{ant}}$ and indicate the changes specific
to this subspace.

\smallskip
\noindent\emph{Step 1: The real antisymmetric reduced residual.}
Recall $\mathcal F^{\rm non}$ from
\eqref{eq:nonlinear-reduced-equation}. If
$q\in\mathbb R^N_{\mathrm{ant}}$ and $\hat\omega\in\mathbb R$, then at
$a=0$ the fixed-point problem defining $z^{\rm non}$ has real
antisymmetric data. The uniqueness in
\cref{prop:nonlinear-projected-reduction} therefore gives
$z^{\rm non}(q,\hat\omega,\varepsilon)\in
\mathcal Z(D)\cap H^1_{\mathrm{ant}}(D;\mathbb R)$.
The Kerr map preserves the antisymmetric subspace:
\eqref{eq:def-Vsigma-q}, $q_{\pi(i)}=-q_i$,
$\sigma_{\pi(i)}=\sigma_i$, and $|D_{\pi(i)}|=|D_i|$ imply
$\bigl(V_\sigma(q)\bigr)_{\pi(i)}=-\bigl(V_\sigma(q)\bigr)_i$.
Together with $V\Pi=\Pi V$, $\Pi C^0=C^0\Pi$, and the reflection
covariance of the terms defining $R^{\rm non}$, these facts show that
$\mathcal F^{\rm non}$ restricts to a real-smooth map
\[
    \mathcal F_{\mathrm{ant}}^{\rm non}
    :
    \mathbb R^N_{\mathrm{ant}}\times\mathbb R\times[0,\varepsilon_0)
    \longrightarrow
    \mathbb R^N_{\mathrm{ant}} .
\]
The restriction of $R^{\rm non}$ follows from the sector property of
$z^{\rm non}$. By \cref{lem:antisymmetric-C1-vanishing},
$C^1q=0$ for all $q\in\mathbb R^N_{\mathrm{ant}}$.
Hence, on this restricted space,
\[
    \mathcal F_{\mathrm{ant}}^{\rm non}(q,\hat\omega,\varepsilon)
    =
    \hat\omega^2\bigl(Vq+V_\sigma(q)\bigr)
    -
    C^0q
    +
    R_{\mathrm{ant}}^{\rm non}(q,\hat\omega,\varepsilon),
\]
where $R_{\mathrm{ant}}^{\rm non}=R^{\rm non}|_{\mathbb R^N_{\mathrm{ant}}}$.
By \eqref{eq:nonlinear-R-estimate}, uniformly for $\hat\omega$ in bounded
sets,
\[
    R_{\mathrm{ant}}^{\rm non}(q,\hat\omega,\varepsilon)
    =
    \mathcal O(\varepsilon^2\|q\|)
    +
    \mathcal O(\varepsilon^4\|q\|^3).
\]
In particular, the order-$\varepsilon$ radiative term present in the general
modal theorem is absent in the antisymmetric subspace.

\smallskip
\noindent\emph{Step 2: Modal coordinates in the antisymmetric subspace.}
Fix a simple antisymmetric eigenpair and, in this proof, write
\[
    p_j:=p_j^{\mathrm{ant}},
    \qquad
    \lambda_j:=\lambda_j^{\mathrm{ant}},
    \qquad
    \hat\omega_j^0:=\sqrt{\lambda_j}.
\]
We use the real modal decomposition
\[
    q=tp_j+\xi,
    \qquad
    \xi\in E_{j,\mathrm{ant}}^\perp,
    \qquad
    E_{j,\mathrm{ant}}^\perp
    :=
    \{\xi\in\mathbb R^N_{\mathrm{ant}}:p_j^\top V\xi=0\}.
\]
Set
\[
    \beta_j^{\mathrm{ant}}
    :=
    p_j^\top V_\sigma(p_j).
\]
The cubic correction is the unique
$\xi_{j,\mathrm{ant}}^3\in E_{j,\mathrm{ant}}^\perp$ that satisfies
\[
    (\lambda_jV-C^0)\xi_{j,\mathrm{ant}}^3
    +
    \lambda_j
    \bigl(
        V_\sigma(p_j)
        -
        \beta_j^{\mathrm{ant}}Vp_j
    \bigr)
    =
    0,
    \qquad
    p_j^\top V\xi_{j,\mathrm{ant}}^3=0,
\]
and we set $U_{j,\mathrm{ant}}^3:=u_{\xi_{j,\mathrm{ant}}^3}$. The forcing
term satisfies $p_j^\top(V_\sigma(p_j)-\beta_j^{\mathrm{ant}}Vp_j)=0$. Moreover,
$\lambda_jV-C^0$ is invertible on $E_{j,\mathrm{ant}}^\perp$ because
$\lambda_j$ is simple in $\mathbb R^N_{\mathrm{ant}}$.

The transverse equation and the divided scalar equation are exactly those in
\cref{thm:nonlinear-modal-continuation}, with
$\mathbb C^N$ replaced by $\mathbb R^N_{\mathrm{ant}}$. Thus, the real
implicit function theorem gives a real-smooth normalized branch
\[
    q_{j,\mathrm{ant}}^{\rm non}(t,\varepsilon)
    =
    tp_j+\xi_j(t,\varepsilon),
    \qquad
    \hat\omega_{j,\mathrm{ant}}^{\rm non}(t,\varepsilon),
    \qquad
    p_j^\top Vq_{j,\mathrm{ant}}^{\rm non}=t,
\]
for $|t|$ and $\varepsilon$ sufficiently small. The derivative of the
divided scalar equation with respect to $\hat\omega$ at
$(t,\xi,\hat\omega,\varepsilon)=(0,0,\hat\omega_j^0,0)$ is
$2\hat\omega_j^0\ne0$, as in the general theorem.

\smallskip
\noindent\emph{Step 3: Sector expansion.}
The coefficient computation at $\varepsilon=0$ is the same as in
\cref{thm:nonlinear-modal-continuation}, now carried out inside
$\mathbb R^N_{\mathrm{ant}}$. It gives
\[
    \hat\omega_{j,\mathrm{ant}}^{\rm non}(t,0)
    =
    \hat\omega_j^0
    -
    \frac12\hat\omega_j^0\beta_j^{\mathrm{ant}}t^2
    +
    \mathcal O(|t|^4),
    \qquad
    q_{j,\mathrm{ant}}^{\rm non}(t,0)
    =
    tp_j
    +
    t^3\xi_{j,\mathrm{ant}}^3
    +
    \mathcal O(|t|^5).
\]
The cancellation $C^1q=0$ removes the order-$\varepsilon$ reduced term.
Since $R_{\mathrm{ant}}^{\rm non}=\mathcal O(\varepsilon^2\|q\|)
+\mathcal O(\varepsilon^4\|q\|^3)$, the remaining $\varepsilon$-dependent
terms start at order $\varepsilon^2$. Thus, the terms corresponding to
$-\mathrm i\,\hat\omega_j^1\varepsilon$ and
$-\mathrm i\,t\varepsilon q_j^1$ in
\cref{thm:nonlinear-modal-continuation} vanish in this subspace. Hence,
\begin{align*}
    \hat\omega_{j,\mathrm{ant}}^{\rm non}(t,\varepsilon)
    =
    \hat\omega_j^0
    -
    \frac12\hat\omega_j^0\beta_j^{\mathrm{ant}}t^2
    +
    \mathcal O(|t|^4+\varepsilon^2),\qquad
    q_{j,\mathrm{ant}}^{\rm non}(t,\varepsilon)
    =
    tp_j
    +
    t^3\xi_{j,\mathrm{ant}}^3
    +
    \mathcal O(|t|^5+|t|\varepsilon^2).
\end{align*}
With $\delta=\varepsilon^2$ and
$\omega_{j,\mathrm{ant}}^{\rm non}
=c_b\varepsilon\hat\omega_{j,\mathrm{ant}}^{\rm non}$, this proves
\eqref{eq:nonlinear-bic-frequency}. The field is reconstructed as
\[
    u_{j,\mathrm{ant}}^{\rm non}
    =
    u_{q_{j,\mathrm{ant}}^{\rm non}}
    +
    z^{\rm non}
    \bigl(
        q_{j,\mathrm{ant}}^{\rm non},
        \hat\omega_{j,\mathrm{ant}}^{\rm non},
        \varepsilon
    \bigr).
\]
By \eqref{eq:nonlinear-basic-z-estimate},
$z^{\rm non}=\mathcal O(|t|\varepsilon^2)=\mathcal O(|t|\delta)$ in
$H^1(D)$. The linearity of $q\mapsto u_q$ then gives
\eqref{eq:nonlinear-bic-field-expansion}; in particular, the field term
corresponding to $-\mathrm i\,tU_j^1\sqrt\delta$ is absent.

\smallskip
\noindent\emph{Step 4: Exact nonlinear BIC property.}
It remains to verify the exact BIC property. The reconstructed interior field is
real antisymmetric and solves the real restricted nonlinear resonance problem. By
\cref{prop:nonlinear-antisymmetric-lifting}, it solves the full nonlinear
resonance problem. Let the same symbol denote its outgoing exterior
continuation. Reflection covariance of the exterior Dirichlet problem gives
$\mathcal R_\ell[u_{j,\mathrm{ant}}^{\rm non}]
=-u_{j,\mathrm{ant}}^{\rm non}$. Therefore
\[
    \bigl(u_{j,\mathrm{ant}}^{\rm non}\bigr)^0_0(\pm h)
    =
    \frac1{|Y|}
    \int_Y
    u_{j,\mathrm{ant}}^{\rm non}(x_\ell,\pm h)\,dx_\ell
    =
    -\bigl(u_{j,\mathrm{ant}}^{\rm non}\bigr)^0_0(\pm h).
\]
Thus, the propagating $\eta=0$ Rayleigh--Bloch coefficient vanishes. In the
subwavelength regime this is the only propagating channel, so the outgoing field
contains only evanescent modes and decays exponentially. Applying the argument
to every simple antisymmetric mode gives the $n_{\pi}^{\mathrm p}$ nonlinear
BIC families.
\end{proof}

\section{Numerical illustrations}
\label{section:numerical-experiments}

We now numerically illustrate the four theoretical conclusions established above: the linear
capacitance expansion in \cref{thm:linear-simple-resonance-branch}, the
reflection-protected linear BIC mechanism in
\cref{thm:linear-bic-symmetry-classification}, the nonlinear modal continuation
in \cref{thm:nonlinear-modal-continuation}, and the persistence of
antisymmetric nonlinear BIC branches in
\cref{thm:nonlinear-bic-symmetry-classification}. For simplicity, all examples are two-dimensional. Unless otherwise stated, we use
\begin{itemize}
    \item incident direction $(0,-1)$ and $\alpha=0$;
    \item period cell $Y=[-L/2,L/2]$ with $L=20$;
    \item nonlinear parameters: $\sigma_j=1$, $1\leq j\leq N$;
    \item material parameters $c_m=c_b=1$ and contrast $\delta=10^{-3}$.
\end{itemize}

\subsection{Linear medium}

For linear experiments, both the capacitance quantities and the reference
resonances are computed by periodic boundary-integral methods. The latter are
obtained from the frequency-dependent boundary-integral formulation
\cite{ammari2021bound}, which gives a nonlinear eigenvalue problem in the
spectral parameter $\omega$. We locate characteristic values by a contour integral method and then refine
them by Newton iteration (see, e.g.
\cite{guttel2017nonlinear}). For each nonzero reduced mode, we use the following notation:
\begin{enumerate}
    \item $\omega_{j,{\rm exact}}^{\rm lin}(\delta)$ denotes the exact
    resonance branch near $c_b\sqrt{\delta\lambda_j}$;
    \item $\omega_{j,{\rm approx}}^{\rm lin}(\delta)$ denotes the capacitance
    approximation from \eqref{eq:linear-resonance-expansion}:
    \[
        \omega_{j,{\rm approx}}^{\rm lin}(\delta)
        =
        c_b\hat\omega_j^0\sqrt\delta
        -
        \mathrm i\,c_b\hat\omega_j^1\delta .
    \]
\end{enumerate}

\subsubsection{Linear subwavelength resonances}

We begin with the randomly generated six-particle geometry in
\Cref{fig:paper-linear6}. At the reference contrast, the five nonzero exact resonances are visually
indistinguishable from their capacitance approximations at the scale shown. To test the
asymptotic order, we then vary
$\delta\in\{0.3,0.5,0.7,1,1.5,2\}\times10^{-3}$. The resulting log--log slopes
agree with the $\mathcal O(\delta^{3/2})$ remainder in
\cref{thm:linear-simple-resonance-branch} and
\eqref{eq:linear-resonance-expansion}.

\begin{figure}[!htbp]
\centering
\begin{minipage}[b]{0.325\textwidth}
\centering
\includegraphics[width=\linewidth]{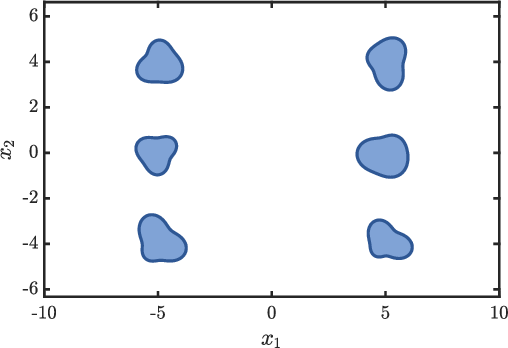}
\end{minipage}
\hfill
\begin{minipage}[b]{0.325\textwidth}
\centering
\includegraphics[width=\linewidth]{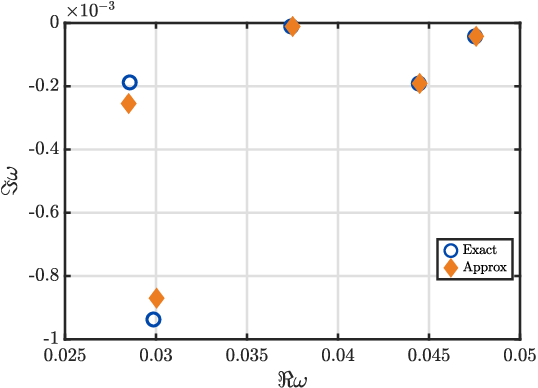}
\end{minipage}
\hfill
\begin{minipage}[b]{0.325\textwidth}
\centering
\includegraphics[width=\linewidth]{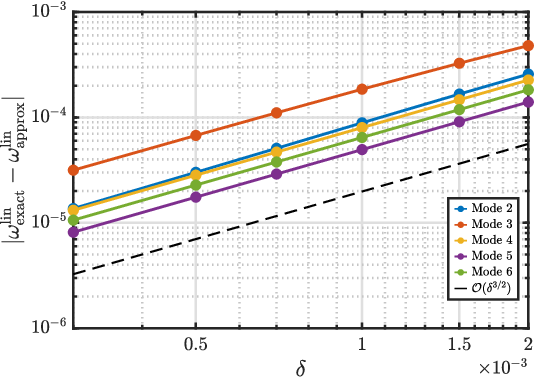}
\end{minipage}
\caption{Linear subwavelength resonances for a six-particle configuration.
Left: particle geometry in one period. Middle: exact resonances and capacitance
approximations in the complex plane. Right: log--log error plot for the five
nonzero modes, with the $\mathcal O(\delta^{3/2})$ reference slope from
\eqref{eq:linear-resonance-expansion}.}
\label{fig:paper-linear6}
\end{figure}
\FloatBarrier

\subsubsection{Reflection-protected linear BICs}

We next turn to the reflection-protected mechanism. The symmetric structure
contains seven particles and is invariant under $\mathcal R_\ell$; the broken
structure is obtained by a small rigid rotation, as shown in
\Cref{fig:paper-bic7-symmetry-breaking}. This pair of geometries allows us to
compare the protected case with a nearby configuration in which the reflection
constraint is removed.

\begin{figure}[!htbp]
\centering
\begin{tabular}{cc}
\includegraphics[width=0.5\textwidth]{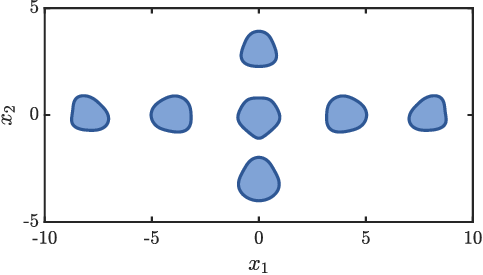}
&
\includegraphics[width=0.5\textwidth]{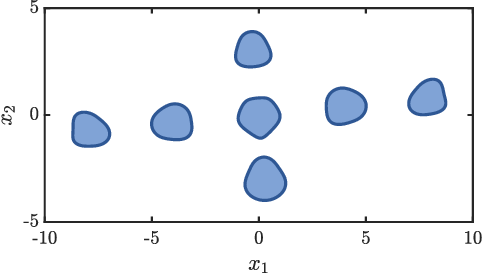}
\end{tabular}
\caption{Seven-particle geometries for the linear BIC. Left:
reflection-symmetric structure. Right: symmetry-broken structure obtained by a
small rigid rotation.}
\label{fig:paper-bic7-symmetry-breaking}
\end{figure}

For both geometries, we compute the exact resonances and the capacitance
approximations. In the symmetric case, two exact resonances have imaginary parts at the
numerical precision level, about $10^{-15}$. This agrees with
\cref{thm:linear-bic-symmetry-classification}: there are two reflected particle
pairs, so $n_{\pi}^{\mathrm p}=2$, and the antisymmetric sector supports two
BIC branches. After symmetry breaking, the corresponding frequencies move into
the lower half-plane as ordinary resonances; see
\Cref{fig:paper-bic7-resonances}.

\begin{figure}[!htbp]
\centering
\begin{tabular}{cc}
\includegraphics[width=0.5\textwidth]{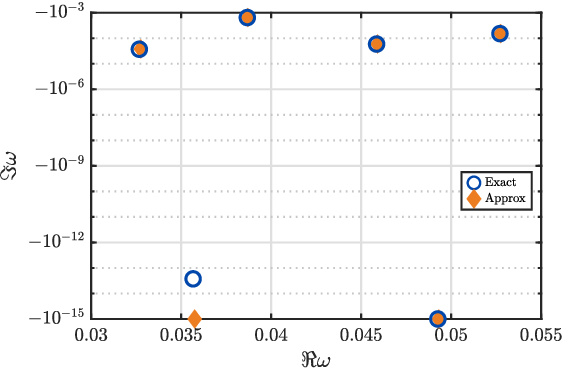}
&
\includegraphics[width=0.5\textwidth]{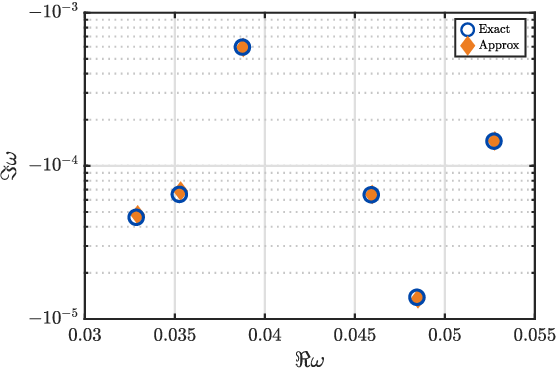}
\end{tabular}
\caption{Exact resonances and capacitance approximations for the seven-particle
BIC test. Left: reflection-symmetric structure, where two antisymmetric modes
lie at the numerical BIC floor. Right: symmetry-broken structure, where the
matched modes become resonances with negative imaginary parts.}
\label{fig:paper-bic7-resonances}
\end{figure}

The transmission calculation gives the complementary scattering signature. In
the symmetric case, the BIC modes do not couple to the incident field. Once the
reflection symmetry is broken, the corresponding quasi-BICs produce Fano-type
features, consistent with the standard BIC-to-Fano picture; see
\cite{ammari2021bound} and \Cref{fig:paper-bic7-transmission}. The associated
mode profiles are displayed in \Cref{fig:paper-bic7-mode-fields}.

\begin{figure}[!htbp]
\centering
\begin{tabular}{cc}
\includegraphics[width=0.5\textwidth]{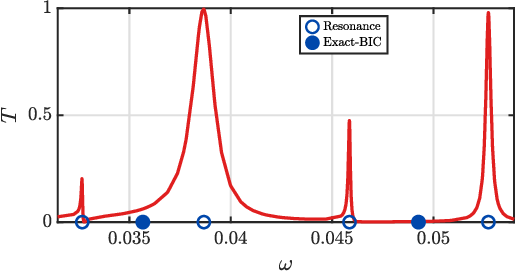}
&
\includegraphics[width=0.5\textwidth]{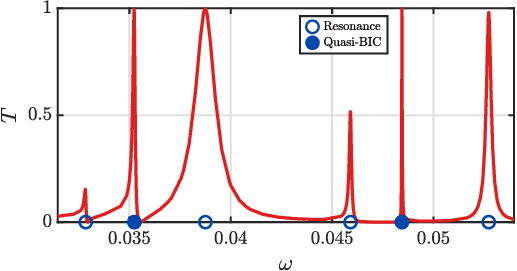}
\end{tabular}
\caption{Transmission coefficient $T$ for the seven-particle BIC test. Left:
reflection-symmetric structure, with ordinary resonances and exact BIC
frequencies marked. Right: symmetry-broken structure, where the matched
quasi-BIC frequencies generate Fano-type transmission features.}
\label{fig:paper-bic7-transmission}
\end{figure}

\begin{figure}[!htbp]
\centering
\setlength{\tabcolsep}{2pt}
\renewcommand{\arraystretch}{0.96}
\begin{tabular}{@{}cc@{}}
\includegraphics[width=0.5\textwidth]{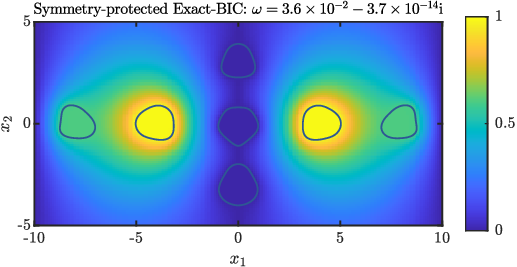}
&
\includegraphics[width=0.5\textwidth]{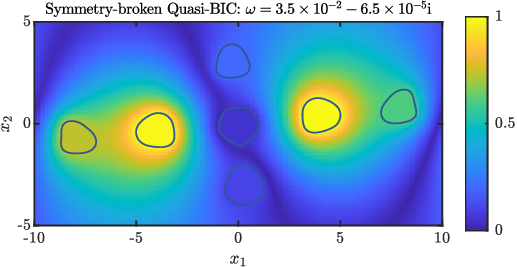}
\\[-0.1em]
\includegraphics[width=0.5\textwidth]{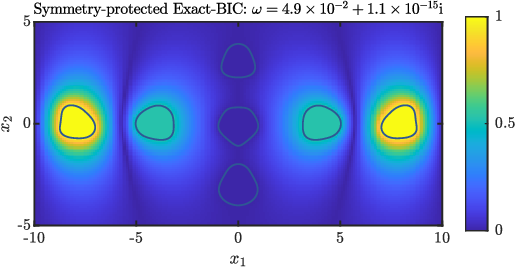}
&
\includegraphics[width=0.5\textwidth]{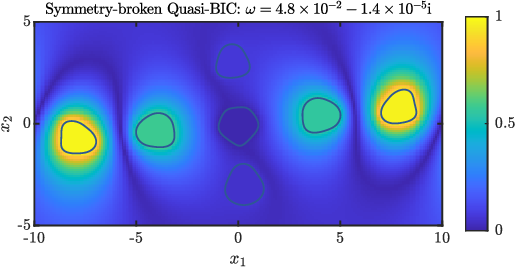}
\end{tabular}
\caption{Mode fields for the two antisymmetric branches identified in
\Cref{fig:paper-bic7-resonances}. Left column: symmetry-protected exact BIC fields. Right column: matched
quasi-BIC fields after symmetry breaking.}
\label{fig:paper-bic7-mode-fields}
\end{figure}
\FloatBarrier

\subsection{Nonlinear medium}

We now turn to the nonlinear reduced model. The full nonlinear resonance or scattering problem is not directly accessible through the linear boundary-integral formulation above, because the Kerr term is
volumetric. A direct computation would require
a volume-integral formulation or a nonlinear variational/FEM solver. Here, we
therefore solve \eqref{eq:nonlinear-reduced-equation} with the remainder
$R^{\rm non}$ dropped. Starting from the corresponding linear reduced resonance, we continue each
branch by Newton iteration in the real amplitude parameter $t=p_j^\top Vq$, with
Kerr strength $\sigma=1$. Set:
\begin{enumerate}
    \item $\omega_{j,{\rm exact}}^{\rm non}(t,\delta)$ denotes the
    solution of the reduced finite-dimensional equation near
    $c_b\sqrt{\delta\lambda_j}$;
    \item $\omega_{j,{\rm approx}}^{\rm non}(t,\delta)$ denotes the asymptotic
    approximation
    \[
        \omega_{j,{\rm approx}}^{\rm non}(t,\delta)
        =
        c_b\hat\omega_j^0\sqrt\delta
        -
        \frac12 c_b\hat\omega_j^0\beta_jt^2\sqrt\delta
        -
        \mathrm i\,c_b\hat\omega_j^1\delta
        =
        \omega_{j,{\rm approx}}^{\rm lin}(\delta)
        -
        \frac12 c_b\hat\omega_j^0\beta_jt^2\sqrt\delta .
    \]
\end{enumerate}

\subsubsection{Nonlinear subwavelength branches}

We first test \cref{thm:nonlinear-modal-continuation} with the four-particle
horizontal configuration in \Cref{fig:paper-nonlinear4-branch}. After ordering
the four reduced linear resonances by modulus, we discard the near-zero branch
and continue the remaining three simple modes. The branch plot shows that
$\Re\omega_j^{\rm non}(t,\delta)$ decreases with $t$, as predicted by the
negative quadratic shift in \eqref{eq:modal-physical-frequency-expansion}.

\begin{figure}[!htbp]
\centering
\begin{tabular}{cc}
\includegraphics[width=0.5\textwidth]{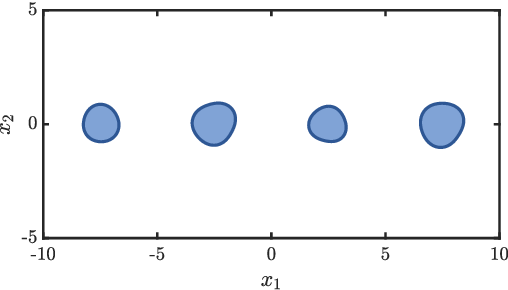} &
\includegraphics[width=0.5\textwidth]{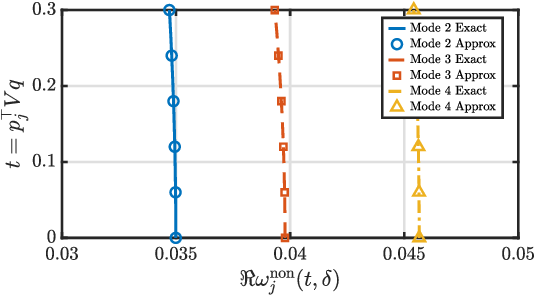}
\end{tabular}
\caption{Reduced nonlinear subwavelength branches for the four-particle
configuration. Left: particle geometry. Right: amplitude $t=p_j^\top Vq$ versus
$\Re\omega_j^{\rm non}(t,\delta)$.}
\label{fig:paper-nonlinear4-branch}
\end{figure}

The next check isolates the leading nonlinear frequency shift in
\eqref{eq:modal-physical-frequency-expansion}. For fixed $\delta$, the shift
from the linear approximate branch scales as $t^2$; for fixed representative
$t$, it scales as $\sqrt\delta$.

\begin{figure}[!htbp]
\centering
\begin{tabular}{cc}
\includegraphics[width=0.5\textwidth]{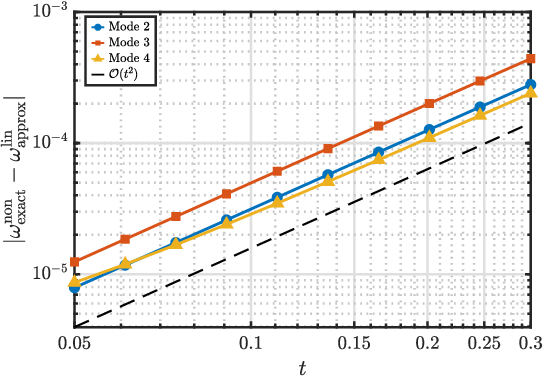} &
\includegraphics[width=0.5\textwidth]{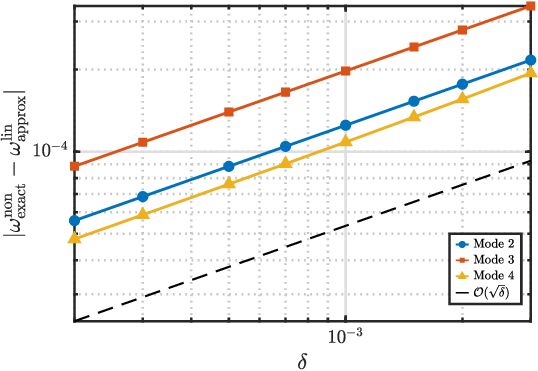}
\end{tabular}
\caption{Order checks for the nonlinear frequency shift from the linear
approximate branch. Left: fixed $\delta=10^{-3}$, showing the
$\mathcal O(t^2)$ rate. Right: fixed amplitude $t=0.20$, showing the
$\mathcal O(\sqrt\delta)$ rate.}
\label{fig:paper-nonlinear4-orders}
\end{figure}
\FloatBarrier

\subsubsection{Nonlinear BIC branches}

Finally, we return to the reflection-symmetric seven-particle geometry in
\Cref{fig:paper-bic7-symmetry-breaking}. The two antisymmetric linear BIC modes
from \cref{thm:linear-bic-symmetry-classification} serve as starting points for
the reduced nonlinear continuation. According to \eqref{eq:nonlinear-bic-frequency},
the Exact--Approx error scales as $\mathcal O(t^4)$ for fixed
$\delta_0=10^{-3}$ and as $\mathcal O(\delta^{3/2})$ under the coupled scaling
$t(\delta)=c\delta^{1/4}$. The fitted slopes in
\Cref{fig:paper-nonlinear-bic7-reduced-orders} are consistent with these rates.

\begin{figure}[!htbp]
\centering
\begin{tabular}{cc}
\includegraphics[width=0.5\textwidth]{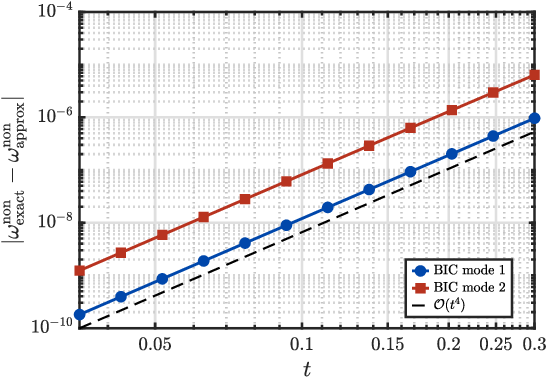} &
\includegraphics[width=0.5\textwidth]{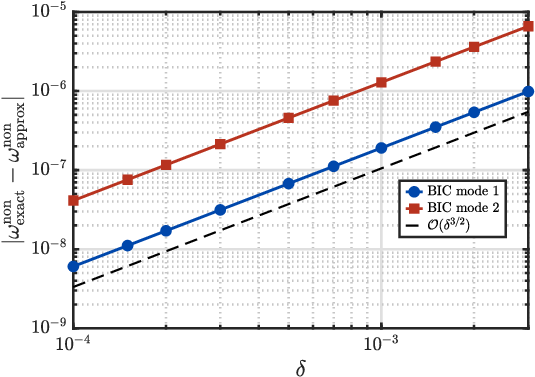}
\end{tabular}
\caption{Reduced nonlinear BIC convergence for the seven-particle symmetric
structure. Left: fixed-$\delta$ Exact--Approx error versus $t$, showing the
$\mathcal O(t^4)$ rate. Right: coupled-$\delta$ Exact--Approx error with
$t(\delta)=c\delta^{1/4}$, showing the $\mathcal O(\delta^{3/2})$ rate. Both
panels include BIC mode $1$ and BIC mode $2$.}
\label{fig:paper-nonlinear-bic7-reduced-orders}
\end{figure}
\FloatBarrier
\section{Concluding remarks}
\label{section:conclusions}

We have developed a mathematical framework for linear and Kerr nonlinear
subwavelength resonances in open periodic arrays of acoustic high-contrast
resonators. The quasiperiodic Dirichlet-to-Neumann map reduces the outgoing
resonance problem to an interior variational problem. Using the decomposition
of the function space in which the variational problem is posed into the direct sum of two function spaces, we separate the resonator amplitudes
from the zero-average correction and derive a finite-dimensional nonlinear
resonance equation. This yields the linear resonance expansion with its leading
radiative correction and proves the small-amplitude nonlinear continuation,
including the nonlinear frequency shift.

We have also proved a symmetry mechanism for exact BICs. In
reflection-symmetric configurations at the center of the Brillouin zone, antisymmetric modes
decouple from the only open Rayleigh--Bloch channel. Hence, the corresponding
branches have real frequencies and zero propagating Rayleigh coefficient. This
gives a classification into symmetric resonances and antisymmetric BICs. Under
the simplicity assumptions, the number of subwavelength BIC branches is the number of reflected pairs. The numerical experiments
confirm the reduced resonance formulas, the nonlinear shift, and the
reflection-protected BIC mechanism.

The reduction developed here suggests two natural directions for future work.
The first is a nonlinear reduced scattering theory for open periodic
high-contrast resonator arrays. In the linear case, the scattering coefficients
depend only on the frequency. In the nonlinear case, the amplitude equation is
nonlinear and the scattering coefficients may also depend on the incident
intensity. This would extend the linear modal decomposition for
subwavelength resonators \cite{feppon2022modal} and provide a reduced framework
for studying nonlinear frequency shifts, amplitude-dependent transmission, and
multiple steady-state scattering responses near subwavelength resonances. The second direction is the study of quasi-BIC and topological BIC phenomena in
acoustic high-contrast resonator arrays. If the symmetry protecting a BIC is
weakly broken, for instance, by a geometric perturbation or by detuning the
quasiperiodicity from the symmetry point, the exact BIC is expected to become a
high-$Q$ quasi-BIC and to produce a narrow Fano-type scattering response
\cite{ammari2021bound,ammari2024fano}. A related question is whether
topological radiation-cancellation mechanisms can be formulated in the present
capacitance-based setting, and whether they provide additional robustness for
embedded non-radiating states under admissible perturbations. Topological BICs
have been studied mainly in photonic systems
\cite{zhen2014topological,bulgakov2017topological,yoda2020generation}, while
related acoustic realizations have also been reported
\cite{chen2019corner,fan2025acoustic}. Combining these ideas with nonlinear
subwavelength reductions may lead to a mathematical framework for nonlinear
BICs and quasi-BIC scattering in acoustic metamaterials.

\appendix

\section{Auxiliary proofs}
\label{app:auxiliary-proofs}

This section gives the deferred proofs of \cref{lem:H0-isomorphism} and
\cref{prop:truncated-single-layer-inverse}.

\begin{proof}[Proof of \cref{lem:H0-isomorphism}]
We prove that $\mathcal H$ is bounded, injective, and Fredholm of index zero.
Since an injective Fredholm operator of index zero has closed range of
codimension zero, it is surjective. The open mapping theorem then gives the
boundedness of the inverse, and hence $\mathcal H$ is an isomorphism.

\smallskip
\noindent\emph{Step 1: Boundedness.}
Let $\mathcal S_D$ denote the free-space single-layer operator. The difference
\[
    \mathcal S_D^{0,0}-\mathcal S_D:
    H^{-1/2}(\partial D)\to H^{1/2}(\partial D)
\]
has a smooth kernel and is therefore compact. Since $\mathcal S_D$ is bounded
\cite[Lemma~6.11]{mclean2000strongly}, so is $\mathcal S_D^{0,0}$. The
functional $\mathfrak m$ is bounded, and the embedding $\mathbb C\ni s\longmapsto s\,1_{\partial D}
    \in H^{1/2}(\partial D)$ is bounded because $\partial D$ is compact. Hence $\mathcal H$ is bounded.

\smallskip
\noindent\emph{Step 2: Injectivity.}
Suppose that $\mathcal H[\psi,s]=(0,0)$. Let $u(x):=\mathcal S_D^{0,0}[\psi](x)$, then
\[
    u+s=\mathcal S_D^{0,0}[\psi]+s=0
    \quad\text{on }\partial D,
    \qquad
    \mathfrak m[\psi]=0 .
\]
From the Fourier representation of
$G^{0,0}$ in \eqref{eq:def-G00} and the condition $\mathfrak m[\psi]=0$, one obtains,
\[
    u(x)
    =
    u_{\pm\infty}
    +
    \sum_{\eta\in\Lambda^*\setminus\{0\}}
        u_{\pm\eta}
        e^{\mathrm i\eta\cdot x_\ell}
        e^{-|\eta||x_d|},\qquad |x_d|\gg1,
\]
where the sign is chosen according to $\pm x_d>0$. In particular, $\nabla u$
decays exponentially as $|x_d|\to\infty$. Since $u=-s$ on $\partial D$, Green's
identity on a truncated period cell, followed by the limit as the height tends
to infinity, gives
\[
    \int_\Omega |\nabla u|^2\,dx
     =
    -\left\langle
        \partial_\nu u|_+
        -
        \partial_\nu u|_-,
        u
    \right\rangle_{\partial D}
    =
    -\langle\psi,u\rangle_{\partial D}
     =
    \overline{s}\,\mathfrak m[\psi]
    =
    0.
\]
Thus, $\nabla u\equiv0$ in $D$ and in $\Omega\setminus\overline D$. The jump
relation then gives $\psi=[\partial_\nu u]=0$. Hence, $u=\mathcal S_D^{0,0}[0]=0$, and the boundary condition yields $s=0$.
Therefore, $\mathcal H$ is injective.

\smallskip
\noindent\emph{Step 3: Fredholm property.}
Introduce
\[
    \mathcal H^{\mathrm{aux}}[\psi,s]
    :=
    \bigl(
        \mathcal S_D[\psi]+s,\,
        \mathfrak m[\psi]
    \bigr).
\]
Then $\mathcal H-\mathcal H^{\mathrm{aux}}$ is compact, since
$\mathcal S_D^{0,0}-\mathcal S_D$ is compact. By
\cite[Theorem~2.26]{ammari2007polarization},
$\mathcal H^{\mathrm{aux}}$ is boundedly invertible and hence Fredholm of index
zero. Therefore, $\mathcal H$ is Fredholm of index zero as well. Together with
injectivity, this proves that $\mathcal H$ is an isomorphism.

The statement for $\mathcal H_0$ follows by restriction. Indeed, for any
$f\in H^{1/2}(\partial D)$, let $ \mathcal H^{-1}[f,0]=(\psi,s)$.
Then $\mathfrak m[\psi]=0$, so
$(\psi,s)\in H^{-1/2}_0(\partial D)\times\mathbb C$, and
$\mathcal H_0[\psi,s]=f$. Thus, $\mathcal H_0$ is surjective. Its injectivity
follows from the injectivity of $\mathcal H$. Hence, $\mathcal H_0$ is an
isomorphism.
\end{proof}

\begin{proof}[Proof of \cref{prop:truncated-single-layer-inverse}]
We isolate the singular $\omega$-dependence of
$\widehat{\mathcal S}_D^{\alpha,k}$, invert the resulting regular leading
operator, and then sum the corresponding Neumann series.

\smallskip
\noindent\emph{Step 1: Decomposition and reduction.}
Every $\psi\in H^{-1/2}(\partial D)$ can be written uniquely as
\[
    \psi
    =
    \psi_0+\mathfrak m[\psi]\,\chi_{\partial D}^{\mathrm{av}},
    \qquad
    \psi_0\in H^{-1/2}_0(\partial D).
\]
To absorb the $\omega^{-1}$ term in \eqref{eq:hatSak-formula}, define
\[
    \Psi_\omega[\psi_0,s]
    :=
    \psi_0+\omega s\,\chi_{\partial D}^{\mathrm{av}},
    \qquad
    (\psi_0,s)\in H^{-1/2}_0(\partial D)\times\mathbb C .
\]
For $\omega\neq0$, $\Psi_\omega$ is bijective onto
$H^{-1/2}(\partial D)$. Substituting
$\psi=\Psi_\omega[\psi_0,s]$ into \eqref{eq:hatSak-formula} gives
\begin{equation}\label{eq:hatS-expansion}
    \widehat{\mathcal S}_D^{\alpha,k}\Psi_\omega
    =
    A_0+\omega A_1,
\end{equation}
where $A_0[\psi_0,s]$ and $A_1[\psi_0,s]$ are defined by
\[
    A_0[\psi_0,s]
    :=
    \mathcal S_D^{0,0}[\psi_0]
    -
    \frac{
        a\cdot\mathfrak m_\ell[\psi_0]
        +
        \mathrm i s
    }{2\tau|Y|},
    \qquad
    A_1[\psi_0,s]
    :=
    s\,f_D .
\]
Thus, both $A_0$ and $A_1$ are independent of $\omega$ and regular at
$\omega=0$.

\smallskip
\noindent\emph{Step 2: Inversion of the leading operator.}
Define
\[
    \mathcal L[\psi_0,s]
    :=
    \left(
        \psi_0,\,
        -\frac{
            a\cdot\mathfrak m_\ell[\psi_0]
            +
            \mathrm i s
        }{2\tau|Y|}
    \right).
\]
Then $\mathcal L$ is a bounded isomorphism on
$H^{-1/2}_0(\partial D)\times\mathbb C$, with inverse
\[
    \mathcal L^{-1}[\psi_0,s]
    =
    \bigl(
        \psi_0,\,
        \mathrm i\,a\cdot\mathfrak m_\ell[\psi_0]
        +
        2\mathrm i\,\tau|Y|\,s
    \bigr).
\]
Since $A_0=\mathcal H_0\mathcal L$, \cref{lem:H0-isomorphism} implies that
$A_0$ is boundedly invertible. If
$(\psi_f^0,s_f^0):=\mathcal H_0^{-1}[f]$, then
\begin{equation}\label{eq:S0-inverse-concrete}
    A_0^{-1}[f]
    =
    \bigl(\psi_f^0,s_f^1\bigr),
    \qquad
    s_f^1
    =
    \mathrm i
    \bigl(
        a\cdot\mathfrak m_\ell[\psi_f^0]
        +
        2\tau|Y|\,s_f^0
    \bigr).
\end{equation}

\smallskip
\noindent\emph{Step 3: Neumann series.}
For $|\omega|$ sufficiently small,
\begin{equation}\label{eq:neumann-series}
    (A_0+\omega A_1)^{-1}
    =
    \sum_{n=0}^{\infty}
        (-\omega)^n
        (A_0^{-1}A_1)^n
        A_0^{-1}
\end{equation}
in operator norm. Hence, for $\omega\neq0$,
\begin{equation}\label{eq:inverse-representation}
    \bigl(\widehat{\mathcal S}_D^{\alpha,k}\bigr)^{-1}
    =
    \Psi_\omega\circ(A_0+\omega A_1)^{-1}.
\end{equation}
The right-hand side extends holomorphically to $\omega=0$: indeed, if
\[
    (\psi_0(\omega),s(\omega))
    :=
    (A_0+\omega A_1)^{-1}[f],
\]
then
\[
    \bigl(\widehat{\mathcal S}_D^{\alpha,k}\bigr)^{-1}[f]
    =
    \psi_0(\omega)
    +
    \omega s(\omega)\chi_{\partial D}^{\mathrm{av}},
\]
which is holomorphic at $\omega=0$. Moreover, from \eqref{eq:S0-inverse-concrete},
\[
    A_1A_0^{-1}[f]
    =
    s_f^1 f_D,
    \qquad
    A_0^{-1}A_1A_0^{-1}[f]
    =
    s_f^1(\psi_D^0,s_D^1).
\]
Iterating gives, for $n\ge1$,
\begin{equation}\label{eq:iterated-formula}
    (A_0^{-1}A_1)^nA_0^{-1}[f]
    =
    s_f^1(s_D^1)^{n-1}
    (\psi_D^0,s_D^1).
\end{equation}

\smallskip
\noindent\emph{Step 4: Closed form.}
Substituting \eqref{eq:iterated-formula} into
\eqref{eq:inverse-representation} and summing the resulting geometric series
yields, for $|\omega s_D^1|<1$,
\[
    \bigl(\widehat{\mathcal S}_D^{\alpha,k}\bigr)^{-1}[f]
    =
    \psi_f^0
    +
    \frac{\omega s_f^1}{1+\omega s_D^1}
    \bigl(
        \chi_{\partial D}^{\mathrm{av}}
        -
        \psi_D^0
    \bigr).
\]
This proves \eqref{eq:closed-form-inverse-hatSak} and it is holomorphic in disk $|\omega |<|s_D^1|^{-1}$. Expanding at $\omega=0$ gives
\eqref{eq:expansion-of-inverse-hatSak}.
\end{proof}

\section{Auxiliary lemmas}
\label{app:auxiliary-lemmas}

This section collects auxiliary lemmas used in the Lyapunov--Schmidt reductions
and reflection arguments: reflection covariance, uniform coercivity on
$\mathcal Z(D)$, a generalized H\"older inequality, and Kerr mapping estimates.

\begin{lemma}[Reflection of boundary integral operators]
\label{lem:reflection-covariance-kernel}
Assume $R_\ell D=D$, and let
\[
    \mathcal Q[\psi](x)
    :=
    \int_{\partial D} Q(x,y)\psi(y)\,d\sigma(y),
    \qquad x\in\partial D .
\]
If $Q(R_\ell x,y)=Q(x,R_\ell y)$ for $x,y\in\partial D$, then $\mathcal R_\ell\mathcal Q
    =
    \mathcal Q\mathcal R_\ell $,
where $\mathcal R_\ell[\psi](x):=\psi(R_\ell x)$ on $\partial D$.
\end{lemma}

\begin{proof}
For $x\in\partial D$, the change of variables $z=R_\ell y$, together with
$R_\ell\partial D=\partial D$ and surface-measure invariance, gives
\[
\begin{aligned}
    \mathcal R_\ell\mathcal Q[\psi](x)
    =
    \int_{\partial D} Q(R_\ell x,y)\psi(y)\,d\sigma(y)
    =
    \int_{\partial D} Q(x,z)\psi(R_\ell z)\,d\sigma(z)
    =
    \mathcal Q\mathcal R_\ell[\psi](x).
\end{aligned}
\]
\end{proof}

\begin{lemma}[Uniform linear coercivity on the zero-average space]
\label{lem:linear-Z-lax-milgram}
Fix $M_{\hat\omega}>0$. There exist
$\varepsilon_0=\varepsilon_0(M_{\hat\omega})>0$, $c_*>0$, and
$C_*>0$ such that, for all $|\hat\omega|\le M_{\hat\omega}$ and
$0<\varepsilon<\varepsilon_0$, the scaled linear form
$a^{\mathrm{lin}}_{\hat\omega,\varepsilon}$ is bounded on
$\mathcal Z(D)\times\mathcal Z(D)$ and satisfies
\begin{equation}
\label{eq:linear-Z-boundedness}
    \left|
        a^{\mathrm{lin}}_{\hat\omega,\varepsilon}(u,v)
    \right|
    \le
    C_*
    \|u\|_{H^1(D)}
    \|v\|_{H^1(D)},
    \qquad
    u,v\in\mathcal Z(D),
\end{equation}
and
\begin{equation}
\label{eq:linear-Z-coercivity}
    \operatorname{Re}
    a^{\mathrm{lin}}_{\hat\omega,\varepsilon}(u,u)
    \ge
    c_*
    \|u\|_{H^1(D)}^2,
    \qquad
    u\in\mathcal Z(D).
\end{equation}
Consequently, for every $F\in\mathcal Z(D)'$, the variational problem
\begin{equation}
\label{eq:property-of-linear-variational-form}
    a^{\mathrm{lin}}_{\hat\omega,\varepsilon}(u,v)=F(v),
    \qquad
    v\in\mathcal Z(D),
\end{equation}
admits a unique solution $u\in\mathcal Z(D)$, and
\begin{equation}
\label{eq:lax-milgram-estimate}
    \|u\|_{H^1(D)}
    \le
    c_*^{-1}\|F\|_{\mathcal Z(D)'} .
\end{equation}
\end{lemma}

\begin{proof}
The space $\mathcal Z(D)$ is a closed subspace of $H^1(D)$ and is a
Hilbert space. Choose $\varepsilon_0>0$ such that
$\varepsilon_0c_bM_{\hat\omega}<\omega_0$. Then, for
$|\hat\omega|\le M_{\hat\omega}$ and $0<\varepsilon<\varepsilon_0$, the
frequency $\omega=\varepsilon c_b\hat\omega$ lies in the low-frequency regime
where the DtN operator $\mathcal T_D^{\alpha,k_m}$ is uniformly bounded by
\cref{lem:DtN-low-frequency-expansion}. Here
$\alpha=\varepsilon c_b\hat\omega a$ and
$k_m=\varepsilon c_b\hat\omega/c_m$. The gradient and mass terms are bounded by
the Cauchy--Schwarz inequality, while the DtN term is bounded by the trace
theorem and the uniform boundedness of $\mathcal T_D^{\alpha,k_m}$. This gives
\eqref{eq:linear-Z-boundedness}.

Next, we prove coercivity. Since the functions in $\mathcal Z(D)$ have zero average
on each connected component, the Poincar\'e--Wirtinger inequality
(see, e.g. \cite[Chapter~5, Section~5.8.1]{evans2010partial}) gives a constant
$\gamma_Z>0$, depending only on $D$, such that
\[
    \|\nabla u\|_{L^2(D)}^2
    \ge
    \gamma_Z\|u\|_{H^1(D)}^2,
    \qquad
    u\in\mathcal Z(D).
\]
Moreover,
\[
    \left|
        \varepsilon^2\hat\omega^2(u,u)_D
        +
        \varepsilon^2
        \langle
            \mathcal T_D^{\alpha,k_m}[u],
            u
        \rangle_{\partial D}
    \right|
    \le
    C\varepsilon^2(1+M_{\hat\omega}^2)\|u\|_{H^1(D)}^2,
\]
where $C$ is independent of $u$, $\hat\omega$, and $\varepsilon$. Hence,
\[
    \operatorname{Re}
    a^{\mathrm{lin}}_{\hat\omega,\varepsilon}(u,u)
    \ge
    \bigl(\gamma_Z-C\varepsilon^2(1+M_{\hat\omega}^2)\bigr)
    \|u\|_{H^1(D)}^2.
\]
After decreasing $\varepsilon_0$ if necessary, we may assume
$C\varepsilon_0^2(1+M_{\hat\omega}^2)\le\gamma_Z/2$. Then
\eqref{eq:linear-Z-coercivity} holds with $c_*=\gamma_Z/2$. The complex Lax--Milgram lemma (see, for instance,
\cite[Theorem~2.32]{mclean2000strongly}) applied on the Hilbert space
$\mathcal Z(D)$ gives \eqref{eq:property-of-linear-variational-form} and
\eqref{eq:lax-milgram-estimate}.
\end{proof}

The following lemma is a convenient form of H\"older's inequality \cite[Appendix~B.2]{evans2010partial}. The
$L^r$ estimate follows from the standard H\"older inequality applied to
$|u_j|^r$ with exponents $p_j/r$, with the convention $p_j/r=\infty$ if
$p_j=\infty$.

\begin{lemma}[Generalized H\"older inequality]
\label{lem:generalized-holder}
Let $(X,\mu)$ be a measure space, and use the convention $1/\infty=0$.
Suppose $r\in(0,\infty)$ and $p_1,\ldots,p_n\in[1,\infty]$ satisfy
\begin{equation}
\label{eq:generalized-holder-estimate-exponents}
    \frac1{p_1}+\cdots+\frac1{p_n}=\frac1r .
\end{equation}
If $u_j\in L^{p_j}(X)$ for $j=1,\ldots,n$, then
$\prod_{j=1}^n u_j\in L^r(X)$ and
\begin{equation}
\label{eq:generalized-holder-estimate}
    \left\|
        \prod_{j=1}^n u_j
    \right\|_{L^r(X)}
    \le
    \prod_{j=1}^n\|u_j\|_{L^{p_j}(X)} .
\end{equation}
\end{lemma}

\begin{lemma}[Kerr mapping and fluctuation estimates]
\label{lem:kerr-local-estimate}
Assume $d\le3$, and set
\[
    \sigma_*:=\|\sigma_D\|_{L^\infty(D)}
    =
    \max_{1\le j\le N}|\sigma_j|.
\]
Then the Kerr map $\mathcal N_\sigma:H^1(D)\to L^2(D)$ is well defined. There
is a constant $C$, depending only on $d$, $D$, and $N$, such that
\begin{equation}
\label{eq:kerr-H1-L2-bound}
    \|\mathcal N_\sigma[u]\|_{L^2(D)}
    \le
    C\sigma_*\|u\|_{H^1(D)}^3,
    \qquad
    u\in H^1(D),
\end{equation}
and
\begin{equation}
\label{eq:kerr-H1-L2-lipschitz}
    \|\mathcal N_\sigma[u]-\mathcal N_\sigma[w]\|_{L^2(D)}
    \le
    C\sigma_*
    \bigl(
        \|u\|_{H^1(D)}^2+\|w\|_{H^1(D)}^2
    \bigr)
    \|u-w\|_{H^1(D)},
    \qquad
    u,w\in H^1(D).
\end{equation}
As a map between real Banach spaces, $\mathcal N_\sigma$ is real analytic, and
\begin{equation}
\label{eq:kerr-real-derivative}
    D\mathcal N_\sigma[u]\,h
    =
    \sigma_D
    \bigl(
        2|u|^2h+u^2\overline h
    \bigr),
    \qquad
    u,h\in H^1(D).
\end{equation}

Let $\mathcal H_q[z]\in\mathcal Z(D)'$ be the functional defined in
\eqref{eq:def-z-projection-functionals}. There is a constant $C$, depending
only on $d$, $D$, $N$, and $\sigma_*$, such that, for all
$q\in\mathbb C^N$ and $z\in\mathcal Z(D)$,
\begin{equation}
\label{eq:pc-nonlinear-Hq-growth}
    \|\mathcal H_q[z]\|_{\mathcal Z(D)'}
    \le
    C
    \bigl(
        \|q\|^2
        +
        \|q\|\|z\|_{H^1(D)}
        +
        \|z\|_{H^1(D)}^2
    \bigr)
    \|z\|_{H^1(D)} .
\end{equation}
Moreover, for all $q\in\mathbb C^N$ and
$z_1,z_2\in\mathcal Z(D)$,
\begin{equation}
\label{eq:pc-nonlinear-Hq-lipschitz}
\begin{aligned}
    &\|\mathcal H_q[z_1]-\mathcal H_q[z_2]\|_{\mathcal Z(D)'}
    \\
    &\le
    C\Big(
        \|q\|^2
        +
        \|q\|
        \bigl(
            \|z_1\|_{H^1(D)}
            +
            \|z_2\|_{H^1(D)}
        \bigr)
        +
        \|z_1\|_{H^1(D)}^2
        +
        \|z_2\|_{H^1(D)}^2
    \Big)
    \|z_1-z_2\|_{H^1(D)} .
\end{aligned}
\end{equation}
\end{lemma}

\begin{proof}
Throughout the proof, $C$ may change from line to line and depends only on
$d$, $D$, and $N$, unless the factor $\sigma_*$ is absorbed into it. Since
$d\le3$, the Sobolev embedding (see, e.g.
\cite[Chapter~5, Section~5.6]{evans2010partial}) gives
\begin{equation}
\label{eq:kerr-sobolev-embedding}
    \|u\|_{L^p(D)}
    \le
    C\|u\|_{H^1(D)},
    \qquad
    2\le p\le6 .
\end{equation}
We shall also use the following pointwise inequalities. For all
$x,y\in\mathbb C$,
\begin{equation}
\label{eq:kerr-pointwise-cubic-lipschitz}
    \bigl||x|^2x-|y|^2y\bigr|
    \le
    C\bigl(|x|^2+|y|^2\bigr)|x-y|.
\end{equation}
Consequently, for all $a,b,b_1,b_2\in\mathbb C$,
\begin{align}
\label{eq:kerr-pointwise-cubic-growth}
    \bigl||a+b|^2(a+b)-|a|^2a\bigr|
    &\le
    C\bigl(
        |a|^2|b|
        +
        |a||b|^2
        +
        |b|^3
    \bigr),
    \\
\label{eq:kerr-pointwise-cubic-fluctuation-lipschitz}
    \bigl||a+b_1|^2(a+b_1)-|a+b_2|^2(a+b_2)\bigr|
    &\le
    C\bigl(
        |a|^2
        +
        |a|(|b_1|+|b_2|)
        +
        |b_1|^2+|b_2|^2
    \bigr)|b_1-b_2|.
\end{align}

Since $|\mathcal N_\sigma[u]|\le\sigma_*|u|^3$, applying
\eqref{eq:kerr-sobolev-embedding} with $p=6$ gives
\[
    \|\mathcal N_\sigma[u]\|_{L^2(D)}
    \le
    \sigma_*\|u\|_{L^6(D)}^3
    \le
    C\sigma_*\|u\|_{H^1(D)}^3.
\]
Thus, $\mathcal N_\sigma:H^1(D)\to L^2(D)$ is well defined and
\eqref{eq:kerr-H1-L2-bound} holds. For the difference estimate,
\eqref{eq:kerr-pointwise-cubic-lipschitz} gives
\[
    |\mathcal N_\sigma[u]-\mathcal N_\sigma[w]|
    \le
    C\sigma_*
    \bigl(|u|^2+|w|^2\bigr)|u-w|.
\]
By \cref{lem:generalized-holder} with $1/2=1/6+1/6+1/6$,
\[
    \bigl\||u|^2|u-w|\bigr\|_{L^2(D)}
    \le
    \|u\|_{L^6(D)}^2
    \|u-w\|_{L^6(D)}.
\]
The same estimate with $w$ in place of $u$, together with
\eqref{eq:kerr-sobolev-embedding}, proves
\eqref{eq:kerr-H1-L2-lipschitz}.

We next prove real analyticity. Write
$\mathcal N_\sigma[u]=T(u,u,u)$, where
$T(u_1,u_2,u_3):=\sigma_Du_1u_2\overline{u_3}$. By
\cref{lem:generalized-holder} and \eqref{eq:kerr-sobolev-embedding},
\[
    \|T(u_1,u_2,u_3)\|_{L^2(D)}
    \le
    \sigma_*
    \|u_1\|_{L^6(D)}
    \|u_2\|_{L^6(D)}
    \|u_3\|_{L^6(D)}
    \le
    C\sigma_*
    \prod_{j=1}^3\|u_j\|_{H^1(D)}.
\]
Therefore, $T:H^1(D)^3\to L^2(D)$ is a bounded real trilinear map. Its diagonal
restriction $u\mapsto T(u,u,u)$ is real analytic as a map between real Banach
spaces. Differentiating this real polynomial map gives
\eqref{eq:kerr-real-derivative}, since
\[
    D\mathcal N_\sigma[u]\,h
    =
    T(h,u,u)+T(u,h,u)+T(u,u,h)
    =
    \sigma_D
    \bigl(
        2|u|^2h+u^2\overline h
    \bigr).
\]

It remains to estimate $\mathcal H_q$. By \eqref{eq:def-tilde-Nsigma},
$\widetilde{\mathcal N}_\sigma[q,0]=0$. Applying
\eqref{eq:kerr-pointwise-cubic-growth} with $a=u_q$ and $b=z$ gives
\[
    |\widetilde{\mathcal N}_\sigma[q,z]|
    \le
    C\sigma_*
    \bigl(
        |u_q|^2|z|
        +
        |u_q||z|^2
        +
        |z|^3
    \bigr).
\]
Because $u_q$ is componentwise constant, $\|u_q\|_{L^\infty(D)}\le C\|q\|$.
Thus, for $v\in\mathcal Z(D)$,
\[
\begin{aligned}
    |\mathcal H_q[z](v)|
    &\le
    C
    \int_D
    \bigl(
        |u_q|^2|z|
        +
        |u_q||z|^2
        +
        |z|^3
    \bigr)|v|\,dx
    \\
    &\le
    C\|q\|^2\|z\|_{L^2(D)}\|v\|_{L^2(D)}
    +
    C\|q\|\|z\|_{L^4(D)}^2\|v\|_{L^2(D)}
    +
    C\|z\|_{L^4(D)}^3\|v\|_{L^4(D)}.
\end{aligned}
\]
Using \eqref{eq:kerr-sobolev-embedding} with $p=2,4$ and then taking the
supremum over $\|v\|_{H^1(D)}\le1$ proves
\eqref{eq:pc-nonlinear-Hq-growth}.

Finally, let $w=z_1-z_2$. By
\eqref{eq:kerr-pointwise-cubic-fluctuation-lipschitz}, for
$v\in\mathcal Z(D)$,
\[
\begin{aligned}
    |\mathcal H_q[z_1](v)-\mathcal H_q[z_2](v)|
    \le
    C
    \int_D
    \bigl(
        |u_q|^2
        +
        |u_q|(|z_1|+|z_2|)
        +
        |z_1|^2+|z_2|^2
    \bigr)|w||v|\,dx .
\end{aligned}
\]
The three contributions are bounded by
\[
\begin{aligned}
    \int_D |u_q|^2|w||v|\,dx
    &\le
    C\|q\|^2
    \|w\|_{L^2(D)}
    \|v\|_{L^2(D)},
    \\
    \int_D |u_q||z_i||w||v|\,dx
    &\le
    C\|q\|
    \|z_i\|_{L^4(D)}
    \|w\|_{L^4(D)}
    \|v\|_{L^2(D)},
    \qquad i=1,2,
    \\
    \int_D |z_i|^2|w||v|\,dx
    &\le
    \|z_i\|_{L^4(D)}^2
    \|w\|_{L^4(D)}
    \|v\|_{L^4(D)},
    \qquad i=1,2.
\end{aligned}
\]
Using \eqref{eq:kerr-sobolev-embedding} again yields
\[
\begin{aligned}
    &|\mathcal H_q[z_1](v)-\mathcal H_q[z_2](v)|
    \\
    &\le
    C\Big(
        \|q\|^2
        +
        \|q\|
        \bigl(
            \|z_1\|_{H^1(D)}
            +
            \|z_2\|_{H^1(D)}
        \bigr)
        +
        \|z_1\|_{H^1(D)}^2
        +
        \|z_2\|_{H^1(D)}^2
    \Big)
    \|z_1-z_2\|_{H^1(D)}
    \|v\|_{H^1(D)}.
\end{aligned}
\]
Taking the supremum over all $v\in\mathcal Z(D)$ with
$\|v\|_{H^1(D)}\le1$ gives \eqref{eq:pc-nonlinear-Hq-lipschitz}.
\end{proof}

%%~~~~~~~~~ Bibliography
%\bibliographystyle{unsrt}
\bibliographystyle{siam}
\bibliography{ref}

@article{ammari2017mathematical,
  title={A mathematical and numerical framework for bubble meta-screens},
  author={Ammari, Habib and Fitzpatrick, Brian and Gontier, David and Lee, Hyundae and Zhang, Hai},
  journal={SIAM Journal on Applied Mathematics},
  volume={77},
  number={5},
  pages={1827--1850},
  year={2017},
  publisher={SIAM}
}

@book{ammari2018mathematical,
  title={Mathematical and computational methods in photonics and phononics},
  author={Ammari, Habib and Fitzpatrick, Brian and Kang, Hyeonbae and Ruiz, Matias and Yu, Sanghyeon and Zhang, Hai},
  volume={235},
  year={2018},
  publisher={American Mathematical Soc.}
}

@article{feppon2022modal,
  title={Modal decompositions and point scatterer approximations near the Minnaert resonance frequencies},
  author={Feppon, Florian and Ammari, Habib},
  journal={Studies in Applied Mathematics},
  volume={149},
  number={1},
  pages={164--229},
  year={2022},
  publisher={Wiley Online Library}
}

@article{ammari2018minnaert,
  title={Minnaert resonances for acoustic waves in bubbly media},
  author={Ammari, Habib and Fitzpatrick, Brian and Gontier, David and Lee, Hyundae and Zhang, Hai},
  journal={Annales de l'Institut Henri Poincar{\'e} C, Analyse non lin{\'e}aire},
  volume={35},
  number={7},
  pages={1975--1998},
  year={2018},
  organization={Elsevier}
}

@article{minnaert1933musical,
  title={XVI. On musical air-bubbles and the sounds of running water},
  author={Minnaert, M.},
  journal={The London, Edinburgh, and Dublin Philosophical Magazine and Journal of Science},
  volume={16},
  number={104},
  pages={235--248},
  year={1933},
  doi={10.1080/14786443309462277}
}

@article{ammari2017effective,
  title={Effective medium theory for acoustic waves in bubbly fluids near {M}innaert resonant frequency},
  author={Ammari, Habib and Zhang, Hai},
  journal={SIAM Journal on Mathematical Analysis},
  volume={49},
  number={4},
  pages={3252--3276},
  year={2017}
}

@article{ammari2019fully,
  title={A fully coupled subwavelength resonance approach to filtering auditory signals},
  author={Ammari, Habib and Davies, Bryn},
  journal={Proceedings of the Royal Society A: Mathematical, Physical and Engineering Sciences},
  volume={475},
  number={2227},
  pages={20190049},
  year={2019},
  doi={10.1098/rspa.2019.0049}
}

@article{guttel2017nonlinear,
  title={The nonlinear eigenvalue problem},
  author={G{\"u}ttel, Stefan and Tisseur, Fran{\c{c}}oise},
  journal={Acta Numerica},
  volume={26},
  pages={1--94},
  year={2017},
  publisher={Cambridge University Press}
}

@book{mclean2000strongly,
  title={Strongly elliptic systems and boundary integral equations},
  author={McLean, William Charles Hector},
  year={2000},
  publisher={Cambridge university press}
}

@book{evans2010partial,
  title={Partial Differential Equations},
  author={Evans, Lawrence C.},
  edition={2},
  series={Graduate Studies in Mathematics},
  volume={19},
  year={2010},
  publisher={American Mathematical Society}
}

@article{ammari2021bound,
  title={Bound states in the continuum and Fano resonances in subwavelength resonator arrays},
  author={Ammari, Habib and Davies, Bryn and Hiltunen, Erik Orvehed and Lee, Hyundae and Yu, Sanghyeon},
  journal={Journal of Mathematical Physics},
  volume={62},
  number={10},
  year={2021},
  publisher={AIP Publishing}
}

@book{ammari2007polarization,
  title={Polarization and moment tensors: with applications to inverse problems and effective medium theory},
  author={Ammari, Habib and Kang, Hyeonbae},
  volume={162},
  year={2007},
  publisher={Springer Science \& Business Media}
}

@article{ammari2022exceptional,
  title={Exceptional Points in Parity--Time-Symmetric Subwavelength Metamaterials},
  author={Ammari, Habib and Davies, Bryn and Hiltunen, Erik Orvehed and Lee, Hyundae and Yu, Sanghyeon},
  journal={SIAM Journal on Mathematical Analysis},
  volume={54},
  number={6},
  pages={6223--6253},
  year={2022},
  publisher={SIAM}
}

@article{ammari2024functional,
  title={Functional analytic methods for discrete approximations of subwavelength resonator systems},
  author={Ammari, Habib and Davies, Bryn and Hiltunen, Erik Orvehed},
  journal={Pure and Applied Analysis},
  volume={6},
  number={3},
  pages={873--939},
  year={2024},
  doi={10.2140/paa.2024.6.873}
}

@article {bryn2019,
    AUTHOR = {Ammari, Habib and Davies, Bryn},
     TITLE = {Mimicking the active cochlea with a fluid-coupled array of
              subwavelength {H}opf resonators},
   JOURNAL = {Proc. A},
  FJOURNAL = {Proceedings A},
    VOLUME = {476},
      YEAR = {2020},
    NUMBER = {2234},
     PAGES = {20190870, 18},
       DOI = {10.1098/rspa.2019.0870},
       URL = {https://doi.org/10.1098/rspa.2019.0870},
}

@article{ammari2025nonlinear,
  title={Nonlinear subwavelength resonances in three dimensions},
  author={Ammari, Habib and Kosche, Thea},
  journal={Studies in Applied Mathematics},
  volume={154},
  number={3},
  pages={e70036},
  year={2025},
  publisher={Wiley Online Library}
}

@article{evlyukhin2012demonstration,
  title={Demonstration of magnetic dipole resonances of dielectric nanospheres in the visible region},
  author={Evlyukhin, Andrey B. and Novikov, Sergey M. and Zywietz, Urs and Eriksen, Rene Lynge and Reinhardt, Carsten and Bozhevolnyi, Sergey I. and Chichkov, Boris N.},
  journal={Nano Letters},
  volume={12},
  number={7},
  pages={3749--3755},
  year={2012},
  publisher={American Chemical Society}
}

@article{tzarouchis2018light,
  title={Light scattering by a dielectric sphere: Perspectives on the {Mie} resonances},
  author={Tzarouchis, Dimitrios and Sihvola, Ari},
  journal={Applied Sciences},
  volume={8},
  number={2},
  pages={184},
  year={2018}
}

@article{kuznetsov2016optically,
  title={Optically resonant dielectric nanostructures},
  author={Kuznetsov, Arseniy I. and Miroshnichenko, Andrey E. and Brongersma, Mark L. and Kivshar, Yuri S. and Luk'yanchuk, Boris},
  journal={Science},
  volume={354},
  number={6314},
  pages={aag2472},
  year={2016},
  publisher={American Association for the Advancement of Science},
  doi={10.1126/science.aag2472}
}

@article{ammari2025dielectric,
    AUTHOR = {Ammari, Habib and Li, Bowen},
     TITLE = {Dielectric scattering resonances for high-refractive
              resonators with cubic nonlinearity},
   JOURNAL = {J. Differential Equations},
  FJOURNAL = {Journal of Differential Equations},
    VOLUME = {475},
      YEAR = {2026},
     PAGES = {Paper No. 114459, 56},
       DOI = {10.1016/j.jde.2026.114459},
       URL = {https://doi.org/10.1016/j.jde.2026.114459},
}

@article{ammari2025resonator,
  title={Analysis of nonlinear resonances in resonator crystals: Tight-binding approximation and existence of subwavelength soliton-like solutions},
  author={Ammari, Habib and Qiu, Jiayu},
  journal={arXiv preprint arXiv:2509.04184},
  year={2025}
}

@book{cbms, 
title={Mathematical Theories for Metamaterials: From Condensed Matter Theory to Subwavelength Physics}, 
author = {Ammari, Habib and Davies, Bryn and Hiltunen, Erik Orvehed},
SERIES = {CBMS Regional Conference Series in Mathematics},
    VOLUME = {136},
 PUBLISHER = {American Mathematical Society, Providence, RI},
year={2026},
}

@article{clemens2026,
  title={Perturbative approach to nonlinear capacitance matrix formulations},
  author={Ammari, Habib and Thalhammer, Clemens},
  journal={arXiv preprint arXiv:2606.20821},
  year={2026}
}

@article{meklachi2018asymptotic,
  title={Asymptotic analysis of resonances of small volume high contrast linear and nonlinear scatterers},
  author={Meklachi, Taoufik and Moskow, Shari and Schotland, John C.},
  journal={Journal of Mathematical Physics},
  volume={59},
  number={8},
  year={2018},
  publisher={AIP Publishing}
}

@article{kivshar2018all,
  title={All-dielectric meta-optics and non-linear nanophotonics},
  author={Kivshar, Yuri},
  journal={National Science Review},
  volume={5},
  number={2},
  pages={144--158},
  year={2018},
  publisher={Oxford University Press}
}

@incollection{boyd2008nonlinear,
  title={Nonlinear optics},
  author={Boyd, Robert W. and Gaeta, Alexander L. and Giese, Enno},
  booktitle={Springer Handbook of Atomic, Molecular, and Optical Physics},
  pages={1097--1110},
  year={2008},
  publisher={Springer}
}

@article{koshelev2020subwavelength,
  title={Subwavelength dielectric resonators for nonlinear nanophotonics},
  author={Koshelev, Kirill and Kruk, Sergey and Melik-Gaykazyan, Elizaveta and Choi, Jae-Hyuck and Bogdanov, Andrey and Park, Hong-Gyu and Kivshar, Yuri},
  journal={Science},
  volume={367},
  number={6475},
  pages={288--292},
  year={2020}
}

@article{rev-acoustics,
  title={Dynamics of phononic materials and structures: Historical origins, recent progress, and future outlook},
  author={Hussein, Mahmoud I. and Leamy, Michael J. and Ruzzene, Massimo},
  journal={Applied Mechanics Reviews},
  volume={66},
  pages={040802},
  year={2014}
}

@article{hsu2016bound,
  title={Bound states in the continuum},
  author={Hsu, Chia Wei and Zhen, Bo and Stone, A. Douglas and Joannopoulos, John D. and Solja{\v{c}}i{\'c}, Marin},
  journal={Nature Reviews Materials},
  volume={1},
  number={9},
  pages={16048},
  year={2016},
  publisher={Nature Publishing Group}
}

@article{friedrich1985interfering,
  title={Interfering resonances and bound states in the continuum},
  author={Friedrich, H. and Wintgen, D.},
  journal={Physical Review A},
  volume={32},
  number={6},
  pages={3231--3242},
  year={1985},
  publisher={American Physical Society}
}

@article{liu2000locally,
  title={Locally resonant sonic materials},
  author={Liu, Zhengyou and Zhang, Xixiang and Mao, Yiwei and Zhu, Y. Y. and Yang, Zhiyu and Chan, C. T. and Sheng, Ping},
  journal={Science},
  volume={289},
  number={5485},
  pages={1734--1736},
  year={2000},
  publisher={American Association for the Advancement of Science}
}

@article{cummer2016controlling,
  title={Controlling sound with acoustic metamaterials},
  author={Cummer, Steven A. and Christensen, Johan and Al{\`u}, Andrea},
  journal={Nature Reviews Materials},
  volume={1},
  number={3},
  pages={16001},
  year={2016},
  publisher={Nature Publishing Group}
}

@article{ma2016acoustic,
  title={Acoustic metamaterials: From local resonances to broad horizons},
  author={Ma, Guancong and Sheng, Ping},
  journal={Science Advances},
  volume={2},
  number={2},
  pages={e1501595},
  year={2016},
  publisher={American Association for the Advancement of Science}
}

@article{bonnet1994guided,
  title={Guided waves by electromagnetic gratings and non-uniqueness examples for the diffraction problem},
  author={Bonnet-Bendhia, Anne-Sophie and Starling, Felipe},
  journal={Mathematical Methods in the Applied Sciences},
  volume={17},
  number={5},
  pages={305--338},
  year={1994},
  publisher={Wiley Online Library}
}

@article{shipman2005resonant,
  title={Resonant transmission near nonrobust periodic slab modes},
  author={Shipman, Stephen P. and Venakides, Stephanos},
  journal={Physical Review E},
  volume={71},
  number={2},
  pages={026611},
  year={2005},
  publisher={American Physical Society}
}

@article{shipman2007guided,
  title={Guided modes in periodic slabs: existence and nonexistence},
  author={Shipman, Stephen P. and Volkov, Darko},
  journal={SIAM Journal on Applied Mathematics},
  volume={67},
  number={3},
  pages={687--713},
  year={2007},
  publisher={SIAM}
}

@article{shipman2012total,
  title={Total resonant transmission and reflection by periodic structures},
  author={Shipman, Stephen P. and Tu, Hairui},
  journal={SIAM Journal on Applied Mathematics},
  volume={72},
  number={1},
  pages={216--239},
  year={2012},
  publisher={SIAM}
}

@article{lin2020mathematical,
  title={A mathematical theory for {F}ano resonance in a periodic array of narrow slits},
  author={Lin, Junshan and Shipman, Stephen P. and Zhang, Hai},
  journal={SIAM Journal on Applied Mathematics},
  volume={80},
  number={5},
  pages={2045--2070},
  year={2020},
  publisher={SIAM}
}

@article{ammari2024fano,
  title={Fano resonances in all-dielectric electromagnetic metasurfaces},
  author={Ammari, Habib and Li, Bowen and Li, Hongjie and Zou, Jun},
  journal={Multiscale Modeling \& Simulation},
  volume={22},
  number={1},
  pages={476--526},
  year={2024},
  publisher={SIAM}
}

@incollection{shipman2010resonant,
  title={Resonant scattering by open periodic waveguides},
  author={Shipman, Stephen P.},
  booktitle={Wave Propagation in Periodic Media: Analysis, Numerical Techniques and Practical Applications},
  editor={Ehrhardt, Matthias},
  series={Progress in Computational Physics (PiCP)},
  volume={1},
  pages={7--49},
  year={2010},
  publisher={Bentham Science Publishers}
}

@article{fano1961effects,
  title={Effects of configuration interaction on intensities and phase shifts},
  author={Fano, Ugo},
  journal={Physical Review},
  volume={124},
  number={6},
  pages={1866--1878},
  year={1961},
  publisher={American Physical Society}
}

@article{lin2021fano,
  title={Fano resonance in metallic grating via strongly coupled subwavelength resonators},
  author={Lin, Junshan and Zhang, Hai},
  journal={European Journal of Applied Mathematics},
  volume={32},
  number={2},
  pages={370--394},
  year={2021},
  publisher={Cambridge University Press}
}

@article{zhen2014topological,
  title={Topological nature of optical bound states in the continuum},
  author={Zhen, Bo and Hsu, Chia Wei and Lu, Ling and Stone, A. Douglas and Solja{\v{c}}i{\'c}, Marin},
  journal={Physical Review Letters},
  volume={113},
  number={25},
  pages={257401},
  year={2014},
  doi={10.1103/PhysRevLett.113.257401}
}

@article{bulgakov2017topological,
  title={Topological bound states in the continuum in arrays of dielectric spheres},
  author={Bulgakov, Evgeny N. and Maksimov, Dmitrii N.},
  journal={Physical Review Letters},
  volume={118},
  number={26},
  pages={267401},
  year={2017},
  doi={10.1103/PhysRevLett.118.267401}
}

@article{yoda2020generation,
  title={Generation and annihilation of topologically protected bound states in the continuum and circularly polarized states by symmetry breaking},
  author={Yoda, Taiki and Notomi, Masaya},
  journal={Physical Review Letters},
  volume={125},
  number={5},
  pages={053902},
  year={2020},
  doi={10.1103/PhysRevLett.125.053902}
}

@article{chen2019corner,
  title={Corner states in a second-order acoustic topological insulator as bound states in the continuum},
  author={Chen, Ze-Guo and Xu, Changqing and Al Jahdali, Rasha and Mei, Jun and Wu, Ying},
  journal={Physical Review B},
  volume={100},
  number={7},
  pages={075120},
  year={2019},
  doi={10.1103/PhysRevB.100.075120}
}

@article{fan2025acoustic,
  title={Acoustic non-Hermitian higher-order topological bound states in the continuum},
  author={Fan, Haiyan and Gao, He and Liu, Tuo and An, Shuowei and Zhu, Yifan and Zhang, Hui and Zhu, Jie and Su, Zhongqing},
  journal={Applied Physics Letters},
  volume={126},
  number={7},
  pages={071702},
  year={2025},
  doi={10.1063/5.0249792}
}

\end{document}